\documentclass[11pt]{article}

\usepackage{amsmath, amssymb, amsthm, enumerate}
\usepackage{graphicx, caption, subcaption}
\usepackage{xcolor}
\usepackage[pagebackref]{hyperref}  %

\usepackage{cleveref}

\usepackage{titlesec}
\theoremstyle{plain}
\newtheorem{thm}{Theorem}[section]
\newtheorem{prop}[thm]{Proposition}
\newtheorem{lemma}[thm]{Lemma}
\newtheorem{sublemma}[thm]{Sublemma}
\newtheorem{claim}[thm]{Claim}
\newtheorem{cor}[thm]{Corollary}

\newtheorem{conj}[thm]{Conjecture}

\theoremstyle{definition}

\newtheorem{rmkx}[thm]{Remark}
\newenvironment{rmk}
{\pushQED{\qed}\renewcommand{\qedsymbol}{$\triangle$}\rmkx}
{\popQED\endrmkx}

 \renewcommand\qedsymbol{$\blacksquare$}

\newcommand{\R}{\mathbb{R}}

\newcommand{\E}{\mathbb{E}}

\newcommand{\M}{\mathcal{M}}

\renewcommand{\H}{\mathbb{H}}
\renewcommand{\P}{\mathbb{P}}

\newcommand{\var}{\operatorname{Var}}
\newcommand{\cov}{\operatorname{Cov}}
\newcommand{\comp}{\operatorname{\# comp}}
\newcommand{\core}{\operatorname{Core}}
\newcommand{\sys}{\operatorname{systole}}

\newcounter{bencomments}

\newcounter{jenyacomments}
 
\newcounter{Qcomments}
 
\newcounter{Xcomments}
 
\newcounter{Ycomments}
 
\newcounter{Zcomments}

\title{Counting geodesics on expander surfaces}  
\author{Benjamin Dozier \thanks{Department of Mathematics, Cornell University, \href{mailto: benjamin.dozier@cornell.edu}{\nolinkurl{benjamin.dozier@cornell.edu}}} \and  Jenya Sapir \thanks{Department of Mathematics, Binghamton University, \href{mailto:sapir@math.binghamton.edu}{\nolinkurl{sapir@math.binghamton.edu}}.} }

\begin{document}
\maketitle

\begin{abstract}
  We study properties of typical closed geodesics on expander surfaces of high genus, i.e. closed hyperbolic surfaces with a uniform spectral gap of the Laplacian.  Under an additional systole lower bound assumption, we show almost every geodesic of length much greater than $\sqrt{g}\log g$ is non-simple.
  And we prove almost every closed geodesic of length much greater than $g (\log g)^2$ is filling, i.e. each component of the complement of the geodesic is a topological disc.
  Our results apply to Weil-Petersson random surfaces, random covers of a fixed surface, and Brooks-Makover random surfaces, since these models are known to have uniform spectral gap asymptotically almost surely.
  
  Our proof technique involves adapting Margulis' counting strategy to work at low length scales.  
\end{abstract}

\setcounter{tocdepth}{1}  %
\tableofcontents
\setcounter{tocdepth}{3} %

\section{Introduction}
\label{sec:intro}

Let $X$ be a hyperbolic surface (that is connected, closed, and orientable; we will assume these three properties throughout).  It is easy to see that the shortest closed geodesic on $X$ is always \emph{simple}, i.e. does not self-intersect.  The number of simple geodesics less than a given length grows polynomially \cite{rivin2001, mirzakhani2008}, while the total number of closed geodesics grows exponentially (by work of Delsarte, Huber, and Selberg; see \cite{buser2010} for references).  Thus non-simple closed geodesics must eventually become predominant.  At what length scale does the transition occur?

We refer to this as the ``birthday problem'' for geodesics, by analogy with the basic probability question about the number of uniform, independent samples with replacement from a collection of $n$ objects needed before some object is picked multiple times.  The answer will depend on particular geometric features of the surface.  In this paper, we address this question for expander surfaces.  We also study the question of the length scale at which almost all closed geodesics are \emph{filling}, i.e. each component of the complement of the geodesic %
is a topological disc.  

The Laplace operator on the hyperbolic surface $X$ has a discrete spectrum and always has a simple eigenvalue of $0$.  The \emph{spectral gap} is the distance to the next smallest eigenvalue.  For $\delta>0$ we say that $X$ is a $\delta$-\emph{expander surface} if its spectral gap is greater than $\delta$.  This terminology is motivated by an analogous and much studied concept for graphs.  Families of $\delta$-expander surfaces exhibit many interesting properties such as fast mixing of geodesic flow and lower bound on Cheeger constant.  Random constructions typically give expander families. 

We denote by $N(X,L)$, respectively $N_{simp}(X,L)$, the number of oriented closed geodesics, respectively oriented simple closed geodesics, on $X$ of length at most $L$.  The \emph{systole} of a hyperbolic surface is the length of the shortest closed geodesic.

\begin{thm}
  \label{thm:simplicity}
  Let $\delta, s_0, \epsilon>0$.  There exists a constant $c=c(\delta,s_0,\epsilon)$ such that for any $\delta$-expander surface $X$ of genus $g$ with systole at least $s_0$, and any $L> c\sqrt{g} \log g$,
  \begin{align*}
    N_{simp}(X,L) \le  \epsilon \cdot N(X,L). 
  \end{align*}
\end{thm}

\begin{conj}
  In the above, one can replace the condition $L > c\sqrt{g} \log g$ by $L>c \sqrt{g}$.  
\end{conj}

Wu and Xue have recently proved the analogous conjecture for the specific case of Weil-Petersson random surfaces \cite[Theorem 4, part (2)]{wx2022}.

It is also conceivable that the theorem (and conjecture) hold for \emph{all} surfaces $X$ (with no assumption about spectral gap or systole).  

\paragraph{Asymptotic notation.} In this paper, we use Hardy's notation $A \prec B$ to mean $A=o(B)$ as the independent variable (typically the genus $g$) goes to  $\infty$, i.e. $A/B \to 0$.    

We also use the standard notation $f(x)=O(g(x))$, which means there exists some positive constant $C$, independent of $x$ (but potentially depending on some other parameters, which will be specified), such that $|f(x)| \le C |g(x)|$ for all $x$; the functions are allowed to take negative values.  And we define $f(x)=\Omega(g(x))$ to mean $g(x)=O(f(x))$ \emph{and} $f(x)\ge 0$ for all $x$. %

\begin{rmk}
  \label{rmk:simple-nonopt}
  For the regime $L\prec \sqrt{g}$, %
  whether $N_{simp}(X,L)$ is dominant depends on more aspects of the geometry of the surface, beyond spectral gap and systole lower bound.

  On the one hand, in this regime Weil-Petersson random surfaces will have $N_{simp}(X,L)>(1-\epsilon) N(X,L)$ asymptotically almost surely \cite[Theorem 4 (1)]{wx2022}.

  On the other hand, surfaces obtained by gluing fixed hyperbolic pairs of pants (say with all cuffs of length $2$) according to a random regular graph, with any twists, asymptotically almost surely form an expander family with lower bound on systole.  This follows by combining (i) a comparison of the Cheeger constant for the surface to that of the graph \cite[Section 4.1]{buser1978}, and (ii) the well-known lower bound on Cheeger constant for random regular graphs.
  For these surfaces, we anticipate that  $N_{simp}(X,L)\prec N(X,L)$ whenever $L\to \infty$ with $g$ (so including many cases in which $L\prec \sqrt{g}$), since every time a geodesic enters a pair of pants it has a definite chance of picking up a self-intersection before leaving.  
\end{rmk}

Let $N_{fill}(X,L)$ denote the number of filling closed geodesics on $X$ of length at most $L$.  

\begin{thm}
  \label{thm:filling}
  Let $\delta, s_0, \epsilon>0$.  There exists a constant $c=c(\delta,s_0,\epsilon)$ such that for any $\delta$-expander surface $X$ of genus $g$ with systole at least $s_0$ and any $L>c \cdot g  (\log g)^2$,
  \begin{align*}
    N_{fill}(X,L) \ge (1-\epsilon) N(X,L).  
  \end{align*}
\end{thm}

\begin{rmk}
  It is conceivable that the $L>c \cdot g (\log g)^2$ condition can be weakened to $L>c\cdot g \log g$, though some new methods would be necessary.  Our technique relies on sampling the geodesic at times that are at least $c \log g$ apart, in order to ensure independence.  But we believe one should be able to argue with less independence.    

  We do not anticipate that the bound can be made smaller than $c\cdot g \log g$.  We now sketch a reason for this.  Consider the surfaces glued from fixed size pants described \Cref{rmk:simple-nonopt}.  Any filling closed geodesic must intersect every pair of pants, and for the decomposition into fixed size pants, we anticipate that this event is governed by the classical ``coupon collector problem.''  This is the problem of determining how many independent, uniform draws (with replacement) from a collection of $n$ different objects are needed before it is highly likely that every object has been drawn at least once.  The transition from low to high probability occurs around $c \cdot  n\log n$ draws.  This also matches the solution to the analogous ``cover time'' problem for random regular graphs \cite{bk1989, cf2005}.

  However, we anticipate surfaces sampled from the three random models we discuss below to behave differently.  In particular we do not anticipate that such surfaces have a decomposition into pants of bounded size.  For these models, it is conceivable that the result above might hold for $L>c\cdot g$ (as suggested in \cite[Question p.5]{wx2022} for the Weil-Petersson model).  
\end{rmk}

\subsection{Applications to random surfaces}
\label{sec:rand}

We now give applications of \Cref{thm:simplicity} to several different models of random surfaces.  
 There is also an analogous story for random regular graphs \cite{ds2022}.  

\subsubsection{Weil-Petersson random surfaces}
\label{sec:wp}

Our original inspiration for this project was \cite[Conjecture 2]{lw2021}, which concerns the birthday problem for Weil-Petersson random surfaces.  While we were writing up our results, this conjecture was resolved in a very precise manner in \cite{wx2022}.  Our methods give a very different proof of part of that result.  We require a length lower bound that is larger than the optimal one by a factor of $\log g$.  On the other hand, our techniques allow us to study other random models as well, described below.  

Let $\P^{WP}_g[\cdot]$ denote the probability of some event with respect to surfaces drawn from the Weil-Petersson measure on $\M_g$, the moduli space of genus $g$ hyperbolic surfaces.

\begin{cor}[Weil-Petersson surfaces]
  \label{cor:wp-simplicity}
  Fix $\epsilon>0$, and let $L$ be some function of genus $g$.  
  \begin{enumerate}[(i)]
  \item \label{item:wp-simp} (Weaker version of \cite{wx2022}, Theorem 4) If $L\succ \sqrt g \log g$, then 
  \begin{align*}
    \lim_{g\to\infty} \P^{WP}_g[N_{simp}(X,L) < \epsilon \cdot N(X,L)] = 1. 
  \end{align*}
  
\item \label{item:wp-filling} If $L \succ g\cdot (\log g)^2$, then
  \begin{align*}
    \lim_{g\to\infty} \P^{WP}_g[N_{fill}(X,L) \ge (1-\epsilon) \cdot N(X,L)] = 1.
  \end{align*}
\end{enumerate}
\end{cor}

\begin{proof}
  Fix $\eta>0$.  By  \cite[Theorem 4.2]{mirzakhani2013}, we can find $s_0>0$ such that
  \begin{align}
    \lim_{g\to\infty} \P_g^{WP}[\sys(X)>s_0] >1-\eta. \label{eq:sys}
  \end{align}

  Also by \cite[Theorem 4.8]{mirzakhani2013}, there exists a $\delta>0$ such that
  \begin{align}
    \lim_{g\to\infty} \P_g^{WP}[X \text{ is } \delta\text{-expander}] =1. \label{eq:prob-wp}
  \end{align}

  Now for this $\delta, s_0, \epsilon$, we apply \Cref{thm:simplicity}.  For the constant $c$ from this theorem, we have, by \eqref{eq:sys} and \eqref{eq:prob-wp}, for all $g$ sufficiently large, and any $L'>c \sqrt{g} \log g$:  
  \begin{align*}
    \P^{WP}_g[N_{simp}(X,L') < \epsilon \cdot N(X,L')] >1-\eta.
  \end{align*}
  In particular, for our $L\succ \sqrt g \log g$ we have
  \begin{align*}
    \lim_{g\to\infty} \P^{WP}_g[N_{simp}(X,L) < \epsilon \cdot N(X,L)] >1-\eta.
  \end{align*}
  Since this holds for any $\eta>0$, we get \eqref{item:wp-simp}.

  The proof of \eqref{item:wp-filling} follows the same pattern, using \Cref{thm:filling}.  
\end{proof}

\subsubsection{Random covers}
\label{sec:covers}

Let $Y$ be a fixed closed hyperbolic surface.  The \emph{random cover model} of random hyperbolic surfaces gives a finitely-supported probability measure on $\M_g$ for each $g$ such that the Euler characteristic $2-2g$ is a a multiple of $\chi(Y)$; it is simply counting measure on the set of all genus $g$ Riemannian covers of $Y$.  

Let $\P^Y_g[\cdot]$ denote the probability of some event with respect to surfaces in $\mathcal M_g$ drawn from this random cover measure. 

\begin{cor}[Random covers]
  \label{cor:covers-simplicity}
  Fix $\epsilon>0$, and let $L$ be a function of genus $g$. 
  \begin{enumerate}[(i)]
  \item  If $L\succ \sqrt g \log g$, then 
  \begin{align*}
    \lim_{g\to\infty} \P^Y_g[N_{simp}(X,L) < \epsilon \cdot N(X,L)] = 1.
  \end{align*}

\item  If $L\succ g \cdot (\log g)^2$, then 
  \begin{align*}
    \lim_{g\to\infty} \P^Y_g[N_{fill}(X,L) \ge (1- \epsilon) \cdot N(X,L)] = 1.
  \end{align*}
    \end{enumerate}
\end{cor}

\begin{proof}
  The structure of the proof is the same as proof of \Cref{cor:wp-simplicity}.  

Control of the systole for random covers is easy.   For \emph{any} cover $X$ of $Y$, we have $\sys(X)\ge \sys(Y)$, since any closed geodesic $\gamma$ on $X$ projects to a closed geodesic on $Y$ with length at most $\ell_X(\gamma)$.

  To control spectral gap, we appeal to \cite[Theorem 1.5]{mnp2022}, which gives that there exists $\delta>0$ (depending on $Y$) such that 
  \begin{align}
    \lim_{g\to\infty} \P_g^{Y}[X \text{ is } \delta\text{-expander}] =1. \label{eq:prob-cover}
    \end{align}
    (The $\delta$ can be taken to be any real less than $\min\{\lambda_1(Y),3/16\}$.)
    
    The rest of the proof is identical to  proof of \Cref{cor:wp-simplicity}, using \Cref{thm:simplicity} and \Cref{thm:filling}.  
  \end{proof}

  \subsubsection{Brooks-Makover (Belyi) random surfaces}
  \label{sec:bm}

  Yet another model of random hyperbolic surfaces was introduced in \cite{bm2004}. 

Gluing together ideal hyperbolic triangles (``midpoint to midpoint'') according to a trivalent ribbon graph yields a cusped hyperbolic surface.  Such a surface can be compactified by considering the corresponding punctured Riemann surface, filling in the puncture, and then taking the uniformizing hyperbolic metric in the conformal class of this closed Riemann surface.

Fix an integer $2n$ and choose the trivalent ribbon graph uniformly at random from the (finite) collection of such on $2n$ vertices.  The resulting closed surface is a \emph{Brooks-Makover random surface}, and we get a finitely-supported probability measure on the set of hyperbolic surfaces.  The genus of the surfaces in the support is not determined by $n$ (though much is known about the distribution of genus; see \cite[Corollary 5.1]{gamburd2006}). We denote by $\P_n^{BM}[\cdot]$ the probability of some event with respect to surfaces drawn from this measure.

\begin{cor}[Brooks-Makover surfaces]
  \label{cor:bm-simplicity}
  Fix $\epsilon>0$, and let $L$ be a function of $n$ (half the number of triangles).
  \begin{enumerate}[(i)]
  \item \label{item:bm-simp} If  $L\succ \sqrt n \log n$, then 
  \begin{align*}
    \lim_{g\to\infty} \P^{BM}_n[N_{simp}(X,L) < \epsilon \cdot N(X,L)] = 1.
  \end{align*}

    \item \label{item:bm-filling} If  $L\succ  n (\log n)^2$, then 
  \begin{align*}
    \lim_{g\to\infty} \P^{BM}_n[N_{fill}(X,L) \ge (1- \epsilon) \cdot N(X,L)] = 1.
  \end{align*}

\end{enumerate}
\end{cor}

\begin{proof}
  Although the genus of the surfaces in the support of $\P^{BM}_n$ is not deterministic, it will be enough for our purposes to use a simple linear upper bound:
  $$n-1 \ge 2g-2 .$$
  This is easily proved via the Euler characteristic formula with $e=3f/2 = 3n$ and $v\ge 1$, where $v,e,f$ are the number of vertices, edges, faces, respectively, of the triangulation.

  Combining this inequality with our assumption that $L \succ \sqrt{n} \log n$, we then get that our function $L$ satisfies
  $$L\succ \sqrt{g} \log g,$$
  for $g$ the genus of any surface in the support of $\P_n^{BM}$.

  By \cite[Theorem 2.2 (a), (c)]{bm2004}, there exist constants $\delta>0$ and $s_0>0$ such that
  \begin{align*}
     \lim_{n\to \infty} \P_n^{BM}[\sys(X)>s_0] =1,  \\ 
    \lim_{n\to\infty} \P_n^{BM}[X \text{ is } \delta\text{-expander}] =1.
  \end{align*}

  Item \eqref{item:bm-simp} then follows by applying \Cref{thm:simplicity}, as for the previous two random models. Item \eqref{item:bm-filling} is proved similarly, using \Cref{thm:filling}.  

\end{proof}

\subsection{Relation to prior work}   The issue of the relative frequency of simple geodesics compared to all geodesics arises when studying the spectral gap, in particular for random surfaces (see \cite{lw2021, wx2021, am2023}).  More broadly, this paper fits into a line of work on the ``shape of a random hyperbolic surface of high genus,'' pioneered by Brooks and Makover \cite{bm2004} for surfaces glued from triangles, and by Mirzakhani \cite{mirzakhani2013} for the Weil-Petersson model.  For behavior of geodesics in this context, see for example \cite{gpy2011,mp2019,mt2022, nwx2023}.   A recent major triumph in this area is the use of a random construction to prove the existence of family of closed hyperbolic surfaces of growing genus and spectral gap approaching $1/4$ \cite{mw2021} \footnote{Since the first version of our paper was released, there has been a great deal of further  progress on spectral gaps of random surfaces, including \cite{am25, mpv25, sw25, hmt25}.}.  
Our main theorem is not in the random setting, but involves conditions that common models of random surfaces satisfy, so our results apply to these, as discussed above. 

\subsection{Discussion and outline of proof}
\label{sec:plan}

The key to our proofs is transferring probabilistic arguments for the ``birthday'' and ``coupon-collector'' problems into the hyperbolic geometry setting using techniques of Margulis for counting closed geodesics.  Our methods are very flexible and should be applicable to other counting problems.  We develop a toolbox for translating results that hold for walks on regular graphs to the surface context.

A crucial ingredient in Margulis' approach is \emph{mixing} of the geodesic flow; in our setting we need \emph{effective mixing}, which follows from the spectral gap assumption.  We also show that effective mixing in fact implies effective \emph{multiple} mixing, using the expansion/contraction properties of hyperbolic geodesic flow.  Multiple mixing can be thought of as a notion of independence (it corresponds to the Markovian property of random walks on graphs).  For these notions, it is very important that we have uniform control over all the constants involved, in a genus independent way (in particular, we need more than just the exponential rate of decay of error terms).  

A significant difference between the graph and surface contexts is that a geodesic returning close to where it has been before is not enough to guarantee a self-intersection (there are arbitrarily long simple closed geodesics on a fixed surface; these must come back very close to previously visited places, but the different strands near such a place are nearly parallel).  So instead we work with a more restrictive property, namely that the geodesic comes back near where it has been and at definite angle bounded away from zero.  This does guarantee a self-intersection.

There are various technical complications that arise because we must discretize our surface in order to leverage the analogy with graphs.  Furthermore, we must do this discretization in a ``uniform'' way across different surfaces with genus going to infinity.

\paragraph{Outline of proof.}
\begin{itemize}
\item In \Cref{sec:effect-mix}, we develop some tools for counting closed geodesics using flow boxes.  We prove effective mixing for flow boxes of the geodesic flow on expander surfaces, using a theorem of Ratner.  We also state and prove an effective Anosov Closing Lemma, and various other results that will be useful for counting.  

\item In \Cref{sec:effect-prime-geod}, we prove an effective prime geodesic theorem, \Cref{thm:effect-pgt}, for expander surfaces.  We follow the strategy of Margulis, using the effective mixing result developed in the previous section.

\item In \Cref{sec:effect-mult-mix}, we prove effective \textit{multiple} mixing
using geometric properties of flow box intersections.  

\item In \Cref{sec:simp-conj}, we prove the required upper bound on the number of simple geodesics, \Cref{prop:num-simple}, and then combine this with our effective prime geodesic theorem to prove \Cref{thm:simplicity}.

\begin{itemize}
 \item In \Cref{sec:prob-problem}, we demonstrate the ideas in the proof of \Cref{prop:num-simple} by first proving an analogous discrete probability result.
 \item In Sections \ref{sec:setup} - \ref{sec:simple-geod-B}, we work towards bounding the number of simple closed geodesics of length roughly $t$ that pass through some flow box $B$ (\Cref{lem:no-inter-return}).
 \item In \Cref{sec:setup}, we find a collection of $2g-2$ pairs of flow boxes with the property that if a geodesic passes through both flow boxes in a pair, it is forced to self-intersect. 
 \item In \Cref{sec:bound-S}, we control the set of directions that do not pass through any pair of these flow boxes. We do this by breaking up such directions further into sets $R$ that avoid too many of our flow boxes, and $Q_k$ that often pass through one flow box of a pair, but not both. We control these separately in \Cref{sec:R} and \Cref{sec:Q}.
  \item In \Cref{sec:simple-geod-B} we use the above to prove \Cref{prop:eps-error}, a bound on the number of times that simple closed curves of length roughly $t$ pass through a given flow box $B$. %
    
  \item In \Cref{sec:proof}, we average the previous count over all possible flow boxes $B$ to bound the number of simple closed geodesics of length at most $L$ in \Cref{prop:num-simple}.  
\end{itemize}

\item In \Cref{sec:filling}, we prove \Cref{thm:filling} on filling geodesics.  We first construct a controlled set of flow boxes with the property that any closed geodesic that intersects all of them must be filling.  We then show that most sufficiently long closed geodesics intersect all of these flow boxes.  
\end{itemize}

\subsection{Acknowledgements}
\label{sec:acknowledgements}
We thank Mike Lipnowski, Katie Mann, Bram Petri, and Alex Wright for helpful conversations and comments.   And we thank Michael Magee for raising the question to us of counting closed geodesics on random cover surfaces.  We also thank the anonymous referees for many helpful suggestions that have improved the exposition.

\section{General results on flow boxes and counting}
\label{sec:effect-mix}
In this section, we work with \textit{flow boxes} $B$, which are small sets of almost-parallel vectors in $T^1X$ (defined in \Cref{sec:notation-setup}).  We relate dynamical properties of these to counting closed geodesics. 

Components of $B\cap g_{-t}B$ correspond to homotopy classes of (approximately) closed geodesics that start and end in $B$, of (approximate) length $t$.   In \Cref{sec:effect-mixing-flow}, we state an effective mixing result (coming from a theorem of Ratner) that allows us to estimate the measure of this set.  In \Cref{sec:inter-flow-box}, we then prove some properties about the geometry of components. Next in \Cref{sec:count-box-surf} and \Cref{sec:bound_given_type}, we show how to turn ``local" estimates (i.e. on geodesics that start and end in a particular box $B$) into global ones. Lastly in \Cref{sec:Anosov_closing}, we state two versions of the classical Anosov Closing Lemma; roughly these say that each component $B\cap g_{-t}B$ corresponds to exactly one closed geodesic.

\subsection{Notation and setup}
\label{sec:notation-setup}

We let $T^1X$ be the unit tangent bundle to $X$.  There is a natural measure $\mu$ on $T^1X$, the Liouville (or Haar) measure.  We normalize $\mu$ to be a probability measure, i.e. $\mu(T^1X)=1$.

\paragraph{Matrices for geodesic and horocylic flows.}

Let 
\begin{align*}
  &g_t = \left(   \begin{array}{cc}
                   e^{-t/2} & 0\\
                   0 & e^{t/2} 
                 \end{array} \right),
                       \quad 
                       r_\theta = \left(   \begin{array}{cc}
                          \cos \theta  & -\sin\theta\\
                           \sin\theta & \cos\theta
          \end{array} \right),  \\             
  & h^s_r = \left(   \begin{array}{cc}
                     1 & r\\
                      0 & 1
    \end{array} \right),
          \quad 
          h^u_r = \left(   \begin{array}{cc}
                           1 & 0\\
                           r & 1
                           \end{array} \right).
\end{align*}
We identify $PSL_2(\R)$ with $T^1\H$ via the map that takes a matrix $A$ to the image $Av_0$, where $v_0$ is the upwards pointing unit tangent vector at $i\in \H$, under (the derivative of) the M\"obius action.  Under this identification, $g_t, h^s_r, h^u_r$ generate the geodesic, stable (contracting) horocycle, and unstable (expanding) horocycle flows, respectively, via multiplication by the inverse on the \emph{right}, e.g. $g_t(Av_0) = Ag^{-1}_tv_0$.  These flows preserve the measure $\mu$.  

\paragraph{Flow boxes $B(v)$.}

Given three parameters $\vec \eta:=(\eta_1,\eta_2,\eta_3)$ and $v\in T^1X$, we define  the \emph{flow box} 
\begin{align*}
    B(v):=\{ h_{r_1}^ug_t h_{r_2}^sv : |r_1|<\eta_1/2, |t|<\eta_2/2 ,|r_2|<\eta_3/2\}.  
\end{align*}
That is, we build a flow box by first flowing $v$ in the contracting direction, then in the geodesic flow direction, and lastly in the expanding direction. %

We will refer to this as the $\vec \eta$ flow box centered at $v$.  By $\eta$ flow box, we will mean an $(\eta, \eta, \eta)$ flow box.

We say a $(\eta_1,\eta_2,\eta_3)$ flow box is \emph{embedded} if the map
$(r_1,t,r_2) \mapsto h_{r_1}^ug_t h_{r_2}^sv$ is an injection on the domain $|r_1|<\eta_1/2,|t|<\eta_2/2,|r_2|<\eta_3/2$.

For $\vec \eta$ small, the coordinates $r_1,t,r_2$ on a flow box behave almost exactly like standard coordinates on a Euclidean rectangular box.  

For technical purposes, given $B=B(v)$ an $\vec \eta$ flow box, and $\kappa>0$, we also define $B^{+\kappa}$ (respectively, $B^{-\kappa}$) to be the $(1+\kappa)\vec \eta$ (respectively, $(1-\kappa)\vec\eta$) flow box centered at $v$.  In many cases we will not need very precise estimates, and it will be convenient to work with flow boxes scaled by factors of $3$ or $1/3$.  To ease notation, we write $B^+:=B^{+(2)}$ and $B^-:=B^{-(2/3)}$  for these flow boxes.  We will also use this notation iteratively -- for instance $B^{++}$ is the $9\vec \eta$ flow box centered at $v$.

\paragraph{Foliation by geodesics.}
Because flow boxes are not exactly products, it is not true that each point lies on a segment of length $\eta$ in the geodesic flow or contracting directions. However, this only becomes a problem near the edges, and for most applications, it is enough to view flow boxes as products.

First, we consider geodesic segments. Let $B$ be an $\eta$ flow box. Each $v \in B$ lies on some geodesic segment inside $B$. Most of these segments have length $\eta$. We quantify this statement here:
\begin{lemma}
\label{lem:full-geod-seg}
    Let $X$ be a hyperbolic surface, and let $\eta >0$ be small enough so that any $\eta$ flow box is embedded in $X$. Let $B$ be an $\eta$ flow box. Let $B^*$ be the flow box with the same center as $B$, of width $e^{-\eta}\eta$ %
    in the expanding direction, and width $\eta$ in the geodesic flow and contracting directions. Then for all $v \in B^*$, $v$ lies on a length $\eta$ geodesic segment in $B$.
\end{lemma}
\begin{proof}
    To build $B$, we start with a vector $v_0$, and flow it first in the contracting horocycle direction, then the geodesic flow direction. This results in a ``rectangle" foliated by length $\eta$ geodesic segments. Let $\alpha$ be such a segment. In particular, for the initial vector $v$ on $\alpha$,    
    \[
    \alpha = \{g_t v \ | \ t \in [0,\eta]\}
    \]    
    For each $t \in [0,\eta]$, we flow $g_t v$ by $\eta/2$ in both directions in the expanding horocycle direction. 

    Take a vector $w = h^u_r(v)$ on the expanding horocycle leaf of the initial endpoint $v$ of $\alpha$. We wish to find a condition for which $w$ is on a length $\eta$ geodesic leaf of $B$. Flowing $w$ for time $t \leq \eta$ we see
    \begin{align*}
     g_t(w) & = g_t(h^u_r(v))\\
      & = h^u_{e^t r}(g_t v) \\
      & = h^u_{e^t r} (w_t)
    \end{align*}
   where $w_t = g_t v$ is another vector on $\alpha$. Now, $h^u_{e^t r} w_t$ is in $B$ only if $e^t r < \eta/2$. %
   This condition is most restrictive when $t = \eta$. In that case, we see 
   \[
   r \leq e^{-\eta}\cdot \eta/2
   \]
    
    In other words,  $w$ is on a length $\eta$ geodesic leaf in $B$ if and only if that leaf intersects $B^*$, the flow box about $v_0$ with expanding width $e^{-\eta} \eta$ and contracting and geodesic flow width $\eta$. 
\end{proof}

We also quantify when points lie on ``large enough" contracting horocycle segments.
\begin{lemma}
\label{lem:large-unstab-seg}

  Let $\eta, \kappa > 0$,  with $\kappa<1$. Let $X$ be a hyperbolic surface so that any $\eta$ flow box in $T^1X$ is embedded.  Then there is some $c_0 = c_0(\eta, \kappa)>0$ so that the following holds. Let $v\in T^1X$ and let $B(v)$ be the $\eta$-flow box centered at $v$.  Then each $w \in B^{-\kappa}(v)$ lies on a contracting  horocycle segment of width at least $c_0\eta$ inside $B(v)$. 
\end{lemma}
\begin{proof}
    This follows from the fact that each flow box is open, so each point lies on a contracting horocycle segment of non-zero length. Moreover, the closure of a flow box is compact, and the function sending each point to the length of its contracting horocycle segment is continuous on the closure. See discussion in \cite[Section 8.7]{fh2019} for a more precise statement.  
\end{proof}

\paragraph{Width of flow boxes.} 
Let $S\subset T^1X$, and $B$ an embedded flow box.  We say $S$ has \emph{full width in the contracting direction} (relative to $B$) if for all $v\in S$, the connected component of $\{h_r^sv : r \in \R\} \cap B$ that contains $v$ is a subset of $S$.  Similarly, we define analogous notions for the expanding (respectively flow) directions by replacing $h^s$ with $h^u$ (respectively $g_t)$ 

We say that $S$ has \emph{width $\xi$ in the expanding direction} if every  $v\in S$ is on a length $\xi$ expanding horocycle segment in $S$.%

To define subsets of non-full width in the contracting and geodesic flow directions, we must take into account the fact that flow boxes are not perfectly foliated by contracting horocycle and geodesic flow segments.

We say $S$ has width $\xi < \eta$ in the geodesic flow direction if every point in $S \cap B_\xi$ lies on a length $\xi$ geodesic segment in $S$, where $B_\xi$ is the subset of $B$ foliated by geodesic segments of length at least $\xi$. We define the width of $S$ in the contracting horocycle direction analogously.

By \Cref{lem:full-geod-seg} and \Cref{lem:large-unstab-seg}, for any flow box $B$, there is a small $\kappa > 0$ so that all points inside a slightly smaller flow box $B^{-\kappa}$ lie on ``large enough" geodesic flow and expanding horocycle segments. So this definition means that all but a small measure of the points in $S$ lie on geodesic flow or expanding horocycle segments of length $\xi$.

We note that if a subset $S \subset B$ has width $c_1 \eta \times c_2 \eta \times c_3 \eta$ in the expanding, geodesic flow, and contracting directions, then $\mu(S) = c_1 c_2 c_3 \mu(B) (1+O(\kappa))$.

Now suppose $S$ has finitely many connected components. For each connected component $C_i$ of $S$, let $\xi_i\ge 0$ be the maximal number such that $C_i$ has width $\xi_i$ in the flow direction.  Then the \emph{average width of $S$ in the flow direction} is the average of the $\xi_i$ over all the connected components.  

\subsection{Effective mixing for flow boxes}
\label{sec:effect-mixing-flow}
The following is the central tool that we use in the paper -- it is the place where the expander assumption is used.   

  \begin{lemma}[Effective mixing for flow boxes]
    \label{lem:effect-mix} 
    There exists some function $f(\vec \eta,\epsilon)$ such that for any $\delta>0$, there exists $\kappa=\kappa(\delta)>0$ with the following property.  Let $X$ be a $\delta$-expander surface, $v,w\in T^1X$, and $\vec \eta \in \R_{>0}^3$ such that the $3\vec \eta$ flow boxes $B^+(v)$ and $B^+(w)$ are embedded.  Then for any $\epsilon>0$, 
    \begin{align}
      \mu\left(g_{-t}B(v) \cap B(w)\right) = \mu(B(v))^2 \cdot \left(1 +O(\epsilon)+ \frac{1}{\mu(B(v))}O(f(\vec \eta,\epsilon) e^{-\kappa t})\right) \label{eq:mixing}
    \end{align}
    for all $t\ge 0$, where the implicit constants in $O()$ are absolute. 
    The function $f$ can be taken to be continuous (on its natural domain, where the inputs are positive).  
  \end{lemma}
   
  \begin{proof}
    Note that $\mu(B(v))=\mu(B(w))$ for any $v,w$ unit tangent vectors on surfaces of the same genus (assuming the flow boxes are embedded).  Let $\chi^{\epsilon}_{B(v)}$ be a smooth ($C^\infty$) approximation to the indicator function $\chi_{B(v)}$. We choose these approximating functions \emph{uniformly} over the possible choices of $X$ and $v$, i.e. the restriction of the function to a small neighborhood of $B(v)$ looks the same over all such $X,v$.  Specifically, we take, for each $\vec \eta$ and $\epsilon$, an $\vec \eta$ flow box $B$ in $T^1\H$, and then define  $\chi^{\epsilon}_{B}:T^1\H \to \R$ such that 
    
    \begin{enumerate}[(i)]
        \item $0\le \chi^{\epsilon}_{B} \le 1$,  \label{item:smooth1}
        \item $\mu(\operatorname{support}(\chi^{\epsilon}_{B})) \le (1+\epsilon) \mu(B)$. \label{item:smooth2}
        \item $\operatorname{support}(\chi^{\epsilon}_{B}) \subset B^+. $
    \end{enumerate}

    Then for any $X,v$ such that the relevant boxes are embedded, note that there is an isometry between a small disc in $X$ and a small disc in $\H$, such that the induced action on unit tangent bundles takes $v$ to the center of $B$.  We then define $\chi^{\epsilon}_{B(v)}:T^1X \to \R$ on $B^+(v)$ by pulling back $\chi^{\epsilon}_{B}$ along this map; on the complement of $B^+(v)$, we take the value of $\chi^{\epsilon}_{B(v)}$ to be $0$.    

    Then let
    $$h_v := \chi^{\epsilon}_{B(v)}- \int_X \chi^{\epsilon}_{B(v)}d\mu.$$
    Note $h_v$ has mean $0$ and is smooth.  

    Now we apply \cite[Theorem 2]{matheus2013} (which is in terms of the spectrum of the Casimir operator, but, as remarked on p. 473 of that paper, the bottom part of the spectrum of Casimir and Laplace operators coincide).  This is an explicit version of \cite{ratner1987}, and gives that there exists an absolute constant $c$, and a $\kappa>0$ depending on $\delta$, such that for any $t\ge 0$,
    \begin{align}
      \langle g_{-t}h_v, h_w\rangle \le & c \bigg[ \|L_W^3h_v \| \left(\|h_w\|+ \|L_W^3h_w\| \right) +  \|L_W^3h_w\| \left(\|h_v\|+ \|L_W^3h_v\| \right)  \\
      &+  \left(\|h_v\| + \|L_W^3h_v\|\left) \right(\|h_w\| + \|L_W^3h_w\|\right)  \bigg] e^{-\kappa t},  \label{eq:matheus} 
    \end{align}
where $\|\cdot \|$ is the $(L^2,\mu)$ norm, and $L_W$ denotes the Lie derivative in the $W$ direction, where
\begin{align*}
  W=
  \left( \begin{array}{cc}
    0 & 1\\ -1 & 0
  \end{array} \right)  \in \mathfrak{sl}_2(\R).  
\end{align*}
(The result in \cite{matheus2013} gives $\kappa$ and $c$ explicitly, but we do not need that level of precision here.  The bound there has a $t e^{-\lambda t}$ term instead of our $e^{-\kappa t}$; we have absorbed the $t$ in the exponential term, at the cost of making the constant $\kappa$ worse. For a similar reason, we don't need to restrict to $t\ge 1$ as \cite{matheus2013} does.)

Now
\begin{align}
  \| L_W^3 h_v \|^2 & =  \|L_W^3 \chi^{\epsilon}_{B(v)}\|^2 \\
                    & \le \sup \left| L_W^3 \chi^{\epsilon}_{B(v)}\right|^2 \cdot \mu\left(\operatorname{support}(\chi^{\epsilon}_{B(v)})\right) \\
                    &\le \tilde f(\vec \eta,\epsilon) \cdot (1+\epsilon) \mu(B(v)),  \label{eq:lie-bound} 
\end{align}
where we take $\tilde f:=\sup \left| L_W^3 \chi^{\epsilon}_{B(v)}\right|^2 $.  This does not depend on $X$ or $v$, because of our uniform definition of the $\chi^{\epsilon}_{B(v)}$ and the fact that $L_W$ is a local differential operator.  

Now, using that $\mu$ is a probability measure, we get
\begin{align}
  \|h_v\|^2 &= \int\left(\chi^{\epsilon}_{B(v)}\right)^2d\mu - \left(\int \chi^{\epsilon}_{B(v)}d\mu\right)^2 
  \le \int\left(\chi^{\epsilon}_{B(v)}\right)^2d\mu
  \\ &\le (1+\epsilon) \mu(B(v)), \label{eq:h-bound}
\end{align}
and we also get the same bounds for $h_w$.  
Using \eqref{eq:h-bound} and \eqref{eq:lie-bound} in \eqref{eq:matheus} gives
\begin{align*}
    \langle g_{-t}h_v, h_w\rangle \le  f(\vec \eta,\epsilon) \cdot \mu(B(v))\cdot e^{-\kappa t}
\end{align*}
for some $f$ depending only on $\vec \eta,\epsilon$.  
Then, using that $h_v,h_w, g_{-t}h_v$ have mean $0$, we get 
    \begin{align}
      \langle g_{-t} \chi^{\epsilon}_{B(v)}, \chi^{\epsilon}_{B(w)} \rangle &= \left\langle g_{-t}\left(h_v + \int \chi^{\epsilon}_{B(v)}d\mu\right), h_w + \int \chi^{\epsilon}_{B(w)}d\mu \right\rangle \\
        &= \left\langle g_{-t}h_v + \int \chi^{\epsilon}_{B(v)} d\mu , h_w + \int \chi^{\epsilon}_{B(w)}d\mu \right\rangle \\
        &= \int \chi^{\epsilon}_{B(v)} d\mu \int \chi^{\epsilon}_{B(w)} d\mu  + \langle g_{-t}h_v, h_w\rangle \\
        & =  \int \chi^{\epsilon}_{B(v)} d\mu \int \chi^{\epsilon}_{B(w)} d\mu + O\left( f(\vec \eta,\epsilon) \cdot \mu(B(v))\cdot e^{-\kappa t}\right), \label{eq:smooth} 
    \end{align}
    where the implied constant in $O(\cdot)$ is absolute.

    Now to get the upper bound part of \eqref{eq:mixing}, we note that we can pick $\chi^{\epsilon}_{B}$ with the additional properties that for all $X,v$, we have  $\chi_{B(v)}\le \chi^{\epsilon}_{B(v)}$ and $\int \chi^{\epsilon}_{B(v)} d\mu \le (1+\epsilon) \mu(B(v))$.   %
    Then \eqref{eq:smooth} gives
    \begin{align*}
        \mu(g_{-t}B(v) \cap B(w)) &= \langle g_{-t} \chi_{B(v)}, \chi_{B(w)} \rangle \\
        & \le \langle g_{-t} \chi^{\epsilon}_{B(v)}, \chi^{\epsilon}_{B(w)} \rangle \\
        & = \int \chi^{\epsilon}_{B(v)} d\mu \int \chi^{\epsilon}_{B(w)} d\mu + O\left( f(\vec \eta,\epsilon) \cdot \mu(B(v))\cdot e^{-\kappa t}\right)\\
        & \le (1+\epsilon)^2 \mu(B(v)) \mu(B(w)) + O\left( f(\vec \eta,\epsilon) \cdot  \mu(B(v))\cdot e^{-\kappa t}\right) \\
        & \le\mu(B(v))^2 \left( 1 + 2\epsilon + \epsilon^2 +  \frac{1}{\mu(B(v))} O\left( f(\vec \eta,\epsilon) \cdot e^{-\kappa t}\right)\right),
    \end{align*}
which is the desired upper bound.  

For the lower bound part of \eqref{eq:mixing}, we pick $\hat \chi^{\epsilon}_{B}$ satisfying  (\ref{item:smooth1}) and (\ref{item:smooth2}) and with the additional properties that for all $X,v$, we have $\hat \chi_{B(v)}^\epsilon \le \chi_{B(v)}$ and $\int \hat \chi^{\epsilon}_{B(v)} d\mu \ge (1-\epsilon) \mu(B(v))$.  Then an analogous calculation to the one above, with inequalities flipped, yields 
    \begin{align*}
        \mu(g_{-t}B(v) \cap B(w)) &= \langle g_{-t} \chi_{B(v)}, \chi_{B(w)} \rangle \\
        & \ge \mu(B(v))^2 \left( 1 - 2\epsilon + \epsilon^2 +  \frac{1}{\mu(B(v))} O\left( f(\vec \eta,\epsilon) \cdot e^{-\kappa t}\right)\right),
    \end{align*}
which is the desired lower bound.  This completes the proof.  
  \end{proof}

The error terms in the above become manageable when $L$ grows logarithmically with genus.  We will often use the following specific instance: 

\begin{cor}
    \label{cor:effect-mix-log} 
    Given $\delta, \vec \eta_0, \epsilon > 0$, there exists $C=C(\delta, \vec \eta_0, \epsilon)$ with the following property.  Let $X$ be a $\delta$-expander surface.  Let $\vec \eta$ be such that each of its components is within a factor of $1000$ of the corresponding component of $\vec \eta_0$.  Let $v,w\in T^1X$ %
    be such that the $3\vec \eta$ flow boxes $B^+(v)$ and $B^+(w)$ are embedded.  Then for all $t\ge C \log g$, 
    \begin{align*}
      \mu\left(g_{-t}B(v) \cap B(w)\right) = \mu(B(v))^2 \cdot \left(1 +O(\epsilon)\right),
    \end{align*}
     where the implicit constant in $O()$ is less than $1$. 
\end{cor}

\begin{proof}
    We apply \Cref{lem:effect-mix}. We need to bound the error term 
    \begin{align*}
        \frac{1}{\mu(B(v))}O(f(\vec \eta,\epsilon) e^{-\kappa t})
    \end{align*}
    from above.  For fixed $\eta_0$, we have that $\mu(B(v))$ is approximately proportional to $1/g$ (recall the $\mu$ is normalized to have total mass $1$).  For $t>C \log g$, we have that $e^{-\kappa t} < g^{-C\kappa}$.  So by taking $C$, large, we can make the above error term arbitrarily small.  The result follows.  
    
\end{proof}

\subsection{Intersections of flow boxes}
\label{sec:inter-flow-box}

We wish to understand the geometry of connected components of $B \cap g_{-t}B$, where $B$ is an $\eta$ flow box, and $t$ is relatively large. Because geodesic flow contracts exponentially in the stable horocycle direction, and expands in the unstable direction, we get that connected components should have width $\eta$ in the stable direction, and $e^{-t} \eta$ in the unstable direction. The following lemma makes this precise to account for ``edge effects".
\begin{lemma}
    \label{lem:inter-box-eps}
 Let $\eta$ be such that $0<\eta<1$.  Let $\kappa$ be such that $0 < \kappa \leq 2$. Then 
there exists a $\tau_0=\tau_0(\kappa,\eta)$ such that for all $t>\tau_0$ the following holds. Let $X$ be a hyperbolic surface so that any $5 \eta$ flow box is embedded.  Then if $B_0, B_1$ are two $\eta$ flow boxes in $T^1X$,  every component of  $B_0 \cap g_{-t}B_1$ is contained in a unique component of  $B^{+\kappa}_0 \cap g_{-t}B^{+\kappa}_1$  that has full width in the contracting direction (relative to $B_0^{+\kappa}$), and width $e^{-t}\eta$ in the expanding direction.
\end{lemma}
\begin{proof}
First, as $B_0 \subset B_0^{+\kappa}$, and likewise for $B_1$, we have that $B_0 \cap g_{-t}B_1 \subset B_0^{+\kappa} \cap g_{-t}B_1^{+\kappa}$. Thus, if $x, y \in B_0 \cap g_{-t}B_1$ lie inside the same connected component in $B_0 \cap g_{-t}B_1$, then they must lie inside the same connected component of $B_0^{+\kappa} \cap g_{-t}B_1^{+\kappa}$. 

On the other hand, suppose $x,y$ lie on different connected components in $B_0 \cap g_{-t}B_1$. Let $h$ be a horocycle segment through $x$ that joins $x$ to the boundary of the connected component of $y$. Then $g_t h$ has both endpoints in $B_1$, but must leave $B_1$ and come back. Thus, $g_t h$ has length at least $s_0 - \eta$, where $s_0$ is the injectivity radius of $X$. We note that since the $5 \eta$ flow box is embedded, $g_t h$ has length at least $4 \eta$. Thus, $g_t h$ is not contained in $B_1^{+\kappa}$, as $\kappa \leq 2$, and so the width of $B_1^{+\kappa}$ in the unstable horocycle direction is at most $3\eta$. And so $x$ and $y$ cannot lie in the same connected component of $B_0^{+\kappa} \cap g_{-t}B_1^{+\kappa}$.

  Take a component $\Delta$ of $B_0^{+\kappa} \cap g_{-t}B_1^{+\kappa}$ that intersects $B_0 \cap g_{-t}B_1$. The fact that $B_0$ is open, and its closure is compact, implies that for our given $\kappa$, there is some length scale $\tau_0$ that every point in $B_0$ lies on a stable and geodesic flow segment of length $\eta$, and an unstable flow segment of length $e^{-t} \eta$ inside $B_0^{+\kappa}$. Moreover, the same is true for $B_1$. Thus, by the same arguments as in \cite[Section 8.7]{fh2019}, $\Delta$ has full width in the stable direction, and width $e^{-t} \eta$ in the geodesic flow direction.
\end{proof}

\begin{lemma}
\label{lem:ave-flow-width}
Fix $\delta, \eta, \epsilon >0$.  There exists a $C$ with the following property. Let $X$ be a $\delta$-expander surface, and let $B, B' \subset T^1X$ be embedded $\eta$ flow boxes.  Then for any $t > C \log g$, the average width of components of $B \cap g_{-t}B'$ in the flow direction is $$(1+O(\epsilon))\eta/2,$$
where the implied constant in $O(\cdot)$ is less than $1$.  
\end{lemma}

\begin{proof}
    We divide $B$ into $k=\lfloor 1/\epsilon \rfloor$ flow boxes $B_1,\ldots, B_k$ that are almost disjoint, have full width in the expanding and contracting directions, and have width $\eta/k$ in the flow direction.  Similarly divide $B'$ into $B_1',\ldots,B_k'$. By \Cref{cor:effect-mix-log} (Effective mixing), for sufficiently large $C$ (depending on $\epsilon)$, each $B_i\cap g_{-t}B_j'$ has measure equal to $(1+O(\epsilon)) \mu(B_i)^2$, when $t>C \log g$; thus, up to small error, these sets have measure independent of $i,j$.  This implies that on average the components of $B\cap g_{-t}B'$ extend half-way through $B$ in the flow direction (up to multiplicative error $1+O(\epsilon)$).  
\end{proof}

\subsection{From counting in flow boxes to total counts}
\label{sec:count-box-surf}
Let $G$ be any subset of the set of closed geodesics on $X$.  We let
\[
N_G(X,L) := \{\gamma \in G: \ell_X(\gamma) \le L\}. 
\]
More generally, for any interval $[a,b]$, we define
\[
 G[a,b] := \{ \gamma \in G \ | \ \ell_X(\gamma) \in [a,b]\},
\]
and we let
\[
N_G(X,[a,b]): = \#G[a,b],
\]
that is, this is the total number of closed geodesics in $G$ whose length falls in the interval $[a,b]$.

Given a flow box $B$, we define $N_G(B,L,\eta)$ to be the number of length $\eta$ geodesic segments in $B$ that lie on some $\gamma \in G[L-\eta,L+\eta]$. In other words, $N(B,L,\eta)$ is the number of times the geodesics in $G$ pass through $B$ in a full length $\eta$ segment. (Note that this excludes edge effects: if a geodesic in $G$ spends less than time $\eta$ in $B$, then it is not counted.)

In this section, we show how to estimate $N_G(X,[L-\eta,L+\eta])$ given an estimate on $N_G(B,L,\eta)$. In other words, we show how to globally estimate the number of geodesics of type $G$ given an estimate of the number of times such geodesics pass through any flow box.
\begin{lemma}
\label{lem:int-flows}
 Let $s_0>0$. Let $0 < \eta < \frac 12 s_0$. Suppose $X$ is a hyperbolic surface whose systole is greater than $s_0$.
 For each $v \in T^1X$, let $B(v)$ be the $\eta$-flow box centered at $v$. Then for all $L$, we have
 \[
  N_G(X,[L-\eta,L+\eta]) \cdot L \cdot \mu(B) = (1 + O(\eta)) \int_{v \in T^1X}\eta \cdot N_G(B(v),L,\eta)d\mu(v)
 \]
where $\mu(B)$ is the measure of any $\eta$-flow box, and the implicit constant in $O(\eta)$ depends on $s_0$, but is independent of $L$ and $X$.
\end{lemma}
\begin{proof}
As $s_0$ bounds the systole of $X$ from below, there are no closed geodesics of length smaller than $s_0$ on $X$. %
Thus, if $L < s_0/2$, then $N_G(X, [L-\eta, L+\eta]) = N_G(B, L, \eta) = 0$ for all $B$.  So in this case the statement of the lemma is trivially true. So we can assume throughout that $L > \frac 12 s_0$.

The flow box $B(v)$ is foliated by geodesic segments. Most segments have length $\eta$, which can be quantified in the following sense. Suppose $\alpha$ is a segment of this foliation whose length is smaller than $\eta$. Then by \Cref{lem:full-geod-seg}, $\alpha$ is disjoint from the flow box $B^*(v)$ that has width $\eta e^{-\eta}$ in the expanding direction, and width $\eta$ in both the geodesic flow and contracting directions. Thus, we note that,
\[
\eta \cdot N_G(B(v),L,\eta)\leq \sum_{\substack{\gamma \in G[L-\eta,L+\eta]}} \int_0^{\ell(\gamma)}\chi_{B(v)}(\gamma(t)) dt
\]
and
\[
\eta \cdot N_G(B(v),L,\eta)\geq \sum_{\substack{\gamma \in G[L-\eta,L+\eta]}} \int_0^{\ell(\gamma)}\chi_{B^*(v)}(\gamma(t)) dt.
\]
We now integrate these estimates for $N_G(B(v), L,\eta)$, with respect to $v \in T^1(X)$. Switching the order of integration, we see that we get a term of the form $\int_{v \in T^1X} \chi_{B(v)}(\gamma(t)) d\mu(v)$, and an analogous term with $B^*(v)$.

We can treat $\chi_B(v)(\gamma(t))$ as a function in both $v$ and $t$. We note that the set of $v$ so that $\gamma(t) \in B(v)$ is almost the same as $B(\gamma(t))$. In fact,
\[
\mu\{v \in T^1 X \ | \ \gamma(t) \in B(v)\} = (1 + O(\eta)) \mu(B)
\]
and likewise for $B^*$. Thus, after switching the order of integration, we get the follow upper bound:
\begin{align*}
 \int_{v \in T^1X} &\eta \cdot N_G(B(v),L,\eta)  d\mu(v)  \\
       & \leq \int_{v \in T^1X} \sum_{\substack{\gamma \in G[L-\eta,L+\eta]}} \left(\int_{t=0}^{\ell(\gamma)} \chi_{B(v)}(\gamma(t)) dt\right) d\mu(v)\\
       &\leq \sum_{\gamma \in G[L-\eta,L+\eta]} \int_{t=0}^{\ell(\gamma)} \left(\int _{T^1X}\chi_{\{v:\gamma(t)\in B(v)\}} d\mu(v)\right) dt\\
       & \leq  N_G(X,L-\eta, L+\eta) \cdot (L+\eta) \cdot \mu(B) (1 + O(\eta)).
\end{align*}

On the other hand, we get lower bounds instead of upper bounds, by using $B^*$ instead of $B$ in the inequalities above. Thus, 
\begin{align*}
\int_{v \in T^1X} &\eta \cdot N_G(B(v),L,\eta)  d\mu(v)  \\
&\geq  N_G(X,L-\eta, L+\eta) \cdot (L-\eta) \cdot \mu(B^*) (1 - O(\eta)).
\end{align*}
 Now we use that $B^*$ has width $\eta e^{-\eta}$ in the expanding direction. %
 Thus,
 \begin{align*}
     \mu(B^*) & = e^{-\eta} \mu(B) \\
       & = (1-O(\eta))\mu(B)
 \end{align*}
 and so
 \begin{align*}
 (1 - O(\eta)) \mu(B^*) & =  (1-O(\eta))^2\mu(B) \\
    & =  (1-O(\eta))\mu(B). 
 \end{align*}
So we can replace $B^*$ with $B$ in the above inequality.

Recall $s_0$ is the lower bound on the systole of $X$, as we can assume that $L > \frac 12 s_0$. Then we have
\begin{align*}
    L + \eta & \leq L + \frac 12 s_0 (2\eta/s_0)\\
     & \leq L + L(2\eta/s_0) \\
     & = L(1 + O(\eta)),
\end{align*}
 where the constant implicit in $O(\eta)$ depends only on $s_0$. Moreover, as $\eta < s_0/2$, we have $(1 + O(\eta))(1+O(\eta)) = 1+O(\eta)$ where again the implicit constant depends only on $s_0$. Thus, $(L+\eta)(1 + O(\eta)) = L(1 + O(\eta))$. Likewise, $(L-\eta)(1-O(\eta)) = L(1-O(\eta)$. Thus, for all $L > s_0$, we have:
\[
 \int_{v \in T^1X} \eta \cdot N_G(B(v),L,\eta)  d\mu(v)  = N_G(X,L-\eta, L+\eta) \cdot L  \cdot \mu(B) (1 + O(\eta))
.\]
Note that the constant in the $O(\eta)$ error term depends only on $s_0$.%

\end{proof}

\subsection{Bounding number of geodesics of given type}
\label{sec:bound_given_type}

In this section, we show how bounds on $N_G(B,L,\eta)$ can be turned in to bounds on $N_G(X,L)$. First, for each $0<a<b$, recall that $N_G(X,[a,b])$ is the
number of closed geodesics in $G$ with length in the interval $[a,b]$. Using the effective prime geodesic theorem to deal with short geodesics in $G$, we can turn bounds on $N_G(X,[t+\eta,t-\eta])$ for all $t$ large enough into bounds on $N_G(X,L)$.
\begin{lemma}
\label{lem:shell-sum}
Let $X$ be a hyperbolic surface. Fix $\eta$ so that $0 < \eta < 1$. Let $L > 0$. Suppose there exists $C > 0$ so that for all $L/2 < t < L$, 
$N_G(X,[t-\eta,t+\eta]) < C \eta e^t/t$. Then,
 \[
 N_G(X,[L/2,L]) < 6 C \frac{e^L}{L}.
 \]
 
\end{lemma}
\begin{proof}
We can bound $N_G(X,[L/2,L])$ by the sum:
\[
 N_G(X,[L/2,L]) \leq \sum_{\frac L{2\eta}\leq i  \leq \frac L\eta }  N_G(X,[\eta (i-1),\eta (i+1)]).
\]
For all $i$ with $L/2\eta \leq i \leq L/\eta$, we have $L/2 \leq \eta i \leq L$. Thus, $N_G(X,[\eta (i-1),\eta (i+1)]) < C\eta \cdot \frac{e^{i \eta}}{i\eta}$. And so,
\[
 N_G(X,[L/2,L]) \leq C \eta \sum_{\frac L{2\eta}\leq i  \leq \frac L\eta }  \frac{e^{i\eta}}{i\eta}.
\]

Next, as $i \eta \geq L/2$ for all $i$, we have
\begin{align*}
\sum_{\frac L{2\eta}\leq i  \leq \frac L\eta }  \frac{e^{i\eta}}{i\eta} &\leq \frac 2L \sum_{\frac L{2\eta}\leq i \leq \frac L\eta } e^{i \eta} \\
& < \frac 2L\sum_{0 \leq i  \leq \frac L\eta } e^{i \eta} \\
& < \frac 2L \cdot \frac{e^\eta}{e^{\eta} - 1} e^L.
\end{align*}

Thus,
\[
 N_G(X,[L/2,L]) \leq  2C \frac{\eta e^\eta}{e^{\eta} - 1}\frac{e^L}{L}.
\]
We note that $e^\eta - 1 > \eta$ for all $\eta > 0$, so in particular, 
\[N_G(X,[L/2,L]) \leq  2C e^\eta \frac{e^L}{L}.\]

Since we chose $\eta < 1$, we see that $e^\eta < 3$, giving us the statement of the lemma.
\end{proof}

Finally, we will combine \Cref{lem:int-flows} and \Cref{lem:shell-sum} to estimate the total number of geodesics $N_G(X,L)$ given a bound on $N_G(B,t,\eta)$ for all flow boxes $B$, and all $t$ large enough. That is, if we can bound how many length $\eta$ arcs in a given flow box $B$ belong to geodesics in $G[t-\eta, t+\eta]$, then we can estimate the total number of geodesics in $G[0,L]$:

\begin{prop}
\label{prop:fb-to-total}
Fix $s_0 > 0$. Then for all $\eta > 0$ sufficiently small, the following holds. Let $X$ be a hyperbolic surface with systole at least $s_0$, and $G$ any set of closed geodesics on $X$. Suppose there is an $\epsilon > 0$ and an $L > 0$ so that, for all $\eta$ flow boxes $B \subset T^1X$ and all $t \in [L/2, L]$ we have
\[
 N_G(B,t,\eta) < \epsilon \mu(B) e^t .
\]
Then, 
\[
N_G(X,[L/2,L]) < 12 \epsilon \frac{e^L}{L}.
\]
\end{prop}
\begin{proof}

Let $X$ be a hyperbolic surface with systole at least $s_0$. Choose $\eta$ so that $0< \eta < s_0/2$. We will show that for all such $\eta$, we have
\[ N_G(X,[L/2,L]) < 6\epsilon (1 + O(\eta)) \frac{e^L}{L}\]
where the implicit constant in $O(\eta)$ depends only on $s_0$. Then for all $\eta$ sufficiently small, the $O(\eta)$ term is smaller than 1, giving us the statement of the proposition.

As $0 < \eta < s_0/2$, we can apply \Cref{lem:int-flows} to get 
 \[
  N_G(X,[t-\eta,t+\eta]) \cdot t \mu(B) = (1 + O(\eta)) \int_{v \in T^1X}\eta N_G(B(v),t,\eta)d\mu(v)
 \]
 where the constant in $O(\eta)$ depends only on $s_0$. 
 
 We assume that for all $t$ with $L/2 \leq t \leq L$, we have $N_G(B(v),t,\eta)< \epsilon  \mu(B)e^t$. So, 
 \[
 \int_{v \in T^1X}\eta N_G(B(v),t,\eta)d\mu(v) < \int_{v \in T^1X}\epsilon \eta \mu(B) e^t d\mu(v) = \epsilon \eta \mu(B) e^t
 \]
since $\mu$ is a probability measure. So,
\[
 N_G(X,[t-\eta,t+\eta]) < \epsilon \eta (1 + O(\eta)) \frac{e^t}{t}
\]
for all $L/2 < t < L$.

Setting $C = \epsilon (1 + O(\eta))$, we apply \Cref{lem:shell-sum} to get that
\[
 N_G(X,[L/2,L]) < 6\epsilon (1 + O(\eta)) \frac{e^L}{L},
\]
where, by \Cref{lem:int-flows}, the implicit constant depends only on $s_0$.

\end{proof}

\subsection{Anosov Closing Lemma}
\label{sec:Anosov_closing}

In this section, we give two versions of the Anosov Closing Lemma, relating components of $B\cap g_{-L}B$ to closed geodesics of length approximately $L$ passing through $B$.  These are standard results, and various versions can be found in the literature (see e.g. \cite{hk1995}).

The first lemma states that at most one simple closed curve of length roughly $L$ can pass through any given component of $B \cap g_{-L}B$. Moreover, any such simple closed curve passing through this component can only pass through at most once. It follows from the fact that geodesic flow $g_L$ has a contracting action in the stable horocycle direction for $L$ positive, and in the unstable horocycle direction for $L$ negative. We give a slightly expanded idea of the proof here for completeness.
\begin{lemma}
  \label{lem:closing-unique}
  Let $s_0 > 0$. There exists $\eta_0, L_0$ with the following property.  Let $X$ be a hyperbolic surface with systole at least $s_0$. Let $\eta\le \eta_0$ and let $B\subset T^1X$ be an embedded $\eta$ flow box.
Let $L > L_0$, and let $P$ be a connected component of $B \cap g_{-L} B$. Suppose $v_1, v_2 \in P$ lie on closed geodesics of length between $L-\eta$ and $L+\eta$. Then $v_1$ and $v_2$ lie on the same geodesic segment in $P$.

\end{lemma}

\begin{proof}

We choose $\eta_0$ so that any $\eta_0$-flow box is embedded. Let $B$ be an $\eta$-flow box for $\eta < \eta_0$. Let $L_0 > 2 s_0 + 1$, and take any $L > L_0$. 

Let $P$ be any component of $B \cap g_{-L}B$. The horocycle and geodesic flows foliate $X$ by center-unstable leaves, which in particular, induce a foliation of $P$. The geodesic flow acts by contraction in the stable direction, which is transverse to the center-unstable leaves. In particular, the map $g_L$ sending $P$ to $g_LP$ acts by contraction, and sends leaves of the foliation of $P$ to leaves of the foliation of $g_L P$. Thus, $g_L$ fixes at most one such leaf. 

Next, the fixed center-unstable leaf is foliated by geodesic segments, which are transverse to the unstable direction. Again, $g_{-L}$ acts by contraction in the unstable direction, so at most one geodesic segment is fixed. If such a fixed geodesic segment exists, then it is the unique segment in $P$ that lies on a closed geodesic of length $\ell \in [L-\eta, L+\eta]$. 

Hence, if $v_1, v_2 \in P$ lie on a closed geodesic of length  $\ell \in [L-\eta, L+\eta]$, then they must lie on the same geodesic segment in $P$.

\end{proof}

The second version of the Anosov Closing Lemma directly relates the number of components of $B \cap g_{-L}B$ to the number of times closed geodesics of length roughly $L$ can pass through $B$.        
\begin{lemma}
  \label{lem:closing-exists}
  Let $\kappa, s_0>0$.  There exist $\eta_0=\eta_0(\kappa,s_0)$ and $L_0=L_0(\kappa,s_0)$ such that the following holds. Let $X$ be any hyperbolic surface with systole at least $s_0$, let $\eta\le \eta_0$, and let $B\subset T^1X$ be an $\eta$ flow box.  Let $N(B,L,\eta)$ be the number of length $\eta$ geodesic segments inside of $B$ that lie on a closed geodesic of length in $[L-\eta,L+\eta]$.  Then for any $L\ge L_0$, 
  \begin{align*}
   \comp(B^{-\kappa} \cap g_{-L}B^{-\kappa}) \le  N(B,L,\eta) \le \comp(B \cap g_{-L}B).
  \end{align*}
\end{lemma}

For an idea of the proof, suppose that our flow box $B$ were really a Euclidean cube. Moreover, $g_L$ maps any component $P$ of $B \cap g_{-L} B$ to a subset $g_LP \subset B$. Up to edge effects, a component $P$ of $B \cap g_{-L}B$ would have full width $\eta$ in the contracting direction. The contracting action of $g_L$ then allows us to find a fixed point in the contracting direction. Likewise, we can use $g_{-L}$ to find a fixed point in the expanding direction. This means that for some $v \in P$, both $v$ and $g_L v$ lie on the same geodesic segment in $B$. But this means that $v$ lies on a closed geodesic of length roughly $L$ passing through $B$. We now make this proof precise.
\begin{proof}
We choose $\eta_0$ so that any $3\eta_0$-flow box is embedded, and so that $1-\kappa/2 < e^{-\eta_0}$. By \Cref{lem:full-geod-seg}, the latter condition guarantees that each $v \in B^{-\kappa/2}$ lies on a geodesic segment of length $\eta$ in $B$. Let $L_0$ be large enough so that any connected component $P$ of $B^{-\kappa/2} \cap g_{-L} B^{-\kappa/2}$ that intersects $B^{-\kappa} \cap g_{-L} B^{-\kappa}$ is a full component of intersection, as defined in \cite[Definition 8.7.4]{hk1995}. 

Suppose $\delta$ is a geodesic segment of length $\eta$ in $B$, so that $\delta$ lies on a closed geodesic of length between $L-\eta$ and $L+\eta$. That is, there is some $v \in \delta$ so that $g_tv = v$ for $v \in [L-\eta, L+\eta]$. But then, there is some $t_0$ with $|t_0| < \eta$ so that $w = g_{t_0}v \in B$, and $g_Lw = g_{L+t_0}v \in B$. In other words, $w$ lies on the same geodesic segment $\delta$ as $v$, and $g_Lw = w$. Thus, $w$ lies in some connected component of $B \cap g_{-L}B$. In other words, $\delta$ intersects some connected component $P$ of $B \cap g_{-L}B$.

Moreover, $\delta$ cannot intersect more than one connected component. By choosing $\eta < s_0/3$, we guarantee that any geodesic segment with endpoints on $B$, and interior disjoint from $B$, must have length at least $2\eta$. But $\delta$ has length $\eta$, so $g_L(\delta) \cap B$ must be connected. Thus, any geodesic segment of length $\eta$ can intersect at most one connected component of $B \cap g_{-L}B$. Therefore,
\[
N(B,L,\eta) \leq \comp(B \cap g_{-L}B).
\]

For the lower bound, let $P^-$ be a connected component of $B^{-\kappa/2} \cap g_{-L}B^{-\kappa/2}$ that intersects $B^{-\kappa} \cap g_{-L}B^{-\kappa}$. By \cite[Lemma 8.7.5]{hk1995}, there is a unique geodesic segment $\delta$ in $B^-$ so that $\delta$ lies on a closed geodesic of length between $L-(1-\kappa/2)\eta$ and $L+ (1-\kappa/2)\eta$, and $\delta \cap P^- \neq \emptyset$. In particular, let $v \in \delta \cap P^-$.  Then, by \Cref{lem:full-geod-seg}, $v$ lies on a geodesic segment of length $\eta$ in $B$. Therefore,
\[
\comp (B^{-\kappa} \cap g_{-L}B^{-\kappa}) \le  N(B,L,\eta). 
\]

\end{proof}

\section{Effective prime geodesic theorem}
\label{sec:effect-prime-geod}  

Using techniques developed by Margulis (\cite{Margulis_thesis},
also see \cite[Section 20.6]{hk1995}) and effective mixing (\Cref{lem:effect-mix}), we prove \Cref{thm:effect-pgt}, an effective version of the prime geodesic theorem for surfaces with definite spectral gap.  While we were writing this paper, Wu-Xue proved a related, somewhat stronger, result \cite[Theorem 2]{wx2022}; they use the Selberg Trace Formula, which is a fundamentally different approach.  

We break the proof into the two lemmas below, dealing with the longer and shorter geodesics, respectively.  The basic strategy is similar for the two, but the technical issues are a bit different.  
\begin{lemma}
    \label{lem:effect-pgt-half}
    Fix $\delta>0, s_0>0, \epsilon>0$.  There exists a constant $c=c(\delta,s_0,\epsilon)$ such that for any $\delta$-expander surface $X$ of genus $g$ with systole greater than $s_0$, and $L> c\log g$,
    \begin{align*}
        N(X,[L/2,L]) = (1+O(\epsilon)) \cdot \frac{e^L}{L},
    \end{align*}
    where the implicit constant in $O()$ has magnitude less than $1$.  
\end{lemma}

\begin{proof}
The only ways in which we will use the particular geometry of the surface $X$ are (i) a lower bound on systole to ensure that the flow boxes are embedded, and (ii) the rate of mixing.  

Choose $\eta$ small, which for now means small enough such that every $5\eta$ flow box is embedded; later we will make $\eta$ even smaller. 
For $\eta$ sufficiently small, each $\eta$ flow box $B:=B(v) \subset T^1X$ behaves very much like a product.

By \Cref{cor:effect-mix-log} applied with the box $B^{+\kappa}$, where $\kappa$ is small, we get there exists $c=c(\delta,\eta,\epsilon)$ such that for $t\ge c\log g$, we have
\begin{align}
    \mu(g_{-t}B^{+\kappa} \cap B^{+\kappa}) = \mu(B^{+\kappa})^2 \left(1 +O(\epsilon) \right)= \mu(B)^2 \left(1 +O(\epsilon) \right), \label{eq:mix-soft}
\end{align}
where the implicit constant in $O()$ is less than $1$.

Now we will use our knowledge of the shape of components of $g_{-t}B\cap B$.  We apply \Cref{lem:inter-box-eps} and enlarge our $c$ so that $c\log g$ is greater than the $\tau_0$ from that lemma (for all $g\ge 2)$.  We get that, for $t\ge c \log g$, each component of $g_{-t}B\cap B$ is contained in a component of $g_{-t}B^{+\kappa}\cap B^{+\kappa}$ that has full width in the contracting direction and width $e^{-t}\eta$ in the expanding direction; denote by $\mathcal C$ the set of these components of $g_{-t}B^{+\kappa}\cap B^{+\kappa}$.  Note that each component of $\mathcal C$ has width at least that of the corresponding component of $g_{-t}B \cap B$; hence the average width in the flow direction over the components in $\mathcal C$ is at least that of the average for $g_{-t}B \cap B$, which by \Cref{lem:ave-flow-width} is $(1+O(\epsilon))(\eta/2)$ (making the constant $C$ larger if necessary).  %
It follows that 
\begin{align}
    \text{ average vol of components of }\mathcal C %
        &\ge (1+O(\epsilon)) (1/2) e^{-t} \mu(B).  \label{eq:ave-vol}
\end{align}
Then using \eqref{eq:mix-soft} and \eqref{eq:ave-vol}, we get
\begin{align}
    \comp(g_{-t}B\cap B) = |\mathcal C| &\le \frac{\mu(g_{-t}B^{+\kappa}\cap B^{+\kappa})}{\text{ average volume of components of }\mathcal C} \\
    & \le  \frac{\mu(B)^2(1+O(\epsilon))}{(1/2)e^{-t} \mu(B) (1+O(\epsilon))} \\
    & = 2 e^t \mu(B) (1+O(\epsilon)).  
\end{align}

A similar argument using the smaller box $B^{-\kappa}$ gives a lower bound akin to the one above, and altogether we get 
\begin{align}
    \comp(g_{-t}B\cap B) = 2 e^t \mu(B) (1+O(\epsilon)). \label{eq:num-comps}
\end{align}

We also get analogous bounds for $\comp(g_{-t}B^{+\kappa} \cap B^{+\kappa})$ and $\comp(g_{-t}B^{-\kappa} \cap B^{-\kappa})$.  Then by \Cref{lem:closing-exists} (Anosov Closing), 
we get that for $L>c \log g$ (where $c$ is made larger if necessary so that $c\log g$ is greater than the $L_0$ from \Cref{lem:closing-exists}), 
\begin{align*}
    N(B,L,\eta) = 2e^L \mu(B) \cdot \left(1 + O(\epsilon)\right).
\end{align*}

To count the number $N(X,[L-\eta,L+\eta])$ of all closed geodesics of length in $[L-\eta, L+\eta]$ we apply \Cref{lem:int-flows} and the above, taking $\eta$ sufficiently small, to get
\begin{align*}
   N(X,[L-\eta, L+\eta]) \cdot L \cdot \mu(B) 
     &=(1+O(\epsilon+\eta)) \int _{T^1X} \eta \cdot N(B(v),L,\eta)  d\mu(v)\\
    &=(1+O(\epsilon+\eta)) \eta \cdot 2e^L\mu(B). 
 \end{align*}

Rearranging gives
\begin{align}
  N(X, [L-\eta, L+\eta]) =  \frac{2\eta e^L}{L} (1+O(\epsilon+\eta)). \label{eq:experror}
\end{align}
To compute $N(X, L)$ we sum the above over $L$ with values separated by $2\eta$ (by doubling $c$, we can ensure the estimate \eqref{eq:experror} above holds for lengths in $[L/2,L]$ when $L\ge c\log g$). %
We approximate the sum by the integral $\int_{L/2}^L (e^t/t) dt$.  By taking $\eta$ small, we can get this estimate accurate to within a factor of $1+\epsilon/2$; this may involve increasing $c$.  This integral is in turn equal to $(1+O(\epsilon+\eta)) e^L/L$ (recall that we are taking $L>c\log g$, where the $c$ can be a large constant, and hence we can assume $L$ is large).  Since $\eta$ is small, this gives the desired result. %

\end{proof}

We now give a coarser bound that works at lower length scales.  

\begin{lemma}
    \label{lem:effect-pgt-upper}
    Fix $\delta>0, s_0>0$.  There exists a constant $c=c(\delta,s_0)$ such that for any $\delta$-expander surface $X$ of genus $g$ with systole greater than $s_0$, and $L> c \log g$,
    \begin{align*}
        N(X,L) \le c e^L.
    \end{align*}
\end{lemma}

\begin{proof}

    This proof parallels the proof of \Cref{lem:effect-pgt-half}, but the technical issues are somewhat different.  We do not need to be as precise, since it is allowable to lose constant factors, but on the other hand, the estimates must also account for geodesics of smaller length.  

    We first address the geodesics of \textit{bounded} length (independent of genus); here we will use an a priori bound that does not need the expander assumption or effective mixing.  Let $\tau$ be any fixed number.  Then, by \cite[Theorem 6.6.4]{buser2010}
    \begin{align*}
        N(X,\tau) \le (g-1) e^{\tau + 6} + (3g-3) \tau/s_0. 
    \end{align*}
    (The second term on the right accounts for geodesics that are iterates of geodesics of length $\le \operatorname{arcsinh} 1$, of which there are at most $3g-3$; the first term accounts for the others.) 
    The relevant feature of the above for us is that it grows at most polynomially (in fact, linearly) in genus.  In particular, when $L>c\log g$ (for $c>1$), we get 
    \begin{align}
        N(X,\tau) = O(e^L) \label{eq:0tau},
    \end{align}
    where the implicit constant in the $O()$ depends only on $c,s_0,\tau$.  
    
    Now fix $\eta$ small, so that every $5\eta$ flow box is embedded (thus how small $\eta$ needs to be depends on the systole $s_0$, but not anything else, and in particular not on genus).  Let $B\subset T^1X$ be any $\eta$ flow box.  

     By \Cref{lem:effect-mix}, for $t\ge 0$ (using any fixed small $\epsilon'$), we get
    \begin{align}
        \mu(g_{-t}B^{+} \cap B^{+}) =  O\left( \mu(B^{+})^2 \left(1 + \frac{1}{\mu(B^{+})}e^{-\kappa t})\right) \right),
        \label{eq:mix-plus} 
    \end{align}
    where $\kappa>0$ only depends on $\delta$, and the implicit constant in $O()$ depends on $\eta$. 

    We apply \Cref{lem:inter-box-eps} with $\kappa =2$, so that $B^{+\kappa} = B^+$.  This gives that for $t\ge \tau_0(2,\eta)$, each component of the set $g_{-t}B\cap B$ is contained in a component of $g_{-t}B^{+}\cap B^{+}$ that has full width in the expanding direction and width $e^{-t}\eta$ in the contracting direction; denote by $\mathcal C$ the set of these components of $g_{-t}B^{+}\cap B^{+}$.  
    It follows that 
    \begin{align}
        \text{ ave vol of components of }\mathcal C &\ge \text{ min vol of components of }\mathcal C \\
            &\ge e^{-t} \mu(B),  \label{eq:ave-vol-upper}
    \end{align}
    where we have used that the width in the flow direction of components of $\mathcal C$ is at least the width of $B$ in the flow direction.

    Then using \eqref{eq:mix-plus} and \eqref{eq:ave-vol-upper}, we get
    \begin{align*}
        \comp(g_{-t}B\cap B) = |\mathcal C| &\le \frac{\mu(g_{-t}B^{+}\cap B^{+})}{\text{ave vol of components of }\mathcal C} \\
        & \le   \frac{O\left( \mu(B^{+})^2 \left(1 + \frac{1}{\mu(B^{+})}e^{-\kappa t})\right) \right)}{e^{-t} \mu(B)} \\
        & = O\left( e^t (\mu(B) + e^{-\kappa t} ) \right),  
    \end{align*}
    where the implicit constant in $O()$ depends on $\eta$.  We also get analogs of the above for $\comp(g_{-t}B^+\cap B^+)$ and $\comp(g_{-t}B^-\cap B^-)$.  Now we use \Cref{lem:closing-exists} (Anosov Closing), combined with the above to get that, for $L_0$ the max of $\tau_0$ and the $L_0$ from \Cref{lem:closing-exists} and any $L\ge L_0$, 
    \begin{align*}
        N(B,L,\eta) = O\left( e^L (\mu(B) + e^{-\kappa L} ) \right).
    \end{align*}

    To count the number $N(X,[L-\eta,L+\eta])$ of all closed geodesics of length in $[L-\eta, L+\eta]$ we apply \Cref{lem:int-flows} and the above to get %
    \begin{align*}
       N(X,[L-\eta, L+\eta]) & \cdot L \cdot \mu(B) 
        = (1+O(\eta)) \int _{T^1X} \eta \cdot N(B(v),L,\eta)  d\mu(v)\\
        & \le  O\left( \eta \cdot e^L (\mu(B) + e^{-\kappa L} ) \right). 
     \end{align*}

     Moving terms to the right (and dropping the resulting $1/L$ term, which is not small) gives 
    \begin{align*}
       N(X,[L-\eta, L+\eta]) & = O\left( \eta \cdot e^L \left(1 + \frac{e^{-\kappa L}}{\mu(B)} \right)  \right).
    \end{align*}

Now, summing the above, and then approximating the sum by an integral, gives
\begin{align*}
    N(&X,[L_0,L])\\
    & \le N(X,[L_0,L_0+2\eta])+N(X,[L_0+2\eta,L_0+4\eta]) + \cdots + N(X,[L-\eta,L+\eta]) \\ 
    &= O\left( e^L \left(1 + \frac{e^{-\kappa L}}{\mu(B)} \right) \right).
\end{align*}

Now approximating $\mu(B)$ by $\eta^3/g$, we note that we can find $c$ (depending on $\eta$ and $\kappa$) such that for $L>c\log g$,
\begin{align*}
    \frac{e^{-\kappa L}}{\mu(B)} \le 1
\end{align*}
and thus %
\begin{align}
    N(X,[L_0,L]) = O(e^L), \label{eq:tauL}
\end{align}
where the implicit constant in $O()$ depends only on $\eta$.  Now $\eta$ depends only on the systole $s_0$, and $\kappa$ only on the spectral gap $\delta$, so in fact the implicit constant depends only on $\delta,s_0$.  

Now using \eqref{eq:0tau}, with $\tau=L_0$, and \eqref{eq:tauL}, we get  
\begin{align*}
    N(X,L) = N(X,L_0) + N(X,[L_0,L]) = O(e^L) + O(e^L) = O(e^L),
\end{align*}
which completes the proof. 

\end{proof}

\begin{thm}[Effective prime geodesic theorem]
\label{thm:effect-pgt}
Fix $\delta, s_0, \epsilon>0$.  There exists a constant $c=c(\delta,s_0,\epsilon)$ such that for any $\delta$-expander surface $X$ of genus $g$ with systole greater than $s_0$, and $L> c\log g$,
  \begin{align*}
    1-\epsilon \le  \frac{N(X,L)}{ e^L/L } \le 1+\epsilon.  
  \end{align*}
\end{thm}

\begin{proof}

Let $c=c(\delta,s_0,\epsilon)$ be the maximum of the constants given by \Cref{lem:effect-pgt-half} and \Cref{lem:effect-pgt-upper}.  

Then for $L> 2 c \log g$, by those lemmas,
\begin{align*}
    N(X,L) = N(X,[L/2,L]) + N(X,L/2) & = (1+O(\epsilon)) \cdot \frac{e^L}{L} + O(c e^{L/2}), 
\end{align*}
where the constants in both the $O()$'s are less than $1$ in magnitude. 
For $L$ large (depending on $c$ and $\epsilon$, but not on $g$), we have $c\cdot e^{L/2} < \epsilon \cdot e^L/L$.  Using this in the above, and making $c$ somewhat larger if necessary, we get, for $L>c\log g$,
\begin{align*}
    N(X,L) =  (1+O(2\epsilon)) \cdot \frac{e^L}{L},
\end{align*}
which yields the desired result.

\end{proof}

 \section{Effective multiple mixing and related tools}
\label{sec:effect-mult-mix}

\subsection{Lemmas on intersections with subboxes}
\label{sec:lems-inter}
Here we prove lemmas that will serve as the analog of the Markov property for random walks (future behavior is independent of past behavior).  These are used in the proof of effective multiple mixing \Cref{thm:effect-mult-mix}, and in the proofs of the main theorems.  Concretely, they concern the intersection of a subbox with the preimage of a full box under geodesic flow for a sufficiently large time.  While we do not have (fast) effective mixing for arbitrary thin subboxes, under certain conditions involving their shape, we can get control using effective mixing of the full flow boxes that contain them.  

\Cref{lem:band-mix} 
concerns a subbox that is full in the expanding direction; this condition corresponds to conditioning only on \textit{past} behavior.  This subbox is intersected with the preimage of a flow box under $g_{-T}$, i.e. conditioning on \textit{future} behavior.  The measure statement in this lemma corresponds to independence of the past and future.   %

We will deduce \Cref{lem:band-mix} from \Cref{lem:band-mix-aux}, which gives more information (specifically on intersection components' \textit{shape}, which corresponds to a condition coming from not too far in the future).  In applications in the sequel, we will frequently need this additional information, (since we apply the result iteratively).  For instance, this is what we do in the proof of effective multiple mixing \Cref{thm:effect-mult-mix}.  

A techincal obstacle is the phenomenon of ``edge effects", which means the shape of some components of intersection differs from the typical shape.  %
To deal with this, we make use of slightly enlarged flow boxes $B_i^{+\kappa}$ (defined in \Cref{sec:notation-setup}).

\begin{lemma}
    \label{lem:band-mix-aux}
        Fix $\delta, \epsilon >0$.  For any $\eta_0$ small (depending on $\delta, \epsilon$), there exists $c=c(\delta, \epsilon, \eta_0)$ with the following property.  Let $X$ be a $\delta$-expander surface.  
    Let $B_0,B_1\subset T^1X$ be $\eta$ flow boxes, where $\eta_0/1000 \le \eta \le \eta_0$, and such that $B_0^+,B_1^+$ are embedded.  Let $T\ge c\log g$.  Then there exists $F\subset T^1X$ such that:
    \begin{enumerate}[(i)]
        \item $F\subset B_0 \cap g_{-T}B_1$, \label{item:subset}
        \item every component of $F$ has (with respect to $B_0$) full width in the contracting direction and width $\ge e^{-T}\eta$ in the expanding direction,  \label{item:Fcomp}
        \item for any $P\subset B_0$ a subbox that is full width in the expanding direction, we have
        $$\mu (P\cap F) = (1+O(\epsilon)) \cdot \mu(P) \mu(B_1),$$ \label{item:muPcapF}
        where the implicit constant in the $O()$ is less than $1$.

    \end{enumerate}
\end{lemma}

\begin{proof}
 Take $\kappa>0$ small (see later in proof for precisely how small).  Recall that $B_i^{-\kappa}$ is a flow box with the same center as $B_i$, but with width $(1-\kappa)\eta$ in each direction.   Let $F$ be the union of components of $B_0\cap g_{-T}B_1$ that also intersect $B_0^{-\kappa} \cap g_{-T}(B_1^{-\kappa})$.  Clearly $F\subset B_0 \cap g_{-T}B_1$, establishing \eqref{item:subset}. By \Cref{lem:inter-box-eps}, if we choose $c$ large enough so that $c\log g$ is greater than the $\tau_0$ from that lemma, then 
the components of $F$ are all full in the contracting direction (as subsets of $B_0$), and width $\ge e^{-T}\eta$ in the expanding direction, establishing \eqref{item:Fcomp}.

  By effective mixing \Cref{cor:effect-mix-log},  we can choose $c>0$ such that if $T>c\log g$,
  \begin{align}
    \mu(B_0 \cap g_{-T}B_1) = (1+O(\epsilon)) \cdot \mu(B_0)\mu(B_1).   \label{eq:B0-B1-mix}
  \end{align}
    Now can apply effective mixing with smaller flow boxes (covering the region between $B_0$ and $B_0^{-\kappa}$) to see that only a small proportion of the measure (arbitrarily small by taking $\kappa$ small) of $B_0\cap g_{-T}B_1$ is not in $F$.  Thus we get that 
    \begin{align}
    \mu(F) = (1+O(\epsilon)) \cdot \mu(B_0 \cap g_{-T}B_1).  \label{eq:muF}
    \end{align}

    Now arguing as in the proof of \Cref{lem:ave-flow-width}, gives that the components of $B_0 \cap g_{-T}B_1$ are close to equidistributed in the flow direction.  The same statement is true of components of $F$, by the reasoning used in the previous paragraph.    %

    \begin{figure}
     \centering 
     \includegraphics [scale=0.63]{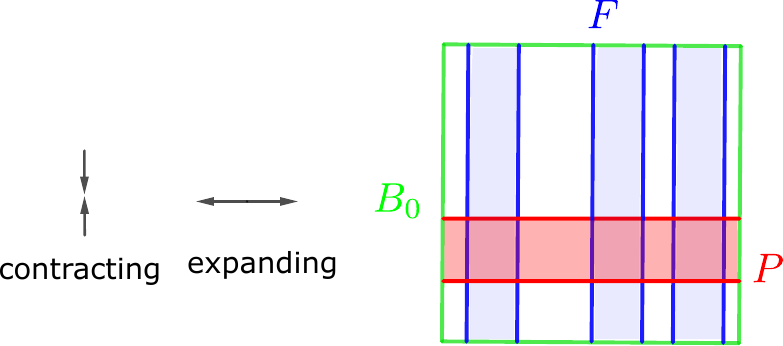}
     \caption{The geometric mechanism corresponding to the ``Markov property" $\mu'(P\cap F)=\mu'(P)\mu'(F)$ in \Cref{lem:band-mix-aux}.  The components of $F$ are full width in the contracting direction, corresponding to a condition on \textit{future} behavior.  The subbox $P$ is full width in the contracting direction, which could come from conditioning on \textit{past} behavior.  The future and past are independent. (The geodesic flow direction has been suppressed in the diagram.) }
     \label{fig:band_mix}
    \end{figure}

    To make the rest of the computation cleaner, we will introduce a measure $\mu'$ on $B_0$, defined by restricting $\mu$ to $B_0$ and then rescaling to get a probability measure. (The advantage of this rescaling is that for two subsets of $B_0$, being ``uncorrelated" corresponds to the measure of their intersection equaling the product of their measures.)  We now consider the geometry of the intersection of $P$ with a component $\mathcal C$ of $F$. We will first suppose (unrealistically) that $P$ and $\mathcal C$ both have full width in the flow direction.  
    Then since $P$ is full width in the expanding direction and $\mathcal C$ is full width in the contracting direction (both with respect to $B_0$), we must have that $\mu'(P\cap \mathcal C) = \mu'(P) \mu'(\mathcal C)$ (see \Cref{fig:band_mix}).  Now, in the previous paragraph we saw that components of $F$ are (close to)  equidistributed in the flow direction.  Hence, we also have that $P$ and $F$ are uncorrelated in the flow direction, and so combined with the previous observation, we get that 
    \begin{align}
        \mu'(P \cap F) = (1+O(\epsilon)) \cdot \mu'(P) \mu'(F) \label{eq:PcapFnorm}.
    \end{align}

    Now we combine the various bounds from above (and use the definition of $\mu'$ in several places):
    \begin{align*}
        \mu(P\cap F) & =  \mu'(P\cap F)\mu(B_0)\\
        & = (1+O(\epsilon)) \cdot \mu'(P) \mu'(F) \mu(B_0)  &&\text{ (by \eqref{eq:PcapFnorm})}\\
        & = (1+O(\epsilon)) \frac{\mu(P)}{\mu(B_0)} \frac{\mu(F)}{\mu(B_0)}  \mu(B_0)\\
        & = (1+O(\epsilon)) \frac{\mu(P)}{\mu(B_0)} \frac{(1+O(\epsilon)) \mu(B_0 \cap g_{-T}B_1)}{\mu(B_0)}  \mu(B_0)  &&\text{ (by \eqref{eq:muF})} \\
        & = (1+O(\epsilon)) \frac{\mu(P)}{\mu(B_0)} \frac{(1+O(\epsilon)) \mu(B_0)\mu(B_1)}{\mu(B_0)}  \mu(B_0)  &&\text{ (by \eqref{eq:B0-B1-mix})} \\
        & = (1+O(\epsilon)) \cdot \mu(P)\mu(B_1),
    \end{align*}
    which is \eqref{item:muPcapF}, and so we are done.  
\end{proof}

\begin{lemma}
    \label{lem:band-mix}
    Fix $\delta, \epsilon >0$.  For any $\eta_0$ small (depending on $\delta, \epsilon$), there exists $c=c(\delta, \epsilon, \eta_0)$ with the following property.  Let $X$ be a $\delta$-expander surface.  
    Let $B_0,B_1\subset T^1X$ be $\eta$ flow boxes, where $\eta_0/1000 \le \eta \le \eta_0$, and such that $B_0^+,B_1^+$ are embedded.  Let $P\subset B_0$ be a subbox that is full width in the expanding direction.  Then for $T\ge c\log g$,
  \begin{align*}
    \mu(P\cap g_{-T}B_1) = (1+O(\epsilon)) \cdot \mu(P) \mu(B_1),
  \end{align*}
  where the implicit constant in the $O()$ is less than $1$.  
\end{lemma}

\begin{proof}
    We first prove the lower bound part of the desired statement.   Apply \Cref{lem:band-mix-aux}.  Using the resulting set $F$ and its properties \eqref{item:subset} and \eqref{item:muPcapF}, we get:
    \begin{align*}
        \mu(P\cap g_{-T} B_1) & \ge \mu(P\cap F) =(1+O(\epsilon)) \cdot \mu(P) \mu(B_1). 
    \end{align*}

    The upper bound is proved in an analogous way (using a version of \Cref{lem:band-mix-aux} that produces a \textit{super} set $F\supset B_0\cap g_{-T} B_1$).

\end{proof}

We now prove a related result where we intersect sets corresponding to a condition from the future (time $t$), and a condition from the \textit{significantly farther} future (time $T\ge t+c\log g$).  The result says that these conditions are very close to independent.  It follows easily from \Cref{lem:band-mix} by applying geodesic flow.

\begin{lemma}
    \label{lem:band-future}
    Fix $\delta, \epsilon >0$.  For any $\eta_0$ small (depending on $\delta, \epsilon$), there exists $c=c(\delta, \epsilon, \eta_0)$ with the following property.  Let $X$ be a $\delta$-expander surface.  
    Let $B_0,B_1\subset T^1X$ be $\eta$ flow boxes, where $\eta_0/1000 \le \eta \le \eta_0$, and such that $B_0^+,B_1^+$ are embedded.    
    Let $t\ge 0$, and let $P\subset B_0$ be a subbox that is full width in the contracting direction, and width $e^{-t}\eta$ in the expanding direction. %
    Then if $T\ge t + c\log g$, 
    \begin{align*}
      \mu(P\cap g_{-T}B_1) = (1+O(\epsilon)) \cdot \mu(P) \mu(B_1),
    \end{align*}
    where the implicit constant in the $O()$ is less than $1$.  
\end{lemma}

\begin{figure}
     \centering 
     \includegraphics [scale=0.63]{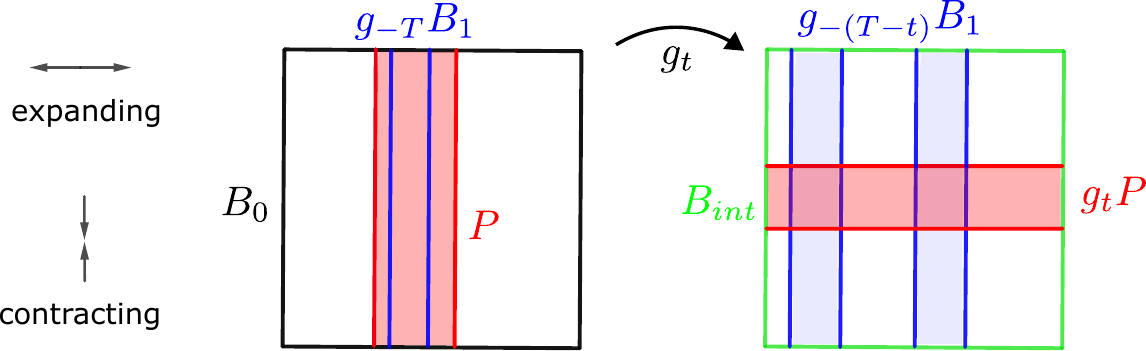}
     \caption{\Cref{lem:band-future} and its proof.}
     \label{fig:band_future}
\end{figure}

  \begin{proof}

    Note that $g_t(P)$ has expanding width $\eta$, contracting width $e^{-t}\eta$ (and flow direction width is unchanged). Take an $\eta$ flow box $B_{int}$ centered at the center of $g_t(P)$.  See \Cref{fig:band_future}.  

    Now we apply \Cref{lem:band-mix} with boxes $B_{int},B_1$, subbox $g_t(P)\subset B_{int}$, and time $T-t \ge c\log g$.  We get that
    \begin{align*}
      \mu(g_t(P)\cap g_{-(T-t)} B_1) =  (1+O(\epsilon)) \cdot \mu(g_tP) \mu(B_1), 
    \end{align*}
and then using invariance of measure under geodesic flow $g_{-t}$ gives
\begin{align*}
  \mu(P\cap g_{-T}B_1) = (1+O(\epsilon))  \mu(P) \mu(B_1).  
\end{align*}
  \end{proof}

The next lemma is a variant of the above.  The conclusion is a %
bound on the number of components (rather than measure).  An additional assumption on fullness in the geodesic flow direction is needed.  This will be used in \Cref{lem:no-inter-return} to count closed geodesics.

\begin{lemma}
\label{lem:comp-P-cap-B}
    Fix $\delta, \epsilon >0$.  For any $\eta_0$ small (depending on $\delta, \epsilon$), there exists $c=c(\delta, \epsilon, \eta_0)$ with the following property.  Let $X$ be a $\delta$-expander surface.  
    Let $B_0,B_1\subset T^1X$ be $\eta$ flow boxes, where $\eta_0/1000 \le \eta \le \eta_0$, and such that $B_0^+,B_1^+$ are embedded.    
    Let $t\ge 0$, and let  $P \subset B_0$ be a subbox that is full width in the geodesic flow and contracting directions, and width $e^{-t}\eta$ in the expanding direction.    Then if $T\ge t + c\log g$, the number of components of intersection satisfies 
     \[
    \comp (P \cap g_{-T}B_1) \leq O(1) e^{T-t} \mu(B_1).
    \]
\end{lemma}

\begin{proof}
    We first apply \Cref{lem:band-future}, which gives a bound on the measure of the set $P \cap g_{-T}B_1$.  Then to get from this to a bound on the number of components of the set, we divide by the average measure of each component. To understand the shape of the components (and thus their typical measure), we argue as in proof of \Cref{lem:effect-pgt-upper} (see also \Cref{lem:band-mix-aux} and its proof, since here we are starting with a subbox $P$, rather than the full box $B_0$).

\end{proof}

\subsection{Effective multiple mixing}

We now prove effective multiple mixing for any finite number $k$ of flow boxes.  The result follows from effective mixing and the expansion/contraction (Anosov) property of the geodesic flow.  Note the error term below becomes bad as $k$ increases; we only use the result for small $k$.  

\begin{thm}[Effective multiple mixing]
  \label{thm:effect-mult-mix} 
  Fix $\delta,\epsilon,\eta_0 >0$.  Then there exists some $c=c(\delta,\epsilon,\eta_0)>0$ with the following property. Let $X$ be a $\delta$-expander surface of genus $g$, and $\eta$ such that $\eta_0/1000 \le \eta \le \eta_0$.  Let $B_1,B_2,\ldots \subset T^1X$ be $\eta$ flow boxes such that each $B_i^+$ is embedded. Let $t_1,t_2,\ldots \ge 0$.  Define
  \begin{align*}
    M_k:=\{ v: g_{t_1}v \in B_1, \ldots, g_{t_k}v \in B_k \}.
  \end{align*}
Then if $t_{i}-t_{i-1} \ge c \log g$ for each $i\ge 2$, we have 
\begin{align}
  \mu(B)^k (1 - \epsilon)^k  \le \mu\left( M_k \right) \le \mu(B)^k (1+ \epsilon)^k. \label{eq:mult-mix}
\end{align}
(Here $B$ is any of the flow boxes, which all have the same $\mu$ measure.) 
\end{thm}

\begin{figure}[ht]
 \centering 
 \includegraphics [scale=0.33]{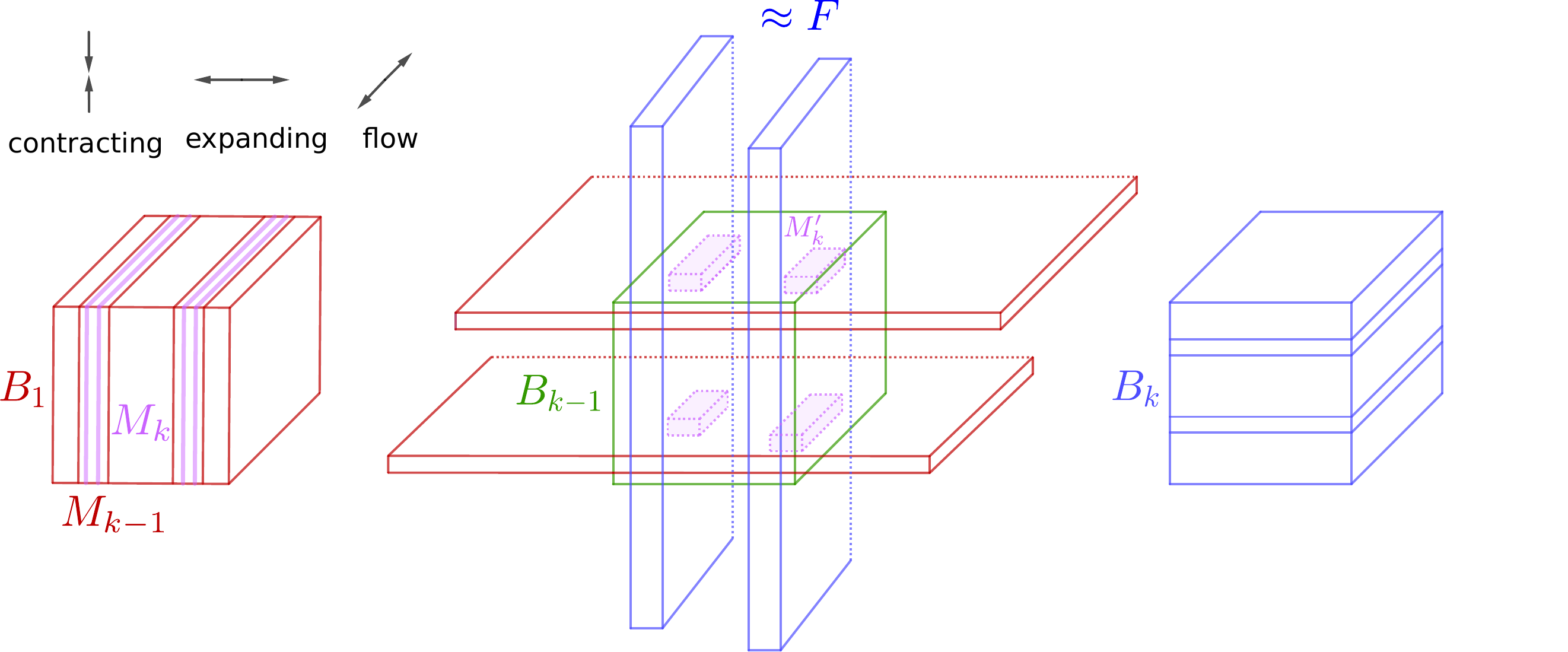}
 \caption{The geometric mechanism that allows one to prove (effective) \emph{multiple} mixing, using mixing.  In the middle box, the red components, coming from a condition on the past, are ``perpendicular'' to the blue components, which come from a condition on the future (compare this middle box to \Cref{fig:band_mix}).  This is the analog in the hyperbolic dynamics setting of the Markov property for random walks. }
 \label{fig:nfold_mix}
\end{figure}

 We will first deal with the upper bound in the statement of \Cref{thm:effect-mult-mix}.  

\begin{lemma}
  \label{lem:mix-lower}
Suppose the same setup as in \Cref{thm:effect-mult-mix} and that $t_1=0$. Then there exist sets $\bar M_1, \bar M_2, \ldots \subset B_1 \subset T^1X$ such that for each $k$
\begin{enumerate}[(i)]
\item $\bar M_k \subset M_k$, \label{item:contain}
\item every component of $\bar M_k$ has full width in the contracting direction (relative to $B_1$) and width $e^{-t_k}\eta$ in the expanding direction,  \label{item:comp-shape} 
\item $\mu(\bar M_k) \ge  \mu(B)^k (1+\epsilon)^k$. \label{item:measure}
\end{enumerate}
\end{lemma}

\begin{proof}
    We will prove the result by induction on $k$.  For $k=1$, we take $\bar M_k :=M_k$, and the properties trivially hold.
    
    So assume we have already constructed $\bar M_{k-1}$.  

    Let $F$ be the set given by applying \Cref{lem:band-mix-aux} with the $\eta$ flow boxes $B_{k-1},B_k$ and $T=t_k-t_{k-1}.$   These sets, as well as those defined below, are illustrated in \Cref{fig:nfold_mix}.

Now let
$$\bar M_k':=\left( g_{t_{k-1}}\bar M_{k-1} \right) \cap F,$$
$$\bar M_k : = g_{-t_{k-1}}\bar M_k'. $$ %

By the inductive hypothesis part \eqref{item:contain}, we have $\bar M_{k-1} \subset M_{k-1}$.   By \Cref{lem:band-mix-aux} \eqref{item:contain}, we have $F\subset B_{k-1} \cap g_{-(t_k-t_{k-1})}B_k\subset g_{-(t_k-t_{k-1})}B_k$,
hence $g_{-t_{k-1}}F \subset g_{-t_k}B_k$.  Then using these observations, and from the definitions immediately above, we see 
\begin{align*}
    \bar M_k = \bar M_{k-1} \cap g_{-t_{k-1}}F \subset M_{k-1} \cap g_{-t_k}B_k = M_k,
\end{align*}
giving \eqref{item:contain}.  

By the inductive hypothesis part \eqref{item:comp-shape}, we get that each component $\mathcal C$ of $g_{t_{k-1}}\bar M_{k-1}$ has full width in expanding direction, and width
$\ge e^{-t_{k-1}}\eta$ in contracting direction.  By \Cref{lem:band-mix-aux} \eqref{item:Fcomp}, each component of $F$ has width $\ge e^{-(t_k-t_{k-1})}$ in the expanding direction and full width in the contracting direction. Hence components of the intersection $g_{t_{k-1}}\bar M_{k-1}\cap F$ have width $\ge e^{-(t_k-t_{k-1})}\eta$ in the expanding direction, and width $\ge e^{-t_{k-1}}\eta$ in the contracting direction (the situation is similar to \Cref{fig:band_mix}).
Applying $g_{-t_{k-1}}$ gives the desired statement about the structure of each component of $\bar M_k=\bar M_{k-1} \cap g_{-t_{k-1}}F$, 
so we have established \eqref{item:comp-shape}.  

Now for each such component $\mathcal C$, \Cref{lem:band-mix-aux} \eqref{item:muPcapF} with $P=\mathcal C$, yields
\begin{align*}
    \mu(\mathcal C \cap F) = (1+O(\epsilon)) \cdot \mu(\mathcal C) \mu(B_k),
\end{align*}
where the implicit constant in the $O(\cdot)$ is less than $1$ (in the above, and also in the below).  

So
\begin{align*}
  \mu(\bar M_k) & = \mu(\bar M_k') \\
                &=\left( g_{t_{k-1}}\bar M_{k-1} \right) \cap F \\
                &=\sum_{\mathcal C} \mu (\mathcal C \cap F) \\
                &=\sum_{\mathcal C} (1+O(\epsilon)) \cdot \mu(\mathcal C) \mu(B_k) \\
                &=(1+O(\epsilon)) \cdot \mu(B_k)\sum_{\mathcal C}  \mu(\mathcal C) \\
                &=(1+O(\epsilon)) \cdot \mu(B_k)\mu \left( g_{t_{k-1}}\bar M_{k-1} \right)  \\
                &=(1+O(\epsilon)) \cdot \mu(B_k)\mu \left(\bar M_{k-1} \right)  \\
                &\ge (1+O(\epsilon)) \cdot \mu(B_k)\left( \mu(B)^{k-1}(1+\epsilon)^{k-1} \right)  &&\text{ \ (by induction, using \eqref{item:measure})}\\ 
                &= \mu(B)^k (1+O(\epsilon))^k,
\end{align*}
which establishes the desired measure bound \eqref{item:measure}.  
\end{proof}

\begin{proof}[Proof of \Cref{thm:effect-mult-mix}]
  The lower bound follows from \Cref{lem:mix-lower} (that assumes $t_1=0$, but we can reduce to this case by applying $g_{-t_1}$, which preserves the measure).  The proof of the upper bound is analogous, using an upper bound analog of \Cref{lem:mix-lower}.  
\end{proof}

We will also use the below variant of effective multiple mixing where each of the target sets can be either a flow box or its \textit{complement}. 

\begin{lemma}
  \label{lem:mix-complement}
  Fix $\delta,\epsilon,\eta_0 >0$.  There exists some $c=c(\delta,\epsilon,\eta_0)>0$ with the following property. Let $X$ be a $\delta$-expander surface of genus $g$, and $\eta$ such that $\eta_0/1000 \le \eta \le \eta_0$.
  Let $B_1,\ldots,B_k\subset T^1X$ be $\eta$ flow boxes, such that each $B_i^+$ is embedded.  For each $i$, let $E_i$ be either $B_i$, or the complement in $T^1X$ of $B_i$, and let 
  \[
   m_i = \left \{ 
   \begin{array}{lll}
   (1+\epsilon)\mu(B_i) & \text{ if } & E_i = B_i, \\ 
   1 - (1-\epsilon)\mu(B_i) & \text{ if } & E_i = T^1X - B_i .\end{array}
   \right .
  \]
  Then 
\begin{align*}
  \mu\left(\{ v: g_{t_i}v \in E_i \text{, for } i=1,\ldots,k \}\right) \le m_1 \cdots m_k
\end{align*}
for any $t_1<\cdots<t_k$ with $t_i-t_{i-1}\ge c \log g$ for each $i\ge 2$.
\end{lemma}

\begin{proof}
  The proof is similar to that of \Cref{thm:effect-mult-mix}. To ensure the complementary components are thick enough, we use slightly contracted flow boxes.
\end{proof}

\section{Simple geodesics}
\label{sec:simp-conj}

In this section we will get an upper bound on simple geodesics $N_{simp}(X,L)$, which will allow us to prove \Cref{thm:simplicity}.  

\subsection{An analogous probability problem}
\label{sec:prob-problem}
In this subsection we solve a problem in basic discrete probability that will serve as an analog of the problem of bounding the number of simple geodesics on expander hyperbolic surfaces.  Although this subsection is, strictly speaking, purely motivational, the reader will find it useful to study before moving on to the hyperbolic setting, where the technical issues are much more formidable.  

The heuristic for \Cref{thm:simplicity} derives from analysis of the ``birthday paradox''.   This involves picking $k$ objects from a collection of $n$, with replacement.  The question is: how large does $k$ need to be to guarantee the chance of getting at least one object more than once is high?  The transition occurs near $k=\sqrt{n}$.

In our situation we have to discretize our continuous space.  We will want the relevant subsets (which will be flow boxes) to be disjoint. 
They will not actually cover the whole space, but rather some definite fraction of it (the ``good'' objects below are the ones corresponding to these disjoint subsets).  In order to prove a geodesic self-intersects, it is \emph{not} enough to show that it comes back close to where it has been previously.  It will be enough to show that it comes back close, \emph{and at a definite angle} (i.e. ``transversely'').  

We incorporate these two differences from the ``birthday'' situation into a modified probability problem, which we  then solve.  Our proof of the \Cref{thm:simplicity} will then be an analog of this, but in the context of hyperbolic dynamics, which, although deterministic, behaves much like a random system.

\begin{prop}
  \label{prop:prob-problem}
  Fix $\alpha$ with $0<\alpha <1/3$.  Let $x_1,\ldots,x_{\ell}$ be samples from a collection $S$ of $n$ distinct objects.  The samples are chosen independently, uniformly at random, and with replacement.  We are additionally given a subset $G\subset S$, the ``good'' objects, which has size at least $\alpha \cdot n$, together with an injective map $T: G\to S$, the ``transverse object'' map.   

  Then
  \begin{align*}
  p: = \P [\not\exists (i,j) \text{ with } x_i\in G \text{ and } x_j = T(x_i)]     \to 0 
  \end{align*}
as $n\to\infty$, provided that $\ell\succ \sqrt{n}$.   
\end{prop}
Note that the map $T: G \to S$ does not have to map ``good" objects to other ``good" objects, i.e. $x_i \in G$ need not imply $T(x_i) \in G$. All we require is that the map $T: G \to S$ is injective.
\begin{proof}
 We wish to bound the probability that, by some time $l$, we never choose an object $g \in G$ and its ``transverse" $T(g)$. Our strategy is to consider two cases.  We look at the number of distinct good objects that we get among the earlier of the choices of the $x_i$.  The first case is that this number is relatively small; we show that the probability of this event is low (by breaking into two further sub-cases).  The second case is when this number is relatively large; on this event we show that it is then unlikely that none of the later choices of objects hits a transverse to one of the good earlier choices.

 Now to begin the proof, we set $k:=\min(\ell,n^{2/3})$. 
  \begin{enumerate}
  \item Let
    \begin{align*}
          r := \P\left[\#\left(\{x_1,\ldots,x_{\lfloor (1-\alpha)k \rfloor}\} \cap G\right) \le \alpha k/4 \right].
    \end{align*}
    This is the probability that few good distinct objects are hit among the early choices.
  \item Let
    \begin{align*}
      q :=  \P [ &\#\left(\{x_1,\ldots,x_{\lfloor (1-\alpha)k \rfloor}\} \cap G\right) > \alpha k/4  \\
                 & \text{ and } x_{\lfloor (1-\alpha)k \rfloor+1},\ldots, x_k \not\in T(\{x_1,\ldots, x_{\lfloor (1-\alpha)k \rfloor}\} \cap G)  ]. 
    \end{align*}
This is the probability that many good distinct objects are hit among the early choices, and none of the later choices hits a transverse to one of the good earlier choices.  
  \end{enumerate}

    The event associated with $p$ is contained in the union of the events associated with $r$ and $q$.  Hence, by the union bound, we have 
    \begin{align}
          p\le r+q. \label{eq:p-union}
    \end{align}
    So our goal now is to show $r$ and $q$ each tend to $0$. 

  To bound $r$, we will define further probabilities based on two cases:
  \begin{enumerate}[(A)]
  \item Let
    \begin{align*}
      r^f := \P\left[ \#\{ i: 1\le i \le (1-\alpha)k, \ x_i \in G\} \le \alpha k/2\right],
    \end{align*}
    the probability that few of the early choices hit good objects.  
  \item Let
    \begin{align*}
      r^m := \P[  &\#\{ i: 1\le i \le (1-\alpha)k, \ x_i \in G\} > \alpha k/2\\
                  &\text{ and } \#\left(\{x_1,\ldots,x_{\lfloor (1-\alpha)k \rfloor}\} \cap G\right) \le \alpha k/4 ],
    \end{align*}
    the probability that many of the early choices hit good objects, but among these there are not many distinct objects hit. 
  \end{enumerate}

  Note that, by a union bound, 
  \begin{align}
      r\le r^f + r^m.  \label{eq:r-union}
  \end{align}

 \medskip
 \noindent \underline{Bounding $r^f$}:
 We use the second moment method.  Let $X_i$ be the indicator random variable of the event that $x_i\in G$.  Let $X=\sum_{i\le (1-\alpha)k}X_i$, so $r^f= \P[X\le \alpha k/2]$.  We first compute the expected value of $X$.  Note that $\E[X_i]=\P[x_i \in G] = \alpha$.  So
 \begin{align}
   \E[X] =  \sum_{i\le(1-\alpha)k} \E[X_i] = \alpha (1-\alpha)k > (2/3)\alpha k, \label{eq:EX}
 \end{align}
 using the assumption $\alpha<1/3$.

Now we compute the second moment, using independence of the $X_i$:
\begin{align*}
  \var(X)&= \sum_{i,j} \cov(X_i,X_j) = \sum_i \var(X_i) = \sum_i (\alpha - \alpha^2)=(\alpha-\alpha^2)(1-\alpha)k\\
         & \le \alpha k.
\end{align*}

Then using the above and \eqref{eq:EX} with Chebyshev's inequality, we get:
\begin{align*}
  r^f &=\P[X\le \alpha k/2] \le \P[|X-\E X|>\alpha k/6] \\%
  &\le \frac{\var(X)}{(\alpha k/6)^2} \le \frac{\alpha k}{(\alpha k/6)^2} = \frac{36}{\alpha } \frac 1 k
\end{align*}

and hence $r^f\to 0$ as $k\to \infty$ (which must happen when $n\to\infty$, since $k\succ \sqrt n$).  

  \medskip

   \medskip
 \noindent \underline{Bounding $r^m$}: We use the first moment method.  

Let $Y_{i,j}$ be the indicator of the event $x_i=x_j$, and $Y=\sum_{i<j}Y_{i,j}$.  Note that on the event defining $r^m$, we must have $Y\ge \alpha k/2-\alpha k/4 = \alpha k/2$, since there are least this many values of $j$ such that $x_j\in G$ and $x_i=x_j$ for some value of $i<j$.  Then by Markov's inequality, we get
\begin{align*}
  r^m \le P[Y\ge \alpha k/2] \le \frac{\E[Y]}{\alpha k/2} \le \frac{k^2/n}{\alpha k/2} = \frac{2}{\alpha} \frac{k}{n}, 
\end{align*}
which, since $k \le n^{2/3}$ , goes to $0$ as $n\to\infty$.  
 
 \medskip

    \medskip
    \noindent \underline{Bounding $q$}:  Consider the conditional probability 
    \begin{align*}
      q' :=  \P[ &x_{\lfloor (1-\alpha)k \rfloor+1},\ldots, x_k \not\in T(\{x_1,\ldots, x_{\lfloor (1-\alpha)k \rfloor}\} \cap G)  \\
      &| \  \#\left(\{x_1,\ldots,x_{\lfloor (1-\alpha)k \rfloor}\} \cap G\right) > \alpha k/4 ].
    \end{align*}
    Note that $q\le q'$, so for our purposes it will suffice to bound $q'$.  Since $T$ is injective, the condition in the probability expression above  
    implies that $T\left (\{x_1,\ldots, x_{\lfloor (1-\alpha)k \rfloor}\} \cap G \right)$ has at least $\alpha k/4$ elements.  So, ala the birthday paradox, we compute the probability that $x_{\lfloor (1-\alpha)k \rfloor+1},\ldots, x_k$ all avoid these $\alpha k/4$ objects (notice that these later choices are independent of those involved in the condition), giving
    \begin{align*}
      q' &\le \left(1-\frac{\alpha k/4}{n}\right)^{\alpha k-1} \\
      &\approx \exp\left(-\frac{\alpha k/4}{n}\right)^{\alpha k-1} \\
      & = \exp(-\Omega(k^2/n)),
    \end{align*}
    (where we have used that for $n$ large, $1-\frac{\alpha k/4}{n} \ge 0$, since $k\le n^{2/3}$).
  The last term goes to $0$ as $n\to\infty$, and hence so does $q$. 

\medskip

    \medskip
    \noindent \underline{Completing the proof}:
    Putting together \eqref{eq:p-union}, \eqref{eq:r-union}, and the results from the above three cases, we get that
    $$p=r+q\le (r^f+r^m) + q \to 0.$$

\end{proof}

\subsection{Flow boxes for proof of \Cref{thm:simplicity}}
\label{sec:setup}

We now begin the process of transferring the proof in \Cref{sec:prob-problem} to the setting of expander hyperbolic surfaces.  

\paragraph{Properties of the flow boxes.}

  Recall from \Cref{sec:effect-mix} the various definitions associated with flow boxes.   In what follows, we will find a collection of disjoint flow boxes that cover a definite proportion of the surface.  Additionally, each box $B$ is paired with a ``transverse'' box $\hat B$ such that a geodesic crossing through $B$ and $\hat B$ is guaranteed to self-intersect transversely.

  For each $v\in T^1X$, we define the rotated vector $\hat v := r_{\pi/2}v$.  If $B=B(v)$ is an $\eta$ flow box centered at $v$, we define the \emph{transverse box} $\hat B$ to be the $\eta$ flow box centered at $\hat v$.

  \begin{prop}
    \label{prop:flow-boxes}
    Fix $s_0>0$.  Given $\eta_0$ small (depending on $s_0$), there exists $\alpha=\alpha(s_0,\eta_0)>0$ with the following property.  For any hyperbolic surface $X$ with systole at least $s_0$, there exists $v_1,\ldots,v_{2g-2} \in T^1X$ such that for all $\eta$ with $\eta_0/1000 \le \eta \le \eta_0$, 
    the $\eta$ flow boxes $B_i:=B(v_i)$ and $\hat B_i:=B(\hat v_i)$ are embedded in $T^1X$ and satisfy 

    \begin{enumerate}[(i)]
    \item  \label{item:enough-measure}$\mu\left(\cup_i B_i\right) > \alpha,$
      
    \item  \label{item:pair-disj} the $B_1,\ldots,B_{2g-2},\hat B_1,\ldots, \hat B_{2g-2}$ are pairwise disjoint,

          \item iff $t_1,t_2\in \R$ and $v$ is such that $g_{t_1}v\in B_i$ and $g_{t_2}v\in \hat B_i$, then the geodesic $t\mapsto g_tv$ has a transverse self-intersection, \label{item:self-inter} %
      
    \item  \label{item:full-box-sep}``full box separated'' i.e. for any $w\in T^1X-\bigcup_{i=1}^{2g-2} \hat B_i^+$, the  $\eta$ flow box $B(w)$ satisfies $B(w) \subset T^1X-\bigcup_{i=1}^{2g-2} \hat B_i$. (This technical condition, which will be used to handle edge effect components, means that there aren't thin regions in between our transverse flow boxes.)
    \end{enumerate}
  \end{prop}
  
   \begin{figure}[h!]
    \centering
    \includegraphics[scale=1.2]{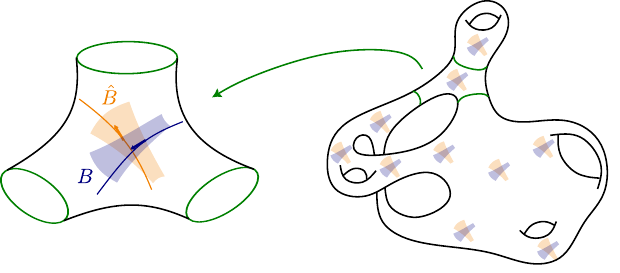}
    \caption{Transverse flow boxes guaranteeing a self-intersection. In (the unit tangent bundle over) every pair of pants, we can fit a pair of transverse flow boxes $B,\hat B$ of definite size.  Any geodesic that lands in both $B$ and $\hat B$ must have a self-intersection.  Thus a simple geodesic cannot hit both boxes in any such pair.  We use this condition to bound the number of simple closed geodesics of length at most $L$. }
    \label{fig:transverse}
  \end{figure}

  The proof of this proposition will depend on the following observation giving that any surface has many points where the injectivity radius is bounded below by a uniform constant.

  \begin{lemma}
  \label{lem:embedded-disc}
  There is a universal constant $r_0>0$ so that any hyperbolic pair of pants $P$ with geodesic boundary contains an embedded disc of radius $r_0$.
\end{lemma}
  \begin{proof}
    We first decompose our pair of pants $P$ into two isometric right-angled hexagons. Let $H$ be one of these hexagons. By the Gauss-Bonnet formula, the area of a $H$ is $\pi$. We can cut $H$ into the union of 4 triangles.  See \Cref{fig:hexagon}. One of these triangles, denoted $T$, must have area at least $\pi/4$.  Every hyperbolic triangle of area at least $A$ contains an embedded disc of radius $r(A)$, for some function $r$.  This follows from compactness of the set of isometry types of such triangles (allowing ideal vertices), since hyperbolic triangles are determined up to isometry by their angles, and the area bound implies the angle sum is bounded from above away from $\pi$.  So we take $r_0=r(\pi/4)$.  

    (Alternatively, we can prove the lemma by using the fact that any hyperbolic pair of pants can be cut into $2$ ideal triangles; we take an embedded disc in one of these triangles.) 

  \end{proof}

  \begin{figure}%
    \centering
    \includegraphics{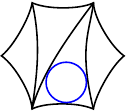}
    \caption{Every right-angled hyperbolic hexagon contains a disc of definite area.}
    \label{fig:hexagon}
  \end{figure}
  
  With this, we can prove the proposition.
  \begin{proof}[Proof of Proposition \ref{prop:flow-boxes}]
  Let $X$ be any hyperbolic surface of genus $g$. Take a pants decomposition $P_1, \dots, P_{2g-2}$ of $X$. By \Cref{lem:embedded-disc}, there exists an $r_0 > 0$ independent of $X$ or the choice of its pants decomposition, so that we can fit a disc $D_i$ of radius $r_0$ inside each pair of pants $P_i$. As the pairs of pants have disjoint interiors, these discs will be pairwise disjoint, as well.
  
  We will find our collection of flow boxes $B_1, \dots, B_{2g-2}, \hat B_1, \dots, \hat B_{2g-2}$ in the unit tangent bundle above these discs. Let $\pi: T^1X \to X$ be the usual projection. We observe that there is some constant $\eta_0$ so that for all $\eta < \eta_0$, if $B$ is an $\eta$ flow box, and $\pi(B) \subset D_i$, then $B$ is embedded in $T^1 X$. In fact, for all $\eta$ small enough, a lift of $B$ will be embedded in the universal cover $T^1 \mathbb{H}^2$, and since $D_i$ is embedded in $X$, then $B$ will be embedded in $T^1 X$.
  
  Fix such an $\eta_0$. Make it smaller if necessary, so that $\eta_0 < r_0/100$ (and note that we might retroactively make $\eta_0$ smaller again later in the proof, but when we do this, it will not depend on geometric or topological features other than the systole $s_0$). We will show that the proposition holds for any $\eta$ with $\eta_0/1000 \le \eta \le \eta_0$.

  Let $p_i$ be the center of $D_i$ and take $v_i$ to be any vector in $T^1_{p_i}X$. Taking $\eta_0$ smaller if necessary, we get that the $\eta$ flow box $B_i: = B(v_i)$ is embedded in $T^1 X$, for any $\eta<\eta_0$.  

    Note that if we fix $\eta_0$, and if $\eta_0/1000 \le \eta \le \eta_0$, then there will be some $\alpha$ depending on $\eta_0$ so that
   \[
    \mu(B_i) > \alpha \mu(T^1 P_i)
   \]
   where $T^1P_i = \pi^{-1}P_i$ is the unit tangent bundle of $P_i$ inside $X$.
   As there is a box above each pair of pants, we see that 
   \[
    \mu \left (\bigcup_{i=1}^{2g-2} B_i \right ) > \alpha.
   \]
   Thus, our collection of boxes $B_1, \dots, B_{2g-2}$ satisfies part (\ref{item:enough-measure}).

   Again making $\eta_0$ smaller if necessary, we get that $B_i$ and $\hat B_i:=B_i(\hat v_i)$ are disjoint, since $v_i$ and $\hat v_i:=r_{\pi/2} v_i$ are uniformly separated. %
  This establishes (\ref{item:pair-disj}).  

  Next, if $w \in B_i$ and $\hat w \in \hat B_i$, then $w$ and $\hat w$ are very close to
  to $v_i$ and $\hat v_i$, respectively.  The complete geodesics $\gamma_{v_i}$ and $\gamma_{\hat v_i}$ tangent to $v_i$ and $\hat v_i$, respectively, must intersect at an angle of $\pi/2$.  So then the geodesics $\gamma_w$ and $\gamma_{\hat w}$ tangent to $w$ and $\hat w$ must intersect at an angle close to $\pi/2$, and in particular they intersect transversely for $\eta_0$ small enough, establishing (\ref{item:self-inter}).  
  
  Lastly, we will show that the flow boxes are ``full box separated''. Recall that $\hat B_i^+$ is the $3\eta$ flow box centered around $\hat v_i$.  We can arrange that $\eta_0$ is sufficiently small so that each box $\hat B_i^+$ is still embedded in $T^1 X$.  Let $w \in T^1X - \bigcup_{i=1}^{2g-2} \hat B_i^+$.  We will show that $B(w)$ is disjoint from $\bigcup_{i=1}^{2g-2} \hat B_i$. In fact, let $w' \in B(w)$. Suppose $w' \in \hat B_i$ for some $i$. If the flow boxes were, in fact, Euclidean boxes, then the fact that $w' \in B(w) \cap \hat B_i$ and that these are both $\eta$ flow boxes would mean that $w$ was in the $2\eta$ flow box around $\hat v_i$. But as $\eta$ tends to 0, the flow boxes get close to Euclidean boxes. So choosing $\eta_0$ smaller, if needed, $w' \in B(w) \cap \hat B_i$ implies that $w$ is in the $3\eta$ flow box about $\hat v_i$. In other words, $w \in \hat B_i^+$, which contradicts our assumptions. Thus, our collection of flow boxes satisfies condition (\ref{item:full-box-sep}).
\end{proof}

\subsection{Bound on measure of simple directions}
\label{sec:bound-S}

Fix an $\eta$ flow box $B\subset T^1X$.  We will focus henceforth on bounding from above the number of simple closed geodesics that intersect $B$.  

Let  
  \begin{align*}
    S:=\{v\in B: t\mapsto g_tv \text{ does not self-intersect}\},
  \end{align*}
  i.e. the set of all vectors in $B$ tangent to (not necessarily closed) simple geodesics.  
We will first prove a measure bound on related sets $\bar S_k$ of tangent vectors that generate geodesics lacking particular types of self-intersection. \footnote{Note that the set $S$ itself has measure $0$.  This follows from \cite{bs85}, and can also be deduced from results in this paper.}  Every vector in the complement of $\bar S_k$ corresponds to an arc with a ``robust self-intersection'', which must occur before a certain time, measured by the parameter $k$.  Later we will add the condition that the arcs return to $B$ at time $L$, since ultimately we are interested in counting \textit{closed} geodesics.   %

  The existence and properties of this set $\bar S_k$ are the content of the next lemma.  This lemma is roughly the analog of \Cref{prop:prob-problem}, and its proof will be parallel.

\begin{lemma}
  \label{lem:no-inter}
  Fix $\delta, s_0, \epsilon>0$.  %
  For any $\eta_0$ sufficiently small, there exists $\alpha>0$ and $c_0$ such that for any $c\ge c_0$ the following holds.  Let $X$ be a $\delta$-expander surface of genus $g$ with systole at least $s_0$.   %
  Let $\eta$ be such that $\eta_0/1000 \le \eta \le \eta_0$, and let $B\subset T^1X$ be an $\eta$ flow box.  
  Let $k$ be any positive integer.  
  Then there exists a set $\bar S_k \subset B$ such that
  \begin{enumerate}[(i)]
  \item $\bar S_k\supset S$, \label{item:barSsimp}
  \item \label{item:barSrobust} For any $u\in B-\bar S_k$, there exists some $\eta$ flow box $B_0$ and $t_1,t_2\in [0,kc\log g]$ such that $g_{t_1}u \in B_0^{++ }$ and $g_{t_2}u \in \hat B^-_0$
    (which implies the geodesic segment through $u$ has a ``robust self-intersection''),
  \item  \label{item:barSmeas}
      $\mu(\bar S_k) \le  \mu(B) \left(\epsilon + O(1/k) + O(k/g) + \left\{1-\Omega(1) k\mu(B)\right\}^{\alpha k}\right),$
    
     where $\{x\}$ denotes $\max(x,0)$.  
    Here the implicit constants in the various $O()$ and $\Omega()$ depend only on $\delta,s_0,\epsilon$.

       \end{enumerate}

\end{lemma}

In the remaining parts of this subsection, we will prepare for the proof of the above lemma by constructing the set $\bar S_k$ as a union of sets $R=R^f\sqcup R^m$ and $Q _k$  (paralleling the cases in the proof of \Cref{prop:prob-problem}, whose probabilities are given by $r\le r^f+r^m$ and $q$).  We prove measure bounds along the way, in \Cref{lem:Rf}, \Cref{lem:Rm}, and \Cref{lem:Qjstruc}.  The ingredients will then be assembled into the proof of the lemma in \Cref{sec:proof-no-inter}.

\paragraph{Collection of flow boxes.}
The sets $R$ and $Q_k$ will be defined in terms of different behaviors with regard to a set of flow boxes 
$$\{B_\nu\}_{\nu=1}^{2g-2} = \{B_1,\ldots,B_{2g-2}\},$$ that are the analogs of the ``good" objects in \Cref{prop:prob-problem}.  We take them to be $\eta$ flow boxes in $T^1X$ given by \Cref{prop:flow-boxes}, where $\eta=\eta_0/27$, for the $\eta_0$ in the statement of \Cref{lem:no-inter}.  

\paragraph{Discrete set of times.}
   Instead of considering all times $t$, we pick a sequence of times at which to sample the geodesic segments, separated enough that we can apply mixing of the geodesic flow. 
   
   Let $t_1,t_2,\ldots $ be such that $t_{i+1}-t_i=c\log g$ (where $c$ will be chosen large later, as mentioned in the lemmas below; it is related to mixing time for geodesic flow, and will depend on an error parameter $\epsilon$).  The definitions of the sets $R$ and $Q_k$ will depend on the $t_i$.  The earlier and later times among these will play somewhat different roles when we study self-intersections (just as the earlier and later choices play different roles in proof of \Cref{prop:prob-problem}).  

\subsubsection{$R$: Vectors that hit few distinct boxes $\{B_\nu\}$}
\label{sec:R}

 Our strategy to produce self-intersections is to witness the geodesic segment intersecting some flow box from our collection, and its transverse box (defined in \Cref{sec:setup}).  To apply the birthday paradox reasoning, we will need the geodesic to hit many different flow boxes among the $\{B_\nu\}$.  The set $R$ defined below consists of vectors for which the associated geodesic does \emph{not} have this property; we will show, with a fairly simple argument, that the measure here is small.  
 
For any positive integer $k$, let 
    \begin{align*}
      R = R(k) := \Big\{ v\in B: \# \{\nu: \exists j\le (1-\alpha)k \text{ such that } g_{t_j}v \in B_{\nu}\} < \alpha k/4 \Big\},
    \end{align*}
    i.e. $R$ consists of vectors that do \emph{not} hit at least $\alpha 4/ k$ distinct elements of $\{B_\nu\}$ among times $t_1,\ldots, t_{\lfloor (1-\alpha)k\rfloor}$.

\paragraph{Decomposition $R = R^f\sqcup R^m$.}
We define a set of starting vectors that visit our collection of flow boxes at relatively few times: 
  \begin{align*}
    R^f= R^f(k):= \left\{ v \in B : \#\{i : i\le (1-\alpha)k,  \  \ g_{t_i}v \in \cup_\nu B_\nu\} \le (\alpha/2) k    \right\}.  
  \end{align*}

The complementary set consists of vectors that hit flow boxes at many times, but that still hit relatively few \emph{distinct} flow boxes:
\begin{align*}
  R^m:=R-R^f.  
\end{align*}

We now analyze the two sets $R^f,R^m$ separately.  
\renewcommand\qedsymbol{$\square$} 
  \begin{sublemma}
    \label{lem:Rf}
    For any $\epsilon>0$, there exists $c_0$ such that for any $c\ge c_0$,  
 \[
  \mu(R^f) \le \mu(B)\left( \epsilon + O(1/k)\right).
 \]
\end{sublemma}

\begin{proof}
 For each $i$ with $i\le (1-\alpha)k$, and each $\nu$, define $X_{i \nu}:B\to\R$ by %
 \[
  X_{i \nu}(v) = \begin{cases}
  \begin{matrix} 1 & \text{ if } g_{t_i}v \in B_\nu\\ 
    0 & \text{otherwise.} \end{matrix}
\end{cases}
\]
Since the boxes $B_\nu$ are all disjoint, we have that \[X_i := \sum_{\nu}X_{i \nu}\] is either 0 or 1 for all $i$, and determines whether $g_{t_i}v$ hits any of the boxes $B_1, \dots, B_{2g-2}$ at time $t_i$. Thus, if we set 
\[
 X = \sum_{i\le (1-\alpha)k} X_i,
\]
then $X(v)$ is the number of times $i$ for which $g_{t_i}v$ hits boxes $B_1, \dots, B_{2g-2}$. So we wish to show 
\[
\mu(\left\{ v: X(v) \leq (\alpha/2)k \right\}) = \mu(B)\left( O(\epsilon) + O(1/k)\right).
 \]

 We will use the second moment method. For this, we will first estimate $\frac{1}{\mu(B)} \int_BXd\mu$ and $\frac{1}{\mu(B)} \int_BX^2d\mu$. To estimate $\int_BX $, we just need to estimate $\int_BX_{i\nu}$ for each $i, \nu$.
 We have
\begin{align*}
  \int_BX_{i \nu} & = \mu(\{v\in B: g_{t_i}v \in B_\nu\}) \\
                       &= \mu(B \cap g_{-t_i} (B_\nu)) \\
   & = \mu(B_\nu)^2 \left (1 + O(\epsilon) \right ),
\end{align*}
for $c$ sufficiently large, where the last line is due to effective mixing (\Cref{cor:effect-mix-log}).  

Then summing gives
\begin{align}
  \frac{1}{\mu(B)} \int_BX  &= \frac{1}{\mu(B)} \sum_{i\le (1-\alpha)k}\sum_\nu \int_BX_{i\nu} \\
  & =\sum_{i\le (1-\alpha)k}\sum_\nu \mu(B_\nu) \left (1 + O(\epsilon) \right ) \\
  & = (1-\alpha)k\alpha (1+O(\epsilon)). \label{eq:first-mom}
\end{align}

Now we must estimate $\frac{1}{\mu(B)} \int_BX^2$. Writing $X = \sum_i X_i$, we see that 
\[
 X^2 = \sum_i X_i + \sum_{i \neq j} X_i X_j 
\]
where we use that $X_i$ is always either 0 or 1, and so $(X_i)^2 = X_i$. To estimate $\int_BX_i X_j$, we use that $X_i = \sum_\nu X_{i \nu}$, and so $X_i X_j = \sum_{\nu, \nu'} X_{i \nu} X_{j \nu'}$. Now for $i\ne j$
\begin{align*}
 \int_B X_{i \nu} X_{j \nu'}& = \mu(\{v\in B: g_{t_i} v \in B_\nu, \ g_{t_j}v \in B_{\nu'}\}) \\
  & = \mu(B\cap g_{-t_i}B_\nu \cap g_{-t_j} B_{\nu'}) \\
  & = \mu(B) \mu(B_\nu) \mu(B_{\nu'})\left (  1+ O(\epsilon) \right)^3,
\end{align*}
for $c$ sufficiently large, where the last line comes from effective 3-mixing, i.e. the $k=3$ case of \Cref{thm:effect-mult-mix}.

For $\epsilon$ small enough, $\left (  1+ O(\epsilon) \right)^3 = \left (  1+ O(\epsilon) \right)$. Thus, summing over all $\nu, \nu'$, we get for $i\ne j$
\[
 \frac{1}{\mu(B)} \int_B X_i X_j = \alpha^2\left (  1+ O(\epsilon) \right).
\]
Using this, \eqref{eq:first-mom}, and the fact that there are at most $(1-\alpha)^2k^2$ pairs of $i \neq j$, we see that
\begin{align}
  \frac{1}{\mu(B)} \int_B X^2 &= \frac{1}{\mu(B)} \int_B X + \sum_{i\ne j}\frac{1}{\mu(B)} \int_B X_i X_j \\
                              &\le (1-\alpha)k\alpha (1+O(\epsilon)) + (1-\alpha)^2k^2 \alpha^2(1+O(\epsilon)). \label{eq:second-mom}
\end{align}

Combining the first and second moment bounds, \eqref{eq:first-mom} and \eqref{eq:second-mom}, we estimate

\begin{align*}
  \frac{1}{\mu(B)}& \int_B (X-(1-\alpha)k \alpha)^2 \\
  = & \frac{1}{\mu(B)}\left(\int_B X^2 - 2(1-\alpha)k \alpha \int_BX + \mu(B)(1-\alpha)^2k^2\alpha^2\right) \\
  \le & \left[(1-\alpha)k\alpha (1+O(\epsilon)) + (1-\alpha)^2k^2 \alpha^2(1+O(\epsilon))\right] \\ 
  & -\left[2(1-\alpha)k \alpha \cdot (1-\alpha)k\alpha (1+O(\epsilon)) \right] + \left[(1-\alpha)^2k^2\alpha^2\right]\\
  = & O(k) + O(\epsilon k^2).
\end{align*}

Using this, we apply a Chebyshev bound.   Note that since we can assume that $\alpha$ is small, if $X \le (\alpha/2)k$ then $|X-(1-\alpha)k \alpha| \ge (1/3)k \alpha$.  So:
\begin{align*}
   \frac{1}{\mu(B)} \cdot \mu\left(\{v: X(v) \le (\alpha/2) k \}\right) &\le  \frac{1}{\mu(B)} \cdot \mu\left( \{v: (X(v)-(1-\alpha) k \alpha)^2 \ge (k \alpha/3)^2\}\right) \\
        & \le \frac{\frac{1}{\mu(B)} \int_B (X-(1-\alpha)k \alpha)^2}{(k \alpha/3)^2}\\
        &\le \frac{O(k)+O(\epsilon k^2)}{(k \alpha/3)^2} \\
        & \le O(1/k) +O(\epsilon),
\end{align*}
which gives the desired result.  
\end{proof}

\begin{sublemma}
  \label{lem:Rm}
  There exists $c_0$ such that for $c\ge c_0$, 
  \begin{align*}
    \mu(R^m) \le \mu(B) O(k/g).  
  \end{align*}
\end{sublemma}

\begin{proof}
  We need to show that, given that there are many $i$ such that $g_{t_i}v$ lands in the union of the $\{B_\nu\}$, it is unlikely that only few \emph{distinct} boxes among the $\{B_\nu\}$ are visited.  We measure repeated visits to the same box
  with the function $Y_{i,j}:B\to \R$  given by
  \begin{align*}
    Y_{i,j}(v) =
    \begin{cases}
      1 \quad \text{if } \exists \nu \text{ such that } g_{t_i}v, g_{t_j}v \in B_\nu\\
      0 \quad \text{otherwise}. 
    \end{cases}
  \end{align*}
  Let $Y(v):=\sum_{i<j\le (1-\alpha)k }Y_{i,j}(v)$.  
  Note that if $v\in R^m$ then $Y(v)\ge \alpha k/2 - \alpha k/4 = \alpha k/4$
  since there must be at least this many values of $j<(1-\alpha)k$ such that $g_{t_j}v$ is in some flow box $B_\nu$ and there exists some $i<j$ for which $g_{t_i}v$ is in the same flow box.  Thus
  \begin{align}
    \mu(R^m) \le \mu(\{v \in B: Y(v) \ge \alpha k/4\}). \label{eq:collision}
  \end{align}
  We now use the first moment method to bound the right hand term above.  Define $Y_{i,j,\nu}:B\to \R$ by
    \begin{align*}
    Y_{i,j,\nu}(v) =
    \begin{cases}
      1 \quad \text{if } g_{t_i}v, g_{t_j}v \in B_\nu\\
      0 \quad \text{otherwise},
    \end{cases}
    \end{align*}
    and note that, by disjointness of the $B_\nu$, we have $Y_{i,j} = \sum_\nu Y_{i,j,\nu}$.  Hence by effective multiple mixing, \Cref{thm:effect-mult-mix},
    
  \begin{align*}
    \int_B Y_{i,j} & = \sum_\nu \int_B Y_{i,j,\nu} = \sum_{\nu} \mu(B\cap g_{-t_i}B_\nu \cap g_{-t_j}B_\nu)\\
                   & = \sum_\nu \mu(B)\mu(B_\nu)^2 O(1)\\
                   & = \alpha \mu(B)^2 O(1).   
  \end{align*}
Then
\begin{align*}
  \int_B  Y = \sum_{i<j\le (1-\alpha)k} \int_B Y_{i,j} \le \left((1-\alpha)k\right)^2 \left(\alpha \mu(B)^2 O(1)\right) \le O(1) k^2\mu(B)^2.
\end{align*}
Using this we apply a Markov bound:
\begin{align*}
  \mu(\{v \in B: Y(v) \ge \alpha k/4\}) \le \frac{ \int_B Y}{\alpha k/4} \le \frac{O(1) k^2\mu(B)^2}{\alpha k/4} = \mu(B) O(k/g),
\end{align*}
since $\mu(B) = O(1/g)$. 
Combining with \eqref{eq:collision} gives the desired result.  

\end{proof}

\subsubsection{$Q$: Vectors that hit many flow boxes $\{B_\nu\}$}
\label{sec:Q}

We now consider those vectors that intersect many flow boxes, and then consider decreasing subsets that avoid progressively more types of self-intersection (arising from hitting both a flow box and its transverse box).  We show that the measure of these subsets decreases in a definite way.  This is where the reasoning from the birthday paradox is applied.  

We fix a positive integer $k$.  Then let 
    \begin{align*}
      Q^= := & B-R \\
          = &\Big\{v\in B: \ \exists  B_{\nu_1(v)}, \ldots, B_{\nu_{\lfloor (\alpha/4) k\rfloor}(v)} \in \{B_\nu\} \text{ distinct s.t. } \forall i \le \frac \alpha 4 k,  \\
      & \text{ \ \ \ } \exists  j\le (1-\alpha)k \text{ s.t. } g_{t_j}v \in B_{\nu_i(v)} \Big\},
    \end{align*}
    i.e. $Q^=$ consists of vectors that hit at least $\frac \alpha 4 k$ distinct elements of $\{B_\nu\}$ among times $t_1,\ldots, t_{\lfloor(1-\alpha)k\rfloor}$, which we label $B_{\nu_1(v)}, \ldots, B_{\nu_{\lfloor (\alpha/4) k\rfloor}(v)}$.

    To deal with ``edge effect components'', we will also consider similar sets defined with respect to flow boxes of slightly different size (as in proofs of \Cref{thm:effect-pgt} and \Cref{thm:effect-mult-mix}).  Hence, on first reading, the reader is advised to ignore the differences between $Q^+,Q^=,$ and $Q$.  Let $Q^+\subset B^+$ be the corresponding set for the $\{B_\nu^+\}$ (recall the enlarged flow boxes $B^+$ as defined in \Cref{sec:notation-setup}), i.e.
        \begin{align*}
          Q^+ := & \Big\{v\in B^+: \ \exists  B^+_{\nu^+_1(v)}, \ldots, B^+_{\nu^+_{\lfloor (\alpha/4) k\rfloor}(v)} \in \{B^+_\nu\}\text{ distinct s.t. } \forall i \le \frac \alpha 4 k,  \\
      & \text{ \ \ \ } \exists  j\le (1-\alpha)k  \text{ s.t. } g_{t_j}v \in B^+_{\nu^+_i(v)} \Big\}.
        \end{align*}

        Note that $Q^=\subset Q^+$.  We then define $Q$ to be the union of components of $Q^+$ that intersect $Q^=$ i.e.
        \begin{align*}
          Q:=\bigcup \{\mathcal{C} \text{ component of } Q^+ : \mathcal{C} \cap Q^=\ne \emptyset\}.
        \end{align*}

        Note that $Q^= \subset Q\subset Q^+$.

In the below lemma, starting with $Q$, we will construct progressively smaller sets $Q_j$ by imposing further conditions.  Part (\ref{item:Qshape})  corresponds to the fact that $Q_j$ is defined in terms of conditions on the behavior of vectors under geodesic flow only up to time $t_j$.  The tangent vectors removed at each stage will correspond to geodesics with ``robust self-intersection", and in particular vectors in $S$, i.e. those corresponding to simple geodesics, will not be removed (though some of these may be in $R$, and thus may not even be in the first set $Q$).  

  \begin{sublemma}
    \label{lem:Qjstruc}
    For each $k\ge 0$, there exist sets 
      \begin{align*}
    Q_{\lceil (1-\alpha)k \rceil} \supset Q_{\lceil (1-\alpha)k\rceil + 1} \supset \cdots \supset Q_k
      \end{align*}
      where $Q_{\lceil (1-\alpha)k \rceil}:=Q$, and such that for each $j=\lceil (1-\alpha)k\rceil, \ldots, k$:
    \begin{enumerate}[(i)]
    \item  $S\cap Q = S\cap Q_{j}$.  Furthermore, for any $u\in Q_{j-1}-Q_{j}$, there exists some $B_\nu$ and $j'\le k$ such that $g_{t_{j'}}u \in B_\nu^{++ }$ and $g_{t_{j}}u \in \hat B^-_\nu$ (which implies the geodesic segment through $u$ has a ``robust self-intersection''), \label{item:Qsimp} %
  \item  $\mu(Q_j) \le \mu(B) \cdot \left\{1-\Omega(1) k\mu(B)\right\}^{j-\lceil (1-\alpha)k \rceil}$, \label{item:Qmeas}
  \item $Q_j$ is a union of subboxes that are full width in the contracting direction, and have width $\ge e^{-t_j}\eta$ in the expanding direction. 
  \label{item:Qshape}
  \end{enumerate}
  \end{sublemma}

  \begin{proof}

    We will inductively construct the sets, verifying the listed properties along the way.

    For the base case, set $Q_{\lceil (1-\alpha)k \rceil}:=Q$.  Then properties (\ref{item:Qsimp}) and (\ref{item:Qmeas}) are immediate.
    For (\ref{item:Qshape}), the idea is that the conditions depend on behavior with regard to lying in certain flow boxes at times only up to $t_{(1-\alpha)k}$, and thus the components should be full width in the contracting direction, and expanding width $\ge e^{t_{(1-\alpha)k}}\eta$.  However, because of ``edge effect components", this is not literally true of components of $Q^=$.  On the other hand, since we have defined $B_\nu$ and $B_\nu^+$ to have sufficient space between them (in particular, any vector in the complement of $B_\nu$ within $B_\nu^+$ is contained in an $\eta$ flow box fully contained in that complement), the desired property will be true of components of $Q$.  

    Suppose we have constructed $Q_j$ with the desired properties; we will now construct $Q_{j+1}$.

      Note that any $v\in Q_j$ is also in $Q$, hence also in $Q^+$, which means there exist $B^+_{\nu^+_1(v)}, \ldots, B^+_{\nu^+_{\lfloor (\alpha/4) k\rfloor}(v)} \in \{B^+_\nu\}$ distinct such that for each  $i \le \frac \alpha 4 k$ there exists a $j\le (1-\alpha)k \text{ such that } g_{t_j}v \in B^+_{\nu_i(v)}$.  Let $A(v)$ be the flow box centered at $v$ that is full width in contracting and flow directions (i.e. width $\eta$), and width $e^{-t_j}\eta$ in the expanding direction.  Now any $v'\in A(v)$ will satisfy the same property as $v$ but with the boxes enlarged, i.e.  $g_{t_j}v' \in B^{++}_{\nu_i(v)}$ (for the same $i$'s and $j$'s as for $v$).   We now apply the Vitali Covering Lemma, %
      taking the ``balls" to be the $A(v)$, which yields a \emph{disjoint} sub-collection $\{A(v)\}$, such that the enlarged sets $\{A(v)^+\}$  cover all of $Q_j$.  
    Hence we can cover at least $1/1000$ of the measure of $Q_j$ %
    by a union of a set $\mathcal P$ of \emph{disjoint} subboxes that have full width in the contracting direction and width $\ge e^{-t_j}\eta$ in the expanding direction, and such that for each $P\in \mathcal P$ we can take a \textit{common value} of $\nu_1(P), \ldots, \nu_{\lfloor (\alpha/)4 k \rfloor}(P)$ that works for all $v'\in P$. That is, for each $i$ there exists a $j\le (1-\alpha)k $ such that $g_{t_j}v' \in B^{++ }_{\nu_i(P)}$ for all $v'\in P$. 
  
  Using this definition of $\mathcal P$, we now define three closely related sets $Q^=_{j+1}$, $Q^-_{j+1}$, $Q_{j+1}$ (the reader is advised to ignore the differences between them on first reading; the somewhat complicated definitions are needed to take care of ``edge effect components").  The idea is to remove vectors $v$ for which $g_{t_{j+1}}v$ hits transverse boxes to certain flow boxes already visited.  We set: 
  \begin{align*}
    Q^=_{j+1} := Q_j - \bigcup_{P \in \mathcal P}  \bigcup_{\ell=1}^{\lceil\frac \alpha 4 k\rceil} P \cap g_{-t_{j+1}}\hat B_{\nu_\ell(P)}, \\
    Q^-_{j+1} := Q_j - \bigcup_{P \in\mathcal P} \bigcup_{\ell=1}^{\lceil \frac \alpha 4 k \rceil} P \cap g_{-t_{j+1}}\hat B_{\nu_\ell(P)}^- .
  \end{align*}
  Note that $Q^=_{j+1} \subset Q^-_{j+1}$.  
  Finally, define $Q_{j+1}$ to be the union of components of $Q^-_{j+1}$ that intersect $Q^=_{j+1}$.  Note that $Q^=_{j+1} \subset Q_{j+1}\subset Q^-_{j+1}$.
  
  We now verify the desired properties of $Q_{j+1}$ (assuming the properties for $Q_j$).  
  
  \begin{enumerate}
    \item    %

    For Property (\ref{item:Qsimp}), we first prove the second more specific statement.  We will prove the condition holds for any tangent vector $u$ in $Q_j-Q_{j+1}^=$, from which the desired result follows since $Q_{j+1}^=\subset Q_{j+1}$.
    Note that for any such $u$, there is some $P\in \mathcal{P}$ and $\ell$ such that $u\in P$, and $g_{-t_{j+1}}\hat B_{\nu_\ell(P)}$, i.e. $g_{t_{j+1}}u\in \hat B_{\nu_\ell(P)}$.  On the other hand, by definition of $\nu_\ell(P)$, there exists $j'$ such that $g_{t_{j'}}u\in B^{++ }_{\nu_\ell(P)}$.  This is the desired statement.

    Next we show that
    $$S\cap Q_j= S\cap Q_{j+1}.$$
    In fact, by the above, for any $u\in Q_j-Q_{j+1}$, its geodesic segment hits both $B^{++}_{\nu_\ell(P)}, \hat B_{\nu_\ell(P)}$, which forces it to have a self-intersection (by \Cref{prop:flow-boxes}, \eqref{item:self-inter}), and hence cannot lie in $S$.  
    Iterating the equality over $j$ gives $S\cap Q= S\cap Q_{j+1}.$

    \item   Property (\ref{item:Qmeas}) we will prove by bounding $\mu(Q^-_{j+1})$ using disjointness of  $\hat B_{\nu_\ell(P)}^- , \hat B_{\nu_{\ell'}(P)}^-$.  

    First note that for any $P\in \mathcal P$, by \Cref{lem:band-future} (for which the width hypothesis holds by property \eqref{item:Qshape} for $Q_j$, which we know by induction; the expanding width might be larger than $e^{-{t_j}}\eta$, in which case we find disjoint subboxes of this width that cover a definite fraction of the band, and apply the lemma to each of these to each of these.)   
    \begin{align}
        \label{eq:band-lower}
        \mu\left(P \cap g_{-t_{j+1}}\hat B_{\nu_\ell(P)}^-\right) \ge \Omega(1) \mu(P)\mu(B_{\nu_\ell(P)}^-) \ge \Omega(1) \mu(P)\mu(B).
    \end{align}

    Then 
    \begin{align*}
        \mu(Q_{j+1}) &\le \mu(Q^-_{j+1}) = \mu \left( Q_j - \bigcup_{P \in\mathcal P} \bigcup_{\ell\le \frac \alpha 4 k} P \cap g_{-t_{j+1}}\hat B_{\nu_\ell(P)}^- \right)  \\
                     & \le \mu(Q_j) - \sum_{P\in \mathcal P} \sum_{\ell \le \frac \alpha 4 k}\mu\left( P \cap g_{-t_{j+1}}\hat B_{\nu_\ell(P)}^- \right)  && \text{ (by disjointness)} \\
                     & \le \mu(Q_j) - \sum_{P\in \mathcal P} \sum_{\ell \le \frac \alpha 4 k}\Omega(1) \mu(P)\mu(B) && \text{ (by \eqref{eq:band-lower})}\\
                     &      \le \mu(Q_j) - \mu(B) \cdot \frac \alpha 4 \cdot k \cdot \Omega(1) \sum_{P\in \mathcal P} \mu(P) \\
                     &      \le \mu(Q_j) - \mu(B) \cdot \frac \alpha 4 \cdot k \cdot \Omega(1)\cdot \frac{\mu(Q_j)}{1000} && \text{ (using property of } \mathcal P)\\
                     &      \le \mu(Q_j)\left\{1- \Omega(1)\mu(B) k\right\}\\
                     & \le \mu(B) \left\{1-\Omega(1)\mu(B) k \right\}^{j-\lceil (1-\alpha)k \rceil} \left\{1- \Omega(1)\mu(B) k\right\} && \text{ (by inductive hyp. for } Q_j)\\
                     & \le \mu(B) \left\{1-\Omega(1)\mu(B) k \right\}^{j+1-\lceil (1-\alpha)k \rceil}. 
    \end{align*}

    \item     For Property  (\ref{item:Qshape}), we use the fact that the definition of $Q_{j+1}$ is in terms of behavior of vectors under geodesic flow only up to time $t_{j+1}$. This itself does not immediately imply the property, due to the possibility of edge effect components.  However, we avoid these by, in the definition of $Q_{j+1}$, only taking those components of $Q_{j+1}^-$ that intersect $Q^=_{j+1}$.  Since everything is happening in $Q$, which is constructed by \emph{removing} vectors that hit certain combinations of flow boxes in the future, we also use the “full box separated” property.

    \end{enumerate}
    
  \end{proof}

\renewcommand\qedsymbol{$\blacksquare$} 

  \subsubsection{Proof of \Cref{lem:no-inter}}
\label{sec:proof-no-inter}

We now combine the ingredients we've just developed.  
\begin{proof}[Proof of \Cref{lem:no-inter}]
  We define
  \begin{align*}
    \bar S_k := R \cup Q_{k},
  \end{align*}
  where $R=R(k)$ was defined in \Cref{sec:R}, and $Q_{k}$ in \Cref{lem:Qjstruc}.  
  \begin{enumerate}
  \item   For (\ref{item:barSsimp}), note that by definition of $Q^=$, we have $B=R\cup Q^=$, and since $Q\supset Q^=$, we also have $B=R\cup Q$.  By applying \Cref{lem:Qjstruc} (\ref{item:Qsimp}) we then see that $S\subset R\cup Q_k= \bar S$.
    
  \item For (\ref{item:barSrobust}), note that any $u\in B-\bar S_k$ is in some $Q_j-Q_{j+1}$ (it is ``removed'' at some stage in the process of paring down the $Q_j$).  By the second part of \Cref{lem:Qjstruc} (\ref{item:Qsimp}),  this means that $u$ satisfies the desired property.  
    
  \item   For the measure bound (\ref{item:barSmeas}), we have
  \begin{align*}
    \mu(\bar S_k) &\le \mu(R) + \mu(Q_k) \le \mu(R^f) + \mu(R^m) + \mu(Q_k) 
  \end{align*}
  and then we bound the terms on the right using the measure estimates of \Cref{lem:Rf}, \Cref{lem:Rm}, and \Cref{lem:Qjstruc} (\ref{item:Qmeas}).

\end{enumerate}

\end{proof}

\subsection{Simple closed geodesics hitting $B$}
\label{sec:simple-geod-B}

In this subsection, we will use the previous results to bound the number of times simple closed geodesics of a given length pass through a given flow box.    

In \Cref{lem:no-inter}, for any flow box $B$, we defined a set $\bar S_k \subset B$ of directions in $B$ that do not have specific types of ``robust" self-intersection after $kc \log g$ time. Recall that $S$ is the set of vectors in $B$ tangent to simple (not necessarily closed) geodesics. So $S \subset \bar S_k$.  We will show that $S$ is actually ``buffered" inside $\bar S_k$ in the following sense:
\begin{lemma}
\label{lem:buffer}
Fix $\delta, s_0, \epsilon > 0$. Then there exist $c_0, \eta_0 > 0$ so that for all $c > c_0$, $\eta$ with $\eta_0/100 < \eta < \eta_0$, and  $k \in \mathbb N$, the following holds. %

Let $X$ be a $\delta$-expander surface of genus $g$ with systole at least $s_0$. Let $B$ be an $\eta$-flow box. Suppose $\bar S_k \subset B$ is the set given by \Cref{lem:no-inter} for the given $\delta, s_0, \epsilon, c$ and $k$. 
Let $v \in B$ lie on a simple geodesic, and let $P \subset B$ be any flow box containing $v$, of width $e^{-ck\log g} \eta$ in the expanding direction, and $\eta$ in both the contracting and geodesic flow directions. Then \[P \subset \bar S_k.\]
\end{lemma}
\begin{proof}
Given $\delta, s_0$ and $ \epsilon$, choose $\eta_0$ so that \Cref{lem:no-inter} holds for $27 \eta_0$, and \Cref{prop:flow-boxes} holds for $27\eta_0$. Given this choice of $\eta_0$, we choose $c_0$ so that \Cref{lem:no-inter} holds for $c_0/2$. Take any $c > c_0$. Let $\eta>0$ with $\eta_0/100 < \eta < \eta_0$. Note that the conclusion of \Cref{lem:no-inter} will hold for this $\eta$ as well. Let $B$ be any $\eta$-flow box. Fix any $k \in \mathbb N$.

Take the set $\bar S_k \subset B$ guaranteed by \Cref{lem:no-inter} for the above choices of $\delta, s_0, \epsilon, c, \eta$ and $k$.

  Let $v \in B$ lie on a simple geodesic. Instead of working with an arbitrary flow box containing $v$, we let $P$ be the $(2 e^{-ck \log g} \eta, 2 \eta, 2\eta)$ flow box centered at $v$.
  Note that $P$ will contain any $(e^{-ck \log g} \eta, \eta, \eta)$ flow box containing $v$, but $P$ need not lie entirely inside $B$. We will show that $P \cap B \subset \bar S_k$.

Suppose for contradiction that there is some $w \in B - \bar S_k$ for which $w \in P$. In that case, by Lemma \ref{lem:no-inter} part (\ref{item:barSrobust}) there is some $\eta$ flow box $B_0$,
and some $t_1, t_2$, with $0 \leq t_1, t_2 \leq k c \log g$, so that 
\[
 g_{t_i} w \in B_0^{++ }, \text{ and } g_{t_j} w \in \hat B^-_0
\]
(recall that each additional $+$ superscript multiplies the dimensions of the respective boxes by 3, while $-$ divides by 3). 

Since $P$ is a $(2 e^{-ck \log g} \eta, 2 \eta, 2\eta)$ flow box, and $0 \leq t_1, t_2 \leq ck\log g$, it follows that $g_{t_1} P$ and $g_{t_2} P$ are flow boxes of dimension at most $2 \eta$ centered at $g_{t_1}v$ and $g_{t_2} v$, respectively. Since 
\[
 g_{t_1} w \in g_{t_1} P \cap B_0^{++}, \text{ and } g_{t_2} w \in g_{t_2} P \cap \hat B_0^-,
\]
we have that
\[
 g_{t_1} v \in B_0^{+++ }, \text{ and } g_{t_2} v \in \hat B_0^+,
\]
where the flow boxes $B_0^{+++ }$ and $B_0^+$ have dimension at most $27 \eta$.

We chose our $\eta_0$ so that $27\eta$ satisfies  \Cref{prop:flow-boxes}. So by part (\ref{item:self-inter}) of that proposition, the geodesic tangent to $v$ has a self-intersection. But this contradicts the assumption that $v$ lies on a simple geodesic. In particular, $P \cap B$ must have been contained inside $\bar S_k$. 
\end{proof}

We now count the number of times that simple closed geodesics of a given length pass through a fixed flow box.  

By \Cref{lem:buffer}, each $v \in S $ lies inside a $(e^{-kc\log g} \eta, \eta, \eta)$ flow box $P_v \subset \bar S_k$. In fact, if $v, v'$ both lie on the same geodesic arc $\sigma \subset B$, then we can choose $P_v = P_{v'}$. So for  each length $\eta$ geodesic arc $\sigma \subset S$, we let $P_\sigma$ denote any $(e^{-kc\log g} \eta, \eta, \eta)$ flow box so that $\sigma \subset P_\sigma \subset \bar S_k$, and set
\[
\core(\bar S_k) := \bigcup_{\sigma} P_\sigma \subseteq \bar S_k
\]
(see Figure \ref{fig:Simple_set_buffers}). We have $S \subset \core(\bar S_k)$.

\begin{figure}[ht]
 \centering 
 \includegraphics{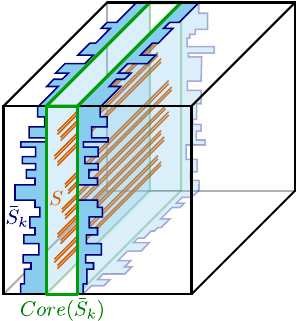}
 \caption{A component of $\core(\bar S_k)$ contained in a component of $\bar S_k$. Both contain the set $S$, that is, those vectors in $B$ that lie on simple geodesics.}
 \label{fig:Simple_set_buffers}
\end{figure}

Next, choose any $t > \eta$, and define
\[
G_{simp}(B,t,\eta)
\]
to be the set of vectors in $B$ that lie on some simple closed geodesic of length between $t-\eta$ and $t+\eta$, that passes through $B$ in a length $\eta$ arc. Let $N_{simp}(B,t,\eta)$ be the number of length $\eta$ arcs in this set. 

The strategy is now as follows. Note that
\begin{align*}
    G_{simp}(B,t,\eta) &\subset S \cap g_{-t}B\\
     & \subset \core(\bar S_k) \cap g_{-t}B.
\end{align*}
First, we will show that each connected component of $\core(\bar S_k) \cap g_{-t}B$ contains at most one length $\eta$ arc $\sigma \subset G_{simp}(B,t,\eta)$ (\Cref{cla:simp-comp-bij}). This allows us to use the measure bound on $\bar S_k$ from \Cref{lem:no-inter} to get an estimate on $N_{simp}(X,t,\eta)$ in \Cref{lem:no-inter-return}. Lastly, we get a cleaner bound on $N_{simp}(B,t,\eta)$ for large $t$ in \Cref{prop:eps-error}.

\begin{lemma}
\label{cla:simp-comp-bij}

Fix $\delta, s_0, \epsilon > 0$. Then there exist $c_0, \eta_0 > 0$ so that for all $c > c_0$, $\eta>0$ with $\eta_0/100< \eta < \eta_0$, and  $t > 2 s_0$, the following holds. Let $X$ be a $\delta$-expander surface with systole at least $s_0$. Let $B$ be an $\eta$-flow box on $X$. Let $k \in \mathbb N$, %
and $\bar S_k$ be the set from \Cref{lem:no-inter} for this $\delta, s_0, \epsilon, c$ and $k$. Let $\core(\bar S_k)$  be the subset of $\bar S_k$ defined above. 
We have
\[
N_{simp}(B,t,\eta) \leq \comp(\core(\bar S_k) \cap g_{-t}B).
\]
\end{lemma}
\begin{proof}
If $v \in G_{simp}(B,t,\eta)$, then $g_\ell v\in B$, for some $\ell$ with $t-\eta \leq \ell \leq t+\eta$.  Let $\sigma$ be the geodesic segment in $B$ containing $v$; there must be some $v' \in \sigma$ so that $g_t v' \in B$. Thus, every such segment $\sigma$ passes through $\core(\bar S_k) \cap g_{-t}B$.

As $\core(\bar S_k) \subset B$, each component of $\core(\bar S_k) \cap g_{-t}B$ lies in some connected component of $B \cap g_{-t}B$. By Lemma \ref{lem:closing-unique}, each component of $B \cap g_{-t}B$ intersects at most one segment $\sigma$ of length $\eta$ of a  (not necessarily simple) closed geodesic with length in $[t-\eta, t+\eta]$.  Thus, at most one such segment passes through each connected component of $\core(\bar S_k) \cap g_{-t}B$.
\end{proof}

To bound $N_{simp}(B,t,\eta)$, it would be natural to try to apply the measure estimate \Cref{lem:no-inter} \eqref{item:barSmeas} with $k$ chosen so that $t=ck\log g$. Most of the terms in  that estimate decay as $k$ increases, but the $O(k/g)$ term does not. Thus, for large $t$, we will in fact use a smaller $k$.  We first prove a flexible estimate that allows for a range of choices of $k$ depending on $t$.  

\begin{lemma}
  \label{lem:no-inter-return}
Fix $\delta, s_0, \epsilon > 0$. Then there exist $c_0, \eta_0, \alpha> 0$ so that for all $c > c_0$ and $\eta$ with $\eta_0/100 < \eta < \eta_0$,  
the following holds. 
Let $X$ be a $\delta$-expander surface of genus $g$ with systole at least $s_0$. Let $t>0$, and suppose $k \in \mathbb N$ so that $ck \log g < t$. 
Let $B\subset T^1X$ be an $\eta$ flow box. 
  Then 
 \[
 N_{simp}(B,t,\eta) \leq O(1) e^{t} \mu(B) \left(\epsilon + O(1/k) + O(k/g) + \left\{1-\Omega(1)k\mu(B)\right\}^{\alpha k}\right).
\] 
\end{lemma}
\begin{proof}
Fix $t > 0$. Choose $\eta_0, c_0$ and $\alpha$ so that Lemmas \ref{lem:comp-P-cap-B}, \ref{lem:no-inter} and \ref{lem:buffer} hold for $c_0/2$ (and the given $\eta_0$ and $\alpha$). Choose any $\eta$ with $\eta_0/100 < \eta < \eta_0$.

Let $c > c_0$. %
Now we let $\bar S_k$ be the set from \Cref{lem:no-inter} for the above $\delta,s_0, \epsilon, \frac c2, \eta$ and $k$. Take its core set $\core(\bar S_k)$. 

Suppose $\core(\bar S_k)$ has $N$ connected components, $C_1, \dots, C_N$. From the way $\core(\bar S_k)$ was defined, it is a union of subboxes that have width $e^{-\frac c2k \log g} \eta$ in the expanding direction (and full width in the contracting and geodesic flow directions). Thus, each $C_i$ is a subbox of width $e^{-w_i}\eta$ in the expanding direction with 
\[
e^{-w_i} \eta > e^{-\frac c2 k \log g} \eta.
\]
Thus, $w_i < \frac c2 k \log g$. As $ck\log g < t$ and $k \geq 1$, we have $t> w_i + \frac c2 k \log g$. So, we can apply \Cref{lem:comp-P-cap-B}, and get 
\[
\comp (C_i \cap g_{-t}B) \leq O(1)e^{t-w_i}\mu(B).
\]
Summing over all components $C_i$, we get 
\begin{align}
\label{eq:comp_core_first}
    \comp(\core(\bar S_k) \cap g_{-t}B) & \leq O(1) \mu(B) e^{t}\sum_{i=1}^N e^{-w_i}.
\end{align}

We must now bound the term $\sum_{i=1}^N e^{-w_i}$, which is the total width of $\core(\bar S_k)$. We do this using the bound on the measure of $\bar S_k$. Using the flow box structure of each component $C_i$, we get
\begin{align*}
    \mu(\core(\bar S_k)) & = \sum_{i=1}^N \mu(C_i) \\
     & = \Omega(1) \sum_{i=1}^N e^{-w_i}\mu(B),
\end{align*}
where the implicit constant in the $\Omega(1)$ depends only on $\eta_0$. On the other hand, $\mu(\core(\bar S_k))$ is bounded above by $\mu(\bar S_k)$, which we have bounded in \Cref{lem:no-inter}:
\begin{align*}
    \mu(\core(\bar S_k)) & \leq \mu(\bar S_k) \\
     & \leq \mu(B) \left(\epsilon + O(1/k) + O(k/g) + \left\{1-\Omega(1)k\mu(B)\right\}^{\alpha k}\right). 
\end{align*}
Combining this with the previous estimate, we see that 
\[
\sum_{i=1}^N e^{-w_i} \leq O(1)\left(\epsilon + O(1/k) + O(k/g) + \left\{1-\Omega(1)k\mu(B)\right\}^{\alpha k}\right).
\]
Finally, applying this estimate to \eqref{eq:comp_core_first} gives
\[
\comp(\core(\bar S_k) \cap g_{-t}B) \leq O(1)e^t\left(\epsilon + O(1/k) + O(k/g) + \left\{1-\Omega(1)k\mu(B)\right\}^{\alpha k}\right).
\] 
By \Cref{cla:simp-comp-bij}, this number of components bounds $N_{simp}(B,t,\eta)$, and so our lemma follows.

\end{proof}

The previous proposition gave a general bound that holds for all lengths $t$ large enough.  We will now specify to the relevant length regime for our purposes.   This involves estimates for several terms that are analogs of the estimates in the various cases in proof of \Cref{prop:prob-problem}.  

\begin{prop}
  \label{prop:eps-error}
  Fix $\delta, s_0, \epsilon > 0$. There exist $d,\eta_0,g_0>0$ such that for any $\eta$ with $\eta_0/100 < \eta < \eta_0$, and any $\delta$-expander surface $X$ of genus $g\ge g_0$ and systole at least $s_0$, the following holds. Let $B\subset T^1X$ be an $\eta$ flow box. Let $t > d\sqrt{g} \log g$. Then 
  \begin{align*}
  N_{simp}(B,t,\eta) \leq \epsilon \cdot e^{t} \mu(B).
  \end{align*}
\end{prop}
\begin{proof}
By \Cref{lem:no-inter-return}, for any fixed $\epsilon' > 0$, there is a $c>0$ so that 
\begin{align}
N_{simp}&(B,t,\eta) \\
& \leq e^{t} \mu(B) \left(O(\epsilon') + O(1/k) + O(k/g) + O\left(\left\{1-\Omega(1)k\mu(B)\right\}^{\alpha k} \right) \right) \label{eq:N-simp-bound}
\end{align}
for any $k<\frac{t}{ c\log g}$.
First, we can choose $\epsilon'$ so that the $O(\epsilon')$ term is bounded above by $\epsilon/4$. Note that this gives us a fixed choice of $c$ that we will use for the remainder of the proof.

We now choose some $d>c$.  As we proceed through the proof, we will need to make it progressively larger, but only depending on the parameters allowed in the statement.  Take 
\begin{align}
    k = \left \lfloor \min \left(\frac{t}{ c\log g}, \ \frac 1d g^{2/3}\right) \right \rfloor. \label{eq:kmin}
\end{align}
The reason for including $\frac{1}{d} g^{2/3}$ in the min is so that the $O(k/g)$ error term in \eqref{eq:N-simp-bound} can be made small, as we will soon see (a similar min appears in the analogous probability problem, \Cref{prop:prob-problem}; any function growing much faster than $\sqrt{g}$ and much slower than $g$ would work).  

We now estimate the remaining error terms in (\ref{eq:N-simp-bound}) one by one.

We have that 
\begin{align*}
    O(1/k) = O\left( \max\left( \frac{c\log g}{t}, \frac{d}{g^{2/3}} \right) \right) \le  O\left(\max\left( \frac{c}{d\sqrt g}, \frac{d}{g^{2/3}} \right) \right),
\end{align*}
which can be made less than $\epsilon/4$, for all $g\ge g_0$, by taking $g_0$ sufficiently large. 

Next 
\begin{align*}
    O(k/g) = O\left( \min\left( \frac{t}{cg\log g} , \frac{1}{dg^{1/3}} \right) \right) \le O(1/d),
\end{align*}
which can be made less than $\epsilon/4$ by taking $d$ large.  

To estimate the last error term, we will use the bound 
\begin{align*}
 \left\{1-\Omega(1)k\mu(B)\right\}^{\alpha k} & \leq \exp\left(-\alpha\Omega(1)\mu(B)k^2\right),
\end{align*}
which follows from the fact that $1-x\le e^{-x}$ for $x\ge 0$. (Recall that in our definition of $\Omega$, the function must be non-negative.  Also note that since $\{x\}:=\max(x,0)$, the bound holds trivially if $1-\Omega(1)k\mu(B)<0$.) 

So we get, using this, the fact that $\mu(B) =\Omega(1/g)$, \eqref{eq:kmin}, and $t> d\sqrt{g} \log g$, that
\begin{align*}
 O\left(\left\{1-\Omega(1)k\mu(B)\right\}^{\alpha k}\right) & \leq \exp\left( - \Omega(1)\mu(B)k^2\right) \\
 & \le \exp\left( - \Omega(1)k^2/g\right)  \\
 & \le \max\left( \exp\left(-\Omega(1)\frac{t^2}{(c\log g)^2}\frac{1}{g}\right), \exp\left(-\Omega(1) \frac{g^{4/3}}{d^2g}\right) \right) \\
 & \le \max\left( \exp\left(-\Omega(1)d^2/c^2\right), \exp\left(-\Omega(1) g^{1/3}/d^2\right) \right),
\end{align*}
which can be made less than $\epsilon/4$ for all $g\ge g_0$ by first taking $d$ large, and then $g_0$ large.

Combining all the above, we see that we have found a $d$ and $g_0$ such that for all $g\ge g_0$,
\begin{align*}
    N_{simp}(B,t,\eta) \le e^t\mu(B) (\epsilon/4 + \epsilon/4 + \epsilon/4 + \epsilon/4) =\epsilon \cdot e^t\mu(B). 
\end{align*}
\end{proof}

 \subsection{Completing the proof of \Cref{thm:simplicity}}
 \label{sec:proof}

We now complete our estimates of the number of simple closed geodesics of length at most $L$, denoted $N_{simp}(X,L)$. In \Cref{prop:eps-error}, we bounded $N_{simp}(B,t,\eta)$, the number of times that simple closed geodesics of length roughly $t$ can pass through any fixed $\eta$ flow box, when $\eta$ is small enough, and $t$ is at least on the order of $\sqrt g \log g$. Applying \Cref{prop:fb-to-total} to the set of $t \in [L/2,L]$, we will now get an upper bound on the total number of closed geodesics of length at most $L$.%

\begin{prop}
\label{prop:num-simple}
     Fix $\delta, s_0,\epsilon>0$.  There exists a constant $d$ such that for any $\delta$-expander surface $X$ of genus $g$, with systole at least $s_0$, and all $L> d\sqrt{g} \log g$,
  \begin{align*}
    N_{simp}(X,L) \le  \epsilon \cdot e^L/L.
  \end{align*}
\end{prop}

\begin{proof}
By \Cref{thm:effect-pgt} (with $\epsilon=1$), there exists $c$ such that for all $L>c \log g$,
\begin{align*}
    N_{simp}(X,L/2) \le N(X,L/2) \le 2 \cdot \frac{e^{L/2}}{L/2} =\frac{4}{e^{L/2}}\cdot \frac{e^L}{L}. 
\end{align*}
By taking $c$ sufficiently large (depending only on $\delta,s_0,\epsilon$), we can ensure that the term $\frac{4}{e^{L/2}}$ in the above is less than $\epsilon,$ and so for all $L>c \log g$, 
\begin{align}
    N_{simp}(X,L/2) < \epsilon \cdot e^L/L. \label{eq:low-length}
\end{align}

Now choose $\eta_0$ small enough to satisfy the condition of \Cref{prop:fb-to-total} and \Cref{prop:eps-error}, and let $\eta$ be such that $\eta_0/100 \le \eta \le \eta_0$.  For each $v \in T^1X$, let $B = B(v)$ be an $\eta$ flow box centered at $v$, which we can assume is embedded in $T^1X$, since $\eta_0$ is small (and can depend on the systole lower bound $s_0$). 

Now take a $d > 0$ and $g_0 > 0$ so that \Cref{prop:eps-error} holds for $d/2$ (and $g_0$). That is, there exists $d$ such that if $X$ is a $\delta$-expander surface of genus $g\ge g_0$ and systole at least $s_0$, then for any $t> d/2  \sqrt{g}\log g$, we have 
\begin{align*}
  N_{simp}(B,t,\eta) \leq \epsilon \cdot e^{t} \mu(B).
\end{align*}
Note that we can increase $d$ as much as we like, and the equation will still hold. So we can assume $d \sqrt g_0 > c$ (and so, $d \sqrt g > c$ for all $g > g_0$.)

Choose $L > d \sqrt g \log g$. Then for all $t \in [L/2, L]$, we have $N_{simp}(B,t,\eta) \leq \epsilon \cdot e^{t} \mu(B)$.
So we can apply \Cref{prop:fb-to-total} with $G=G_{simp}$, the set of simple closed geodesics. This gives, for all $L> d \sqrt{g} \log g$ and $g\ge g_0$, 
\begin{align}
    N_{simp}(X,[L/2,L]) < 12 \epsilon \cdot  \frac{e^L}{L}. \label{eq:mid-to-high}
\end{align}

We ensured that $d \sqrt g > c$. So if $L > d \sqrt g \log g$, then we also have $L > c \log g$. 
So we can combine \eqref{eq:low-length} and \eqref{eq:mid-to-high} to get, for any $L> d\sqrt{g} \log g$ and $g\ge g_0$
\begin{align}
    N_{simp}(X,L) = N_{simp}(X,L/2) + N_{simp}(X,[L/2,L]) \le 13\epsilon \cdot e^L/L. \label{eq:high-genus}
\end{align}

It remains to resolve the low genus case $(g\le g_0)$, for which we use a completely different argument (not using the expander assumption at all).  

By \cite{rivin2001}, we have that for any fixed $X$, there exists $C_X$ such that for all $L$,
\begin{align*}
    N_{simp}(X,L) \le C_X \cdot L^{6g-6}.  
\end{align*}
Since an exponential dominates any polynomial, it follows there exists some $L_0=L_0(X,\epsilon)$ such that for all $L\ge L_0$, 
\begin{align}
    N_{simp}(X,L) \le \epsilon \cdot e^L/L.  \label{eq:simp-upper}
\end{align}
We claim that we can take this $L_0$ to be locally constant on each $\M_g$.  Take $L_0$ sufficiently large so that $C_X(2L)^{6g-6} \le \epsilon \cdot e^L/L$ for all $L\ge L_0$.  We can choose a small neighborhood of $X$ in $\M_g$ such that any $X'$ in this neighborhood admits a Bilipschitz map to $X$ with both the map and its inverse having Lipschitz constant at most $2$.  Then for any $L\ge L_0$,
\begin{align*}
    N_{simp}(X',L) \le N_{simp}(X,2L) \le C_X(2L)^{6g-6} \le \epsilon \cdot e^L/L,  
\end{align*}
which shows that $L_0$ can taken to be locally constant. 

Then by the Mumford compactness criterion, %
it follows that there is a single $L_0$ (still depending on $\epsilon$) such that \eqref{eq:simp-upper} holds for any $X$ of genus $\le g_0$ and systole $\ge s_0$.  By enlarging $d$ so that $d\sqrt{2} \log 2 > L_0$, we get that when $L> d\sqrt{g} \log g$, 
\begin{align*}
    N_{simp}(X,L) \le \epsilon \cdot e^L/L, 
\end{align*}
for any $X$ of genus $g \le g_0$ and systole $\ge s_0$ (here we are using that $g\ge 2$, since this is closed hyperbolic surface).  

So combined with \eqref{eq:high-genus} this yields the desired result.

\end{proof}

\begin{proof}[Proof of \Cref{thm:simplicity}]
  We combine \Cref{prop:num-simple} and \Cref{thm:effect-pgt}.  
  
\end{proof}

\section{Filling geodesics}
\label{sec:filling}

\begin{figure}[ht]
  \centering
  \includegraphics[scale=1]{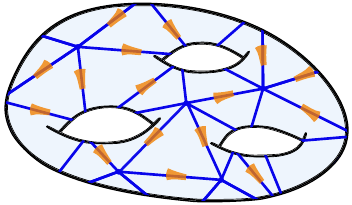}
  \caption{The projection of any closed geodesic $\gamma$ that intersects all the orange flow boxes (which are actually in the unit tangent bundle $T^1X$) has a subset that mimics the blue triangulation. In particular, this subset cuts the surface into topological discs, and hence $\gamma$ is filling.  The blue triangulation is obtained by finding a ``net'' of points that are neither too close together, nor too separated, and then taking the associated Delaunay triangulation.  The resulting triangles have bounded geometry, which means that the orange flow boxes can be taken to be uniformly sized.}
  \label{fig:filling}
\end{figure}

In this section we prove \Cref{thm:filling}.  In \Cref{lem:filling-net}, we construct a set of order $g$ flow boxes (of uniform size) that allow us to detect filling in the sense that a closed geodesic that hits all of them must fill.  In \Cref{lem:meas-avoid-box}, we control the measure of the set of tangent vectors corresponding to geodesics that fail to hit every element among such a set of flow boxes (and that start and end in fixed starting box $B$).  In \Cref{lem:NBL}, we translate this into a bound on the number of \textit{closed geodesics} with the same property.  We combine these pieces in \Cref{lem:non-fill-bd}, an upper bound on non-filling geodesics.  Using this, in \Cref{sec:proof-filling}, we quickly deduce the theorem.

\subsection{Flow boxes to detect filling geodesics}
\label{sec:flow-boxes-filling}

\begin{lemma}
  \label{lem:filling-net}
  Let $s_0>0$. There exists $\eta<s_0/1000$ and $C=C(s_0)$ with the following property.  For $X$ a hyperbolic surface of genus $g$ with systole at least $s_0$, there exist $\eta$ flow boxes $B_1,\ldots, B_{Cg}\subset T^1X$ such that if $\gamma$ is a closed geodesic that intersects every $B_i$, then $\gamma$ is filling.  
\end{lemma}

\begin{proof}

  The idea of the proof is illustrated in \Cref{fig:filling}. 
  
  Choose a finite collection of discs $D_i$ of radius $r\le s_0/6$ that cover $X$ (we may need to make $r$ smaller than this, as discussed later in the proof).   
  By the Vitali covering lemma, among these discs, we can find $D_1,\ldots,D_k$ that are \emph{disjoint}, and such that $3D_1,\ldots, 3D_k$ (where $3D_i$ is the disc with the same center as $D_i$, and $3$ times the radius) cover $X$.  Let $p_i$ be the center of $D_i$.  
  Now let $\mathcal D$ be the Delaunay triangulation of $X$ with respect to the set of vertices $\{p_1,\ldots,p_k\}$ (the Delaunay tessellation may have some faces that are not triangles; if so we perturb the points very slightly, and then we will get a triangulation).  From our assumption that $r\le s_0/6$, every region in the Delaunay tesselation is homeomorphic to a disc. 

    \begin{figure}[h]
      \centering
      \includegraphics[scale=1.4]{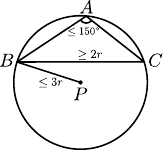}
      \caption{The edge length bounds and Delaunay property give control of angles} 
      \label{fig:Delaunay_triangle}
  \end{figure}

  \begin{claim}
    \label{cla:delaunay}
    Let $r$ be sufficiently small (depending on $s_0$ but not any other features).  Then for any triangle in our $\mathcal D$, the edge lengths lie in $[2r, 6r]$ and the angles lie in  $[4^\circ, 150^\circ]$. 
  \end{claim}

  \renewcommand\qedsymbol{$\square$}
  \begin{proof}
    For  the lower bound on edge lengths, note that since the discs $D_1$,  $\ldots,$ $D_k$ are disjoint, any pair of distinct centers cannot be closer than distance $2r$.  

    For the upper bound on edge lengths, first note that any point in $X$ is at most distance $3r$ from one of $p_1,\ldots,p_k$, since the discs $3D_1,\ldots, 3D_k$ are assumed to cover $X$.  This implies that in the Voronoi tessellation with respect to $\{p_1,\ldots,p_k\}$, any point $x$ in the Voronoi cell $V_i$ containing $p_i$ satisfies $d(x,p_i)\le 3r$.  So if $V_i,V_j$ are adjacent Voronoi cells, then $d(p_i,p_j) \le 6r$.  Since the Delaunay triangulation is the dual of the Voronoi tessellation, all edges in $\mathcal D$ must have length at most $6r$.

    To prove the bounds on angles, consider some triangle $ABC$ in $\mathcal D$; we will first show that $\angle BAC < 150^\circ$.  This triangle has a circumcircle (its center $P$ is at the corresponding Voronoi vertex), and its radius is at most $3r$.  %
    If $\angle BAC \le 90^\circ$, there is nothing to prove, so we will assume that $\angle BAC > 90^\circ$.  If $ABC$ were a Euclidean triangle, we could use the following formula (see \Cref{fig:Delaunay_triangle}):
    
    $$\angle BAC = 180^\circ - 2 \arcsin\left(\frac{|BC|}{2\cdot |PB|}\right)$$
     (note that since $\angle BAC > 90^\circ$, $BC$ must separate the circumcenter $P$ from $A$).  
    Then using that $|BC|\ge 2r$, and $|PB|\le 3r$, the above would give $\angle BAC \le 180^\circ - 2 \arcsin (1/3) < 142^\circ$.  This calculation is for a Euclidean triangle, but since our  hyperbolic triangle is very close to the Euclidean one with the same side lengths (since $r$ is small; make it smaller here if necessary), we get that our actual angle $\angle BAC$ is  less than $150^\circ$.

    We now deduce the lower bound on angles, i.e. $\angle BAC > 4^\circ$.   If the triangle were Euclidean, since we know from above that $\angle BAC<150^\circ$, we could deduce that $\angle B + \angle C > 30^\circ$.  So one of these angles, $\angle B$ without loss of generalization, would satisfy $\angle B > 15^\circ$.  Hence we'd have $\sin B > 0.25$.  By the Law of Sines, and using our upper and lower bounds on edge lengths, it would follow that 
    $$\sin A = \frac{a}{b} \sin B > \frac{1}{3} (0.25) > 0.08.$$
    Hence we'd have $\angle A > \arcsin(0.08) > 4.5^\circ.$  Since our  hyperbolic triangle is very close to the Euclidean one with the same side lengths (since $r$ is small; make it smaller here if necessary), we get that our actual angle $\angle A$ is at least $4^\circ$, as desired.  
  \end{proof}
  \renewcommand\qedsymbol{$\blacksquare$}

  Now we will define the flow boxes $B_1,\ldots,B_{k}$.  For each edge $e_i$ of $\mathcal D$,  let $v_i$ be one of the two vectors tangent to $e_i$ at its midpoint. Then we let $B_i := B(v_i)$ be the $\eta$ flow box centered at $v_i$.
  Using \Cref{cla:delaunay}, we see that if we take $\eta$ sufficiently small (depending only on systole lower bound $s_0$, and not on other features of the surface such as the genus), then any closed geodesic $\gamma$ that intersects all such $B_i$ will contain a subset whose projection to the surface has all complementary regions homeomorphic to discs.  In particular $\gamma$ will be filling.  We can also arrange that $\eta<s_0/1000$. 

  \begin{figure}[ht]
    \centering 
      \includegraphics[scale=1.4]{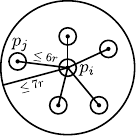}
      \caption{The Delaunay vertices have bounded degree}
      \label{fig:Vertex_degree}             
    \end{figure}

  All that remains is to estimate $k$, the number of the $B_i$.  First, note that $\operatorname{\# vertices}(\mathcal D) \le c\cdot g$, for some $c$ depending only on the systole lower bound $s_0$.  This is because the discs $D_1,\ldots,D_k$, whose centers are exactly the vertices of $\mathcal D$, are disjoint and have radius $r$ (depending only on systole bound $s_0$), while the area of the surface $X$ is $(4g-4)\pi$.  The valence of each vertex $p_i$ in $\mathcal D$ is uniformly bounded by some absolute constant $d$.  In fact, for each vertex $p_j$ that is a neighbor of $p_i$, the disc $D_i$ centered at $p_j$ is entirely contained in the disc of radius $7r$ about $p_i$; since each such $D_i$ has radius $r$ and they are disjoint, there are a bounded number of them.  See \Cref{fig:Vertex_degree}. Hence
  \begin{align*}
    \operatorname{\# edges}(\mathcal D) \le d \cdot \operatorname{\# vertices}(\mathcal D) \le d\cdot c \cdot g.  
  \end{align*}
 Since $k=\operatorname{\# edges}(\mathcal D)$,  taking $C=dc$ gives the desired bound of $Cg$ flow boxes. 
  
\end{proof}

\subsection{Geodesics avoiding flow boxes}
\label{sec:geod-avoid-flow}

Next we get an upper bound on the measure of those $v$ avoiding some box among the $B_1,\ldots,B_m$ at all the prescribed times.   We will use this later with a collection of boxes that has the property that any non-filling closed geodesic must avoid at least one of them.  %

\begin{lemma}
  \label{lem:meas-avoid-box}
   Fix $\delta, s_0, \epsilon>0$ and $\eta<s_0/1000$.  There exists a constant $c$ with the following property.  Let $X$ be a $\delta$-expander surface with systole at least $s_0$, and let $B,B_1,\ldots, B_{m}\subset T^1X$ be $\eta$ flow boxes. Then
  \begin{align*}
    \mu\left(\{ v\in B \cap g_{-t_{k+1}}B: \{g_{t_1}v,\ldots,g_{t_k}v\} \cap B_i = \emptyset \text{ for some } i=1,\ldots,m\}\right)\\
    \le (1+\epsilon) \mu(B)^2 \sum_{i=1}^m \left[1-(1-\epsilon)\mu(B_i)\right]^k ,
  \end{align*}
  for any $t_1<\cdots<t_{k+1}$ with $t_i-t_{i-1}\ge c \log g$ for each $i\ge 2$.
\end{lemma}

\begin{proof}
  For $i=1,\ldots,m$, let
  \begin{align*}
    S_i = \{  v\in B \cap g_{-t_{k+1}}B: \{ g_{t_1}v,\ldots,g_{t_k}v \} \cap B_i = \emptyset \} .
  \end{align*}

  We have 
  \begin{align}
    \mu\left( S_i \right) &=\mu\left(\{ v\in B : g_{t_1}v,\ldots,g_{t_k}v \in T^1X-B_i, g_{t_{k+1}}v \in B \}\right) \\
                          & \le (1+\epsilon)^2 \mu(B)^2 \cdot  \left[1-(1-\epsilon)\mu(B_i)\right]^k, \label{eq:Si}
  \end{align}
  by \Cref{lem:mix-complement} (applied with $k+2$ sets; note that $B^+$ and all the $B_i^+$ are embedded since $\eta<s_0/1000$).  

  Then by a union bound 
  \begin{align*}
    \mu \big(\{ v\in B \cap g_{-t_{k+1}}B &: \{g_{t_1}v,\ldots,g_{t_k}v\} \cap B_i = \emptyset \text{ for some } i=1,\ldots,m\}\big) \\
    &= \mu\left(  \bigcup_{i=1}^m S_i\right) \\
    &\le \sum_{i=1}^m \mu(S_i) \\
    &\le \sum_{i=1}^m(1+\epsilon)^2 \mu(B)^2 \cdot  \left[1-(1-\epsilon)\mu(B_i)\right]^k,
  \end{align*}
  where  in the last line we have used %
  \eqref{eq:Si}. %
  
\end{proof}

Recall $N_G(X,L)$, defined in \Cref{sec:count-box-surf} to be the count of closed geodesics of type $G$.   And, given an $\eta$ flow box $B$ in $T^1X$, we denote by $N_G(B,L,\eta)$ the number of geodesic segments of length $\eta$ in $B$ that are part of a closed geodesic of type $G$.  

Given a set of flow boxes $B_1,\ldots,B_n$ we denote by $\mathcal A = \mathcal A (B_1,\ldots,B_n)$ the set of closed geodesics $\gamma$ on $X$ for which there exists some $B_i$ with $\gamma\cap B_i = \emptyset$ ($\mathcal A$ stands for ``avoid").

\begin{lemma}
    \label{lem:NBL}
    Fix $\delta, s_0, \epsilon, C>0$ and $\eta<s_0/1000$.  There exists $c$ with the following property.  Let $X$ be a $\delta$-expander surface with systole at least $s_0$.  Let  $B_1,\ldots, B_{Cg}$ be $\eta$ flow boxes in $T^1X$, and let $B$ be an $\eta/3$ flow box. %
    Then for $\mathcal A = \mathcal A(B_1,\ldots,B_{Cg})$ and any $L > c \cdot g (\log g)^2$, we have
    \begin{align*}
      N_{\mathcal A}(B,L,\eta/3) \le \epsilon \cdot \mu(B)e^L. 
    \end{align*}
\end{lemma}

\begin{proof}
    In preparation, we will study the measure of the set of tangent vectors in $B$ whose trajectories avoid some $B_i$, without the condition of being tangent to an actual closed geodesic.
  
  Let $\mathcal B^-:= \{B_1^-,\ldots, B_{Cg}^-\}$ be contracted versions of our flow boxes.  Let
  $$t_1=c'\log g,  \ t_2=2 c'\log g, \ldots, \ t_k=kc' \log g, \ t_{k+1}=L.$$
  Here we take $c'$ to be the constant given by \Cref{lem:meas-avoid-box} (using our $\delta,s_0,\epsilon$). And take $k=\lceil d g \log g \rceil$, where $d$ is a constant that will be specified later.  
  We want the gap between all successive pairs of times, including $t_{k+1}-t_k$, to be at least $c'\log g$.
 To ensure this, we need
  $$L> (k+1)c'\log g = (d g \log g+1) c' \log g.$$
  Since the right-hand term is $O(g (\log g)^2)$, the above inequality will hold if we take the $c$ in the current lemma sufficiently large.  

    Now take
      \begin{align*}
        S_{\mathcal B^-}:= \{ v\in B \cap g_{-L}B : \{g_tv\}_{0\leq t \leq L} \cap B_i^- = \emptyset \text{ for some } i=1,\ldots,Cg\},
      \end{align*}
    i.e. the set of those tangent vectors that start and end up in $B$, and that avoid some element of $\mathcal B^-$.    We will bound the measure of $S_{\mathcal B^-}$ from above.   
    
  Now with the $t_i$ defined above, we get
  \begin{align*}
        S_{\mathcal B^-} \subset \{ v\in B \cap g_{-L}B: \{g_{t_1}v,\ldots,g_{t_k}v\} \cap B_i^- = \emptyset \text{ for some } i=1,\ldots,Cg\}. 
  \end{align*}
To bound the measure of the right-hand side above, we use  \Cref{lem:meas-avoid-box} with flow boxes $B,B_1^-,\ldots,B_{Cg}^-$, giving
  \begin{align}
    \mu(S_{\mathcal B^-})  & \le   (1+\epsilon) \mu(B)^2 \sum_{i=1}^{Cg} \left[1-(1-\epsilon)\mu(B_i)\right]^k \\
                           & = (1+\epsilon) \mu(B)^2 (Cg) \left[1-(1-\epsilon)\mu(B_i)\right]^k \\
                           & \le (1+\epsilon) \mu(B)^2 (Cg) \left[1-\frac{\Omega(1)}{g} \right]^k, \label{eq:SB}
  \end{align}
  where the implicit constant in $\Omega()$ is absolute.  
  Now we use the approximation
  $$g \left(1-\frac{\Omega(1)}{g} \right)^k \le  O(1) g \exp\left(-\frac{\Omega(1)}{g}\right)^k \le O(1) g \exp\left(\frac{-k\Omega(1)}{g}\right),$$
  where all the implicit constants are absolute.  %
  Now is when we choose $d$ -- noting that $k \ge d g \log g$, we take $d$ large enough that the right hand side of the above is smaller than $\frac{\epsilon}{(1+\epsilon)C}$.   Applying this to \eqref{eq:SB} we get
  \begin{align}
    \mu(S_{\mathcal B^-}) \le \epsilon \cdot \mu(B)^2. \label{eq:SBbound}
  \end{align}

  \begin{claim}
    \label{cla:comp-frac}
   Let $\kappa$ be the fraction of the components of $B\cap g_{-L}B$ that are completely contained in $S_{\mathcal B^-}$.  Then $\kappa = O(\sqrt{\epsilon}).$
 \end{claim}

   \renewcommand\qedsymbol{$\square$}
  \begin{proof} 
    Let $N= \comp(B\cap g_{-L}B)$, and let $P_1,\ldots, P_N$ be the components of $B\cap g_{-L}B$, ordered such that $\mu(P_1) \le \cdots \le \mu(P_N).$  Then
    \begin{align*}
      \mu(S_{\mathcal B^-}) \ge \sum_{i< \kappa N} \mu(P_i) \ge \sum_{(\kappa/2)N < i < \kappa N} \mu(P_i) \ge \mu(P_{\lfloor (\kappa/2)N\rfloor}) \cdot (\kappa/2)N.
    \end{align*}
Note that $N=\Omega(\mu(B) e^L)$; see \eqref{eq:num-comps}. 
    Using this in the above, together with \eqref{eq:SBbound}, we get 
    \begin{align}
       \mu(P_{\lfloor (\kappa/2)N\rfloor})  &\le \mu(S_{\mathcal B^-})\cdot \frac{2}{\kappa \cdot N} \le \epsilon \cdot \mu(B)^2 \cdot \frac{2}{\kappa \cdot \Omega(\mu(B)e^L)} \\
       &\le  O\left((\epsilon/\kappa) \cdot e^{-L} \mu(B)\right). \label{eq:P-upper}
    \end{align}
On the other hand, as in the proof of \Cref{lem:ave-flow-width}
we see that the components of $B\cap g_{-L}B$ have widths in the flow direction that are close to equidistributed in $[0,\eta/3]$.  In particular the width of $P_{\lfloor (\kappa/2)N\rfloor}$ in the flow direction is $\Omega(\kappa \eta)$.  On the other hand, this set has full width in the contracting direction, and width in the expanding direction $\Omega(e^{-L} \eta)$ (here we are ignoring ``edge effects", but these can be dealt with by using contracted/expanded flow boxes, as in \Cref{lem:effect-pgt-half}).  Combining these estimates gives
\begin{align*}
  \mu(P_{\lfloor (\kappa/2)N\rfloor})  \ge   \Omega \left( \kappa  e^{-L}\mu(B) \right).%
\end{align*}
Combining this with the upper bound \eqref{eq:P-upper}, we then see that $\kappa = O(\sqrt \epsilon)$, as desired.  

\end{proof}
\renewcommand\qedsymbol{$\blacksquare$}

  Observe that if $v \in B\cap g_{-L}B$ and $v\not\in S_{\mathcal B^-}$, then for any $w$ in $v$'s connected component of $B\cap g_{-L}B$, we have $w\not\in S_{\mathcal B}$.  This is because all vectors in the component travel closely together up to time $L$, and since $\{g_tv\}$ hits every $B_i^-$, we see that $\{g_tw\}$ will hit every $B_i$.

Thus any segment counted by $N_{\mathcal A}(B,L,\eta/3)$ intersects a component of $B\cap g_{-L}B$ that is entirely contained in  $S_{\mathcal B^-}$; let $K$ be the number of such components. It follows from \eqref{eq:num-comps} that 
$$\comp(B\cap g_{-L}B) = O(\mu(B) e^L).$$
And then from \Cref{cla:comp-frac} that 
$$K\le O(\sqrt{\epsilon}) \cdot \comp(B\cap g_{-L}B) \le O\left( \sqrt{\epsilon} \cdot \mu(B)e^L \right).$$

Now \Cref{lem:closing-unique} implies that each component of $B\cap g_{-L}B$ intersects at most one $\eta/3$ segment of a closed geodesic with length in $[L-\eta/3,L+\eta/3]$.  So
\begin{align*}
  N_{\mathcal A}(B,L,\eta/3) & \le K \le O\left( \sqrt{\epsilon} \cdot \mu(B)e^L \right).
\end{align*}
Choosing $\epsilon$ appropriately gives the desired result.  

\end{proof}

\begin{lemma}
  \label{lem:non-fill-bd}
  Fix $\delta, s_0, \epsilon>0$.  There exists $c$ with the following property.  Let $X$ be a $\delta$-expander surface with systole at least $s_0$.  Let  $NF$ be the set of closed geodesics on $X$ that are non-filling.  
  Then for any $L > c \cdot g (\log g)^2$, we have
  \begin{align*}
   N_{NF} (X, L)\le \epsilon \cdot e^L/L .
  \end{align*}  
\end{lemma}

\begin{proof}
Let $\eta<s_0/1000$ be given by \Cref{lem:filling-net}, and let $B_1,\ldots,B_{Cg}$ be the $\eta$ flow boxes in $T^1X$ given by that lemma.  We get that $NF$ is a subset of $\mathcal A=\mathcal A(B_1,\ldots,B_{Cg})$, and hence for any $L$,
\begin{align}
    N_{NF}(X,L) \le N_{\mathcal A}(X, L). \label{eq:NF-A}
\end{align}

We then note that \Cref{lem:NBL} gives a $c$ such that for $B$ any $\eta/3$ flow box and the $B_1,\ldots,B_{Cg}$ from above, we have
\[
N_{\mathcal A}(B,L,\eta/3) \le \epsilon \cdot  \mu(B) e^L,
\]
for any $L > c \cdot g (\log g)^2$.  
This allows us to apply \Cref{prop:fb-to-total}, which then gives that 
\begin{align}
    N_{\mathcal A}(X,[L/2,L]) \le 12 \epsilon \cdot  e^L/L, \label{eq:Alarge}
\end{align}
for any $L>2 c\cdot g (\log g)^2$.

We now deal with geodesics of length $\le L/2$.  By \Cref{thm:effect-pgt} (with $\epsilon=1$), for large $c$ and $L>c \log g$, we have
\begin{align*}
     N(X,L/2) \le 2 \cdot \frac{e^{L/2}}{L/2} =\frac{4}{e^{L/2}}\cdot \frac{e^L}{L}. 
\end{align*}
By taking $c$ sufficiently large (depending only on $\delta,s_0,\epsilon$), we can ensure that the term $\frac{4}{e^{L/2}}$ in the above is less than $\epsilon$ when $L>c \log g$, and so we m get 
\begin{align}
     N_\mathcal A (X,L/2) < \epsilon \cdot e^L/L. \label{eq:Asmall}
\end{align}

Summing \eqref{eq:Alarge} and \eqref{eq:Asmall} gives
\begin{align*}
    N_\mathcal A(X,L) \le 13 \epsilon \cdot e^L/L.
\end{align*}

Then combining the above with \eqref{eq:NF-A} yields the desired conclusion.%

\end{proof}

\subsection{Completing proof of \Cref{thm:filling}}
\label{sec:proof-filling}

\begin{proof}[Proof of \Cref{thm:filling}]
We combine \Cref{lem:non-fill-bd}, which bounds from above the number of non-filling geodesics, with the lower bound on number of all closed geodesics $N(X,L)$ given by \Cref{thm:effect-pgt} to conclude the desired result.
\end{proof}

\bibliographystyle{amsalpha}
  \bibliography{sources}

\end{document}